\documentclass[12pt,english,a4paper]{smfart}

\usepackage[utf8]{inputenc}
\usepackage{graphicx}
\usepackage[T1]{fontenc}
\usepackage[english]{babel}
\usepackage{lmodern}
\usepackage{smfthm}
\usepackage[headings]{fullpage}
\usepackage{amssymb}
\usepackage[all]{xy}
\usepackage[colorlinks=true, allcolors=magenta]{hyperref}
\usepackage{url}
\usepackage{tikz}
\usepackage{comment}
\usepackage{breqn}
\usepackage{tikz-cd}
\usepackage{lipsum}
\usepackage{adjustbox}
\usepackage{csquotes}
\usepackage[normalem]{ulem}
\usepackage{dirtytalk}

\theoremstyle{plain}
\newtheorem{thm}{Theorem}
\newtheorem{cor}[thm]{Corollary}
\newtheorem{pro}[thm]{Proposition}
\theoremstyle{definition}
\newtheorem{rmk}[thm]{Remark}

\newcommand{\red}[1]{\textcolor{red}{#1}}

\let\rm\mathrm
\let\cal\mathcal
\let\goth\mathfrak

\let\hat\widehat
\let\tilde\widetilde
\let\phi\varphi

\let\epsilon\varepsilon

\def\Q{{\bf Q}} 
\def\Z{{\bf Z}}
\def\C{{\bf C}}
\def\N{{\bf N}}
\def\R{{\bf R}}
\def\A{{\bf A}}
\def\E{{\bf E}}
\def\B{{\bf B}}
\def\G{{\cal G}}
\def\F{{\bf F}}
\def\D{{\mathbf{D}}}

\def\Kunder{{\underline{K}}}
\def\Lunder{{\underline{L}}}

\def\End{{\mathrm{End}}}
\def\GL{{\mathrm{GL}}}
\def\k{{\bf k}}
\def\j{{\bf j}}
\def\i{{\bf i}}
\def\D{{\bf D}}
\def\M{{\bf M}}
\def\c{{\bf c}}
\def\n{{\underline{n}}}
\def\r{{\underline{r}}}
\def\s{{\underline{s}}}
\def\Id{{\mathrm{Id}}}
\def\id{{\mathrm{id}}}
\def\Aut{{\mathrm{Aut}}}

\def\Mat{{\mathrm{Mat}}}
\def\Fil{{\mathrm{Fil}}}

\def\pa{{\mathrm{pa}}}

\def\an{{\mathrm{an}}}
\def\Sen{{\mathrm{Sen}}}
\def\dR{{\mathrm{dR}}}
\def\crys{{\mathrm{crys}}}
\def\ra{{\rightarrow}}

\def\dif{{\mathrm{Dif}}}
\def\vall{{\mathrm{val}_{\Lambda}}}
\def\fin{{\mathrm{fin}}}
\def\gr{{\mathrm{gr}}}
\def\HT{{\mathrm{HT}}}

\renewcommand{\O}{{{\cal O}}}

\def\Zp{{\Z_p}}
\def\Qp{{\Q_p}}
\def\Cp{{\C_p}}
\def\Fp{{\F_p}}

\def\At{{\tilde{\bf{A}}}}
\def\Atplus{{\tilde{\bf{A}}^+}}
\def\Bt{{\tilde{\bf{B}}}}
\def\Btplus{{\tilde{\bf{B}}^+}}

\def\Bplusmax{{\bf{B}^{+}_{\mathrm{max}}}}
\def\Bdr{{\bf{B}_{\mathrm{dR}}}}
\def\Bdrplus{{\bf{B}_{\mathrm{dR}}^+}}

\def\Btrig{{\Bt_{\mathrm{rig}}^{\dagger}}}
\def\Btrigplus{{\Bt_{\mathrm{rig}}^+}}
\def\Btlogplus{{\Bt_{\log}^+}}
\def\Btrigr{{\Bt_{\mathrm{rig}}^{\dagger,r}}}

\def\Gal{{\mathrm{Gal}}}

\def\GL{{\mathrm{GL}}}

\def\k{{\mathbf{k}}}
\def\la{{\mathrm{la}}}
\def\cbf{{\bf{c}}}
\def\kbf{{\bf{k}}}

\def\cycl{{\mathrm{cycl}}}
\def\crys{{\mathrm{crys}}}
\def\st{{\mathrm{st}}}
\def\rig{{\mathrm{rig}}}
\def\M{{\mathrm{M}}}
\def\Btrig{{\Bt_{\rig}}}

\def\Kbar{{\overline{K}}}

\def\EtplusDel{{\tilde{\bf{E}}^+_{\Delta}}}
\def\AtplusDel{{\tilde{\bf{A}}^+_{\Delta}}}

\def\BdrDel{{{\bf{B}}_{\mathrm{dR},\Delta}}}
\def\BdrplusDel{{{\bf{B}}_{\mathrm{dR},\Delta}^+}}

\def\BtrigDel{{\Bt_{\mathrm{rig},\Delta}^{\dagger}}}
\def\BtrigrDel{{\Bt_{\mathrm{rig},\Delta}^{\dagger, r_{m_1},\dots,r_{m_{\delta}}}}}
\def\BtrigplusDel{{\Bt_{\mathrm{rig},\Delta}^+}}
\def\BtlogplusDel{{\Bt_{\log,\Delta}^+}}
\def\Btrigr{{\Bt_{\mathrm{rig},\Delta}^{\dagger, r}}}
\def\AplusmaxDel{{\bf{A}^{+}_{\mathrm{max},\Delta}}}

\def\BtlogDel{{\Bt_{\mathrm{log},\Delta}^{\dagger}}}
\def\Btlog{{\Bt_{\mathrm{log}}^{\dagger}}}

\def\DcrysDel{{\mathbf{D}_{\mathrm{crys},\Delta}}}
\def\DrigDel{{\mathbf{D}^{\dagger}_{\mathrm{rig},\Delta}}}
\def\DstDel{{\mathbf{D}_{\mathrm{st},\Delta}}}

\def\wotimes{{\hat{\otimes}}}

\newcommand{\Btint}[2]{\tilde{\bf{B}}_{\left[#1;#2\right]}}

\newcommand{\AtintDel}[2]{\tilde{\bf{A}}_{\left[#1;#2\right],\Delta}}

\newcommand{\AtdaggerradDel}[1]{\tilde{\bf{A}}^{\dagger,#1}_{\Delta}}
\newcommand{\teich}[1]{\left[#1\right]}
\newcommand{\overbar}[1]{\mkern 1.5mu\overline{\mkern-1.5mu#1\mkern-1.5mu}\mkern 1.5mu}

\newcommand{\norm}[1]{\left\|#1\right\|}
\newcommand{\tate}[2]{#1\langle#2\rangle}

\def\varpibar{{\overline{\varpi}}}
\def\u{{\overline{u}}}
\def\Btrig{{\Bt_{\mathrm{rig}}^{\dagger}}}

\def\wsum{{\hat{\bigoplus}}}
\def\Lambt{{\tilde{\Lambda}}}

\def\val{{\rm{val}}}

\title{Multivariable $p$-adic Hodge theory for products of Galois groups}
\author{Léo Poyeton, Pietro Vanni}
\date{}

\begin{document}

\begin{abstract}
    In this paper we explain how to attach to a family of $p$-adic representations of a product of Galois groups an overconvergent family of multivariable $(\phi,\Gamma)$-modules, generalizing results from Pal-Zabradi and Carter-Kedlaya-Zabradi, using Colmez-Sen-Tate descent. We also define rings of multivariable crystalline and semistable periods, and explain how to recover this multivariable $p$-adic theory attached to a family of representations from its multivariable $(\phi,\Gamma)$-module. We also explain how our framework allows us to recover the main results of Brinon-Chiarellotto-Mazzari on multivariable $p$-adic Galois representations.
\end{abstract}

\maketitle
\tableofcontents


\section*{Introduction}
Let $p$ be a prime number and let $K$ be a finite extension of $\Qp$. 
Let $\G_K$ be the absolute Galois group of $K$.
The study of $p$-adic representations of $\G_K$, that is finite dimensional $\Qp$-vector spaces, endowed with a continuous action of $\G_K$, and more generally of families of $p$-adic representations of $\G_K$, classicaly relies on the rings of periods of $p$-adic Hodge theory \cite{fontaine1994corps} and on $(\phi,\Gamma)$-modules \cite{Fon90}. 

Recently, interest has risen around the study of $p$-adic representations of products of Galois groups of $p$-adic fields, initiated by Zabradi \cite{ZabMulti0} and followed up by Zabradi and other authors \cite{PalZab21,KedCarZab,BriChiaMaz21}. Products of Galois groups naturally appear in the approach to geometric Langlands developed for $\GL_2$ by Drinfeld \cite{drinfeld1980langlands} and extended to $\GL_n$ by L. Lafforgue \cite{lafforgue1997chtoucas} and to other reductive groups by V. Lafforgue \cite{lafforgue2014introduction,lafforgue2018chtoucas}, and the development of a \textit{multivariable $p$-adic Hodge theory} appears to be a necessary step in order to understand the $p$-adic representations of those products of Galois groups. 

In the work of Scholze and Weistein \cite{weinstein2017gal,scholze2020berkeley}, a product of Galois groups $\G_{K_1} \times \cdots \times \G_{K_\delta}$ of $\delta$ $p$-adic fields, with $\delta\in\N$, can be understood as the fundamental group of a diamond $\mathrm{Spd} K_1 \times \cdots \times \mathrm{Spd} K_\delta$. It then makes sense to study $p$-adic representations for this fundamental group, viewed as coefficients for the diamond $\mathrm{Spd} K_1 \times \cdots \times \mathrm{Spd} K_\delta$. 

In this article, we develop an analogue of the framework of classical $p$-adic Hodge theory and $(\phi,\Gamma)$-modules in the multivariable setting, and prove various results which are analogue to the classical ones, following and expanding upon the work of Carter, Kedlaya, Pal and Zabradi \cite{ZabMulti0,PalZab21,KedCarZab} and of Brinon, Chiarellotto and Mazzari \cite{BriChiaMaz21}. 

In what follows, $\Delta$ is a finite set of cardinality $\delta$, and for each $\alpha \in \Delta$, $K_\alpha$ is a finite extension of $\Qp$. We also let $\G_{\Kunder,\Delta}:= \prod_{\alpha \in \Delta}\G_{K_\alpha}$. 

In the classical setting, that is when $\delta = 1$, Fontaine has constructed \cite{Fon90} an equivalence of categories $V \mapsto D(V)$ between the category of $p$-adic representations of $\G_K$ and the category of étale $(\phi,\Gamma_K)$-modules, which are finite dimensional vector spaces over a local field $\B_K$ of dimension $2$, endowed with commuting semilinear actions of a Frobenius $\phi$ and of the Galois group $\Gamma_K$ of the cyclotomic extension $K_\infty$ of $K$. The étaleness condition only depends on the $\phi$-action. 

In order to study $p$-adic representations of $\G_K$, Fontaine has moreover introduced (see \cite{fontaine1994corps}) rings of $p$-adic periods $\B_{\crys}, \B_{\st}$ and $\Bdr$, which are topological $\Qp$-algebras endowed with an action of $\G_K$ and with additional structures, such that if $B$ is any of those rings and if $V$ is a $p$-adic representation of $\G_K$, then $\D_B(V):=(B \otimes_{\Qp}V)^{\G_K}$ is a $B^{\G_K}$-module that inherits those additional structures and provides interesting invariants attached to $V$. A $p$-adic representation $V$ is said to be admissible if $V \otimes_{\Qp}B \simeq B^d$ as $\G_K$-modules. The case where $B = \Cp$ had been previously studied by Tate \cite{Tat67} and Sen \cite{sen1980continuous}. 

In \cite{ZabMulti0}, Zabradi has constructed the analogue of Fontaine's $(\phi,\Gamma)$-modules for the multivariable case. As in the classical setting, $p$-adic representations of $\G_K$ are classified by étale $(\phi_{\Delta},\Gamma_{\Kunder,\Delta})$-modules over some ring $\cal{E}_\Delta$, where the ring $\cal{E}_\Delta$ is equipped with $\delta$ distinct Frobenii and $\Gamma_{\Kunder,\Delta}$ is the product of the Galois groups $\Gamma_{K_\alpha}$, and $(\phi_{\Delta},\Gamma_{\Kunder,\Delta})$-modules are defined as projective modules over $\cal{E}_\Delta$, endowed with commuting semilinear actions of $\Gamma_{\Kunder,\Delta}$ and of each partial Frobenius. In contrast with what happens in the classical setting, the étaleness condition also involves the Galois action in the multivariable setting. 

In \cite{BriChiaMaz21}, Brinon, Chiarellotto and Mazzari have defined $\B_{\rm{dR},\Delta}$, an analogue of the ring $\Cp$ and of the ring of periods $\Bdr$, which are the first step to constructing analogues of all the rings of periods of classical $p$-adic Hodge theory. 

One key result in the classical setting is the \textit{overconvergence} of $p$-adic representations, proven in \cite{cherbonnier1998representations} by Cherbonnier and Colmez, which means that the $(\phi,\Gamma)$-modules of Fontaine over $\B_K$ are the extensions of scalars to $\B_K$ of $(\phi,\Gamma)$-modules defined over a smaller ring $\B_K^\dagger$ which has an analytic interpretation as bounded Laurent series defined over an annulus whose outer boundary is the unit circle. This result of overconvergence is crucial to recover the $p$-adic Hodge theory attached to a representation $V$ from its $(\phi,\Gamma)$-module \cite{Ber02}, using the Robba ring $\B_{\rig,K}^\dagger$ as a bridge. 

The theory can be extended to $p$-adic families: let $S$ be a $\Qp$-Banach algebra with maximal spectrum $\cal{X}$ and such that for $x \in \cal{X}$, $S/\mathfrak{m}_x$ is a finite extension of $\Qp$. A family of representations of $\G_K$ is a free $S$-module of finite rank endowed with a continuous $S$-linear action of $\G_K$.

The theorem of overconvergence of Cherbonnier and Colmez has been extended to families \cite{BC08} thanks to the use of so-called ``Colmez-Tate-Sen'' conditions, and one can once again relate the $p$-adic Hodge theory attached to a family of representations with its overconvergent $(\phi,\Gamma)$-module \cite{Bellovinsheaves}. The functor $V \mapsto D^\dagger(V)$ defined by Berger and Colmez is however no longer an equivalence of categories.  

The constructions of \cite{ZabMulti0, PalZab21, KedCarZab, BriChiaMaz21} suggest that the right point of view in order to define analogues of the classical rings appearing in $p$-adic Hodge theory for products of Galois groups is to take the completed tensor product (which sometimes means that we have to make a choice on what topology we put on our rings) of $\delta$ copies of the classical rings of $p$-adic Hodge theory, endowed with natural actions of $\G_{\Kunder,\Delta}$ and of partial Frobenii $\phi_\alpha$ (when they exist on the classical rings) each corresponding to the action of the Frobenius on one copy of those rings. 

Our first step is to construct various rings of periods relying on this formalism of taking completed tensor products of $\delta$ copies of classical rings of periods, some of which were already defined this way in \cite{ZabMulti0, PalZab21, KedCarZab, BriChiaMaz21}, and prove that some of them are actually isomorphic to some rings defined in a different fashion in one of \cite{ZabMulti0, PalZab21, KedCarZab, BriChiaMaz21}. In particular, we define a ring $\B_{\Kunder,\Delta}$ which is isomorphic to the ring $\cal{E}_\Delta$ of \cite{ZabMulti0,KedCarZab}, and an overconvergent subring $\B_{\Kunder,\Delta}^\dagger$ isomorphic to the ring $\cal{E}_\Delta^\dagger$ appearing in \cite{PalZab21,KedCarZab}. This allows us to define the same multivariable $(\phi,\Gamma)$-modules over either of the rings $\At_{\Kunder,\Delta}^\dagger$, $\At_{\Kunder,\Delta}$ $\A_{\Kunder,\Delta}^\dagger$ or $\A_{\Kunder,\Delta}$, which are defined in \S 3, and to apply the results from those papers to ours rings. Just as in the classical case, the ring $\B_{\Kunder,\Delta}$ does not have a nice analytic interpretation but the ring $\B_{\Kunder,\Delta}^\dagger$ does: it corresponds to bounded Laurent series over a $\delta$-dimensional polyannulus whose outer boundary is the unit polycircle. 

One consequence of this point of view when defining an analogue of some of the rings appearing in the classical setting is that the Colmez-Tate-Sen conditions can be applied ``factor by factor'' and we are thus able to recover the multivariable Sen theory of \cite{BriChiaMaz21} over the ring $C_\Delta:=\Cp \wotimes_{\Qp} \cdots \wotimes_{\Qp} \Cp$, and generalize it to families admitting a Galois stable integral lattice. In what follows, $\n$ denotes a $\delta$-uple of positive integers $(n_\alpha)_{\alpha \in \Delta}$ and if $\Lunder = (L_\alpha)_{\alpha \in \Delta}$ then $L_{\n,\Delta} = L_{\alpha_1}(\mu_{p^{n_{\alpha_1}}})\otimes_{\Qp} \cdots \otimes_{\Qp} L_{\alpha_\delta}(\mu_{p^{n_{\alpha_\delta}}})$.

\begin{theo}
    Let $S$ be a $\Qp$-Banach algebra, let $T$ be an $\O_S$-representation of dimension $d$ of $\G_{\underline{K},\Delta}$, and let $V = S \otimes_{\O_S}T$. Let $\underline{L} = (L_1,\ldots,L_\delta)$ be such that for all $i \in \Delta$, $L_i/K_i$ is Galois extension and such that $\G_{\underline{L},\Delta}$ acts trivially on $T/12p^{\delta}$. Then there exists $\n(\underline{L})$ depending only on $\Lunder$ such that for all $\n \geq \n(\underline{L})$, $(S\hat{\otimes}\C_{\Delta})\otimes_S V$ contains a unique sub-$(S \otimes L_{\n,\Delta})$-module $\D_\Sen^{L_{\Delta,\n}}(V)$, free of rank $d$, such that:
\begin{enumerate}
\item $\D_\Sen^{L_{\Delta,\n}}(V)$ is fixed by $H_{\underline{L},\Delta}$ and stable by $\G_{\Kunder,\Delta}$;
\item $\D_\Sen^{L_{\Delta,\n}}(V)$ contains a basis over $S \otimes L_{\n,\Delta}$ which is $c$-fixed by $\Gamma_{\underline{L},\Delta}$, for some $c\in\R_{>0}$ (in the sense of \cite[Proposition 3.3.1]{BC08});
\item the natural map $(S\hat{\otimes}\C_{\Delta}) \otimes_{S \otimes L_{\n,\Delta}}\D_\Sen^{L_{\Delta,\n}}(V) \ra (S\hat{\otimes}\C_{\Delta})\otimes_S V$ is an isomorphism.
\end{enumerate} 
Moreover, we have $S/\mathfrak{m}_x\otimes_S \D_\Sen^{L_{\Delta,\n}}(V) \simeq \D_\Sen^{L_{\Delta,\n}}(V_x)$.
\end{theo}

Additionally we explain how to use the notion locally analytic vectors to recover this multivariable Sen theory, both for $C_\Delta$ and for $\B^+_{\rm{dR},\Delta}$, following the ideas of \cite{Ber14SenLa}. 

The same approach of the Colmez-Tate-Sen axioms for completed tensor products allows us to derive an overconvergence result for multivariable $(\phi,\Gamma)$-modules, which recovers the main results from \cite{PalZab21} and \cite{KedCarZab}, and which once again also applies to families of representations admitting a Galois stable integral lattice:

\begin{theo}
    If $V$ is an $S$-family of representations of $\G_{\Kunder,\Delta}$, free of dimension $d$, admitting a Galois-stable integral lattice, and if $\s \geq \s(V)$ (a set of integers depending only on $V$), then:
    \begin{enumerate}
        \item $\D_{\Kunder,\Delta}^{\dagger,\s}(V)$ is a projective $S \hat{\otimes}_{\Qp}\B_{\Kunder,\Delta}^{\dagger,\s}$-module of rank $d$;
        \item the map $(S \hat{\otimes}_{\Qp}\Bt_{\Kunder,\Delta}^{\dagger,\s})\otimes_{S \hat{\otimes}_{\Qp}\B_{\Kunder,\Delta}^{\dagger,\s}}\D_{\Kunder,\Delta}^{\dagger,\s}(V) \ra (S \hat{\otimes}_{\Qp}\Bt_{\Kunder,\Delta}^{\dagger,\s})\otimes_SV$ is an isomorphism;
        \item if $x \in \cal{X}$, the map $S/\mathfrak{m}_x \otimes_S\D_{\Kunder,\Delta}^{\dagger,\s}(V) \ra \D_{\Kunder,\Delta}^{\dagger,\s}(V_x)$ is an isomorphism.
    \end{enumerate}
\end{theo}

We explained above that unfortunately and in contrast with the classical setting, the notion of étaleness involves the $\Gamma_{\Kunder,\Delta}$-action. Theorem 6.19 of \cite{KedCarZab} shows that this condition can be relaxed by asking that the $\Gamma_{\Kunder,\Delta}$-action is ``bounded'' but the authors are only able to prove it for $(\phi_{\Delta},\Gamma_{\Kunder,\Delta})$-modules over $\A_{\Kunder,\Delta}^\dagger$ or $\A_{\Kunder,\Delta}$. The Colmez-Sen-Tate conditions allow us to descend multivariable $(\phi,\Gamma)$-modules over $\At_{\Kunder,\Delta}^\dagger$ to  multivariable $(\phi,\Gamma)$-modules over $\A_{\Kunder,\Delta}^\dagger$ while preserving this boundness condition, and this fact allows us to show that their result is also true for $(\phi,\Gamma)$-modules with coefficients in these larger rings. 

We then prove that the ring $\B^+_{\rm{dR},\Delta}$ defined in \cite{BriChiaMaz21} can also be defined within our framework (so that their ring is isomorphic to the completed tensor product of $\delta$ copies of the classical $\Bdrplus$). This allows us to prove that a representation is ``multivariable de Rham'' if and only if its restriction at each factor $\G_{K_\alpha}$ is de Rham (this result was also proved in \cite{KedlayadeRhamprop} using different methods), and thus to give a positive answer to the following question of Brinon Chiarellotto and Mazzari: if $V$ is a multivariable representation such that its multivariable de Rham module $\D_{\dR,\Delta}:=(\BdrDel\otimes_{\Qp}V)^{\G_{\Kunder}}$ if free of rank $\dim_{\Qp}V$ over $(\BdrDel)^{\G_{\Kunder,\Delta}}$, is $V$ admissible for $\BdrDel$ ?

Finally, we give an analogue of the results of \cite{Ber02} and \cite{Bellovinsheaves} for the classical setting, defining multivariable crystalline and semistable representations, and explaining how to recover the modules $\D_{\crys,\Delta}(V)$ and $\D_{\st,\Delta}(V)$ attached to a (family of) $p$-adic representation(s) $V$ from the multivariable $(\phi,\Gamma)$-module $\D_{\rig,\Delta}^{\dagger}(V)$ attached to $V$ and defined over the multivariable Robba ring $\B_{\rig,\Kunder,\Delta}^{\dagger}$:

\begin{theo}
     Let $V$ be a representation of $\G_{\underline{K}}$ over $S$ admitting an invariant $\cal{O}_S$-lattice. Then 
    \begin{equation*}\label{comp crys phigamma}
    \mathbf{D}_{\mathrm{crys},\Delta}(V)\cong(\mathbf{D}^{\dagger}_{\mathrm{rig},\Delta}(V)[1/t_{\Delta}])^{\Gamma_{\underline{K},\Delta}},
\end{equation*}
and 
\begin{equation*}\label{comp crys phigamma}
    \mathbf{D}_{\mathrm{st},\Delta}(V)\cong(\mathbf{D}^{\dagger}_{\mathrm{log},\Delta}(V)[1/t_{\Delta}])^{\Gamma_{\underline{K},\Delta}}.
\end{equation*}
\end{theo}

\subsection*{Acknowledgments}
The authors would like to thank Olivier Brinon and Bruno Chiarellotto for their interest and for several fruitful discussions on the content of the article.

\section{Non archimedean functional analysis}
This section is devoted to some lemmas and propositions related to completed (projective) tensor products.
As we will interchange the viewpoint of normed modules (as in \cite{BoschAna}), topological vector spaces (as in \cite{Schnei}), and adic rings (as in \cite{huber1993continuous}), we firstly make sure that in our case the various notions of completed tensor product coincide.
\begin{rema}
    Let $V$ and $W$ be Banach spaces over a nonarchimedean field $\mathcal{K}$. Then the completed tensor product of \cite{Schnei} is the same as the completed tensor product of \cite{BoschAna}, as the projective tensor product can be defined using the tensor seminorm, see \cite[p. 103]{Schnei}.
\end{rema}
Let now $n\in\N$, and let $R_1,\dots, R_n$ be topological rings over $\Zp$, where the topology in $R_i$ is defined by a principal ideal $(r_i)$, for $r_i\in R_i$ and $i=1,\dots, n$. Suppose that $R_i$ is separated for the $r_i$-adic topology, for $i=1,\dots, n$. We can define a norm on $R_i$ setting, for $a_i\in R_i$, $r=1,2$, 
\begin{equation}\label{adic norm}\norm{a_i}_{R_i}=2^{-m},\end{equation}
where $m$ is the biggest natural number such that $a_{i}\in (r_i)^{m}$.
It is clear that the $(r_i)$-adic topology and the topology defined by the norm \eqref{adic norm} on $R_i$ coincide. We assume that $p\in (r_i)$, for $i=1,\dots, n$. This implies that $R_i$ is a normed module over $\Zp$, as 
$$\norm{pa}_{R_i}\le |p|_{\Zp}\norm{a}_{R_i},$$ for $a\in R_i$ (if we normalize the norm on $\Zp$ setting $|p|_{\Zp}=\frac{1}{2}$).
\begin{prop}
    The topology induced by the tensor seminorm on $R_1\otimes_{\Zp}\dots\otimes_{\Zp}R_n$ coincides with the $(r_1,\dots,r_n)$-adic topology.
\end{prop}
\begin{proof}
    Suppose $\norm{a}_{R_1\otimes_{\Zp}\cdots\otimes_{\Zp}R_n}\le c$, for $c$ a nonzero rational number. Thus since the tensor seminorm is discrete $a$ can be written as a finite sum $\sum_{j\in J}a^1_j\otimes_{\Zp}\dots\otimes_{\Zp}a^n_j$, for $a_j^1\in R_1$, \dots, $a_j^n\in R_n$, such that $\bar{j}\in J$ that satisfies $$\norm{a_{j}^1}_{R_1}\dots\norm{a_{j}^n}_{R_n}\le c,$$
    for every $j\in J$.
    But this clearly implies $a\in (r_1,\dots, r_n)^{\lfloor{\rm{max}\{0,-\rm{log}_{2}(c)}\rfloor\}}$. Conversely suppose that $a\in (r_1,\dots,r_n)^m$, for $m\in\N$. Then $a$ can be written as $a=i^mb$, for $i\in (r_1,\dots,r_n)^{m}$ and thus it is clear that $\norm{a}_{R_1\otimes_{\Zp}\dots\otimes_{\Zp}R_n}\le 2^{-m}$.
\end{proof}
In particular the Proposition above implies that the completion of $R_1\otimes_{\Zp}\cdots\otimes_{\Zp} R_n$ for the tensor seminorm coincides with its $(r_1,\dots,r_n)$-adic completion.\\
\begin{prop}
\label{prop tens Fréchet inj}
     Let $K$ be a finite extension of $\Qp$. If $U_0 \ra U_1$ and $V_0 \ra V_1$ are injective continuous maps of $K$-Fréchet spaces, then the induced map $U_0 \hat{\otimes}_{K} V_0 \ra U_1 \hat{\otimes}_{K} V_1$ is injective.
\end{prop}
\begin{proof}
    This follows from \cite[Proposition 1.1.26]{Emer} (note that Fréchet spaces are Hausdorff and complete locally convex topological vector spaces by definition, and they are bornological by \cite[Proposition 8.2]{Schnei}).
\end{proof}
In this article we will need to take group invariants of iterated completed tensor products. To do so we mainly rely on the following lemmas. 
\begin{lemm}\label{inv and tensors}
    Let $K$ be a finite extension of $\Qp$.
    Let $\mathcal{A}=\mathrm{lim}_{i\in I}A_i$ and $\mathcal{A}'=\mathrm{lim}_{j\in J}A'_j$ be Fréchet spaces, where $A_i$ and $A'_j$ are Banach spaces over $K$ for each $i\in\N$ and for each $j\in\N$. 
    Suppose that there is a topological group $\goth{G}$ acting on $\mathcal{A}$, that stabilizes each $A_i$, and for which the diagram defining $\mathrm{lim}_{i\in I}A_i$ is $\goth{G}$-equivariant. Assume that the action is continous and $K$-linear. Then \begin{equation*}(\mathcal{A}\wotimes_{K}\mathcal{A}')^{\goth{G}}=\mathcal{A}^{\goth{G}}\wotimes_{K}\mathcal{A}'.\end{equation*}  
\end{lemm}
\begin{proof}
    By \cite[Proposition 1.1.29]{Emer} we have \begin{equation*}\mathcal{A}\wotimes_{K}\mathcal{A}'\cong\mathrm{lim}_{i\in I}\mathrm{lim}_{j\in J}A_i\wotimes_{K}A'_j.\end{equation*}
    Choosing a basis $\cal{B}_j$ of $A_j$ for each $j\in J$ we obtain an isomorphism
    \begin{equation*}
         A'_j\cong_{\cal{B}_j}\hat{\bigoplus}_{i'\in I_j}K
    \end{equation*}
    for each $j\in J$ (this can be done by \cite[Proposition 10.1]{Schnei}).
    These isomorphisms induce a $\goth{G}$-equivariant isomorphism
    \begin{equation*}
        \mathcal{A}\wotimes_{K}\mathcal{A}'\cong_{\cal{B}}\mathrm{lim}_{i\in I}\mathrm{lim}_{j\in J}\hat{\bigoplus}A_i
    \end{equation*}
    using \cite{BoschAna}, Proposition 2.1.7.8..
    Thus we obtain
    \begin{equation*}
        (\mathcal{A}\wotimes_{K}\mathcal{A}')^{\goth{G}}\cong_{\cal{B}}(\mathrm{lim}_{i\in I}\mathrm{lim}_{j\in J}\hat{\bigoplus}A_i)^{\goth{G}}=\mathrm{lim}_{i\in I}\mathrm{lim}_{j\in J}\hat{\bigoplus}A_i^{\goth{G}}.
    \end{equation*}
    But this implies the claim.
\end{proof}
The lemma above implies the following corollary
\begin{coro}\label{cor invariant of frechet factor by factor}
    With the hypotheses of Lemma \ref{inv and tensors}, 
    let $n\in\N$, and let $\cal{A}_i=\rm{lim}_{j_i\in J_i}A_{j_i}$, with $A_{j_i}$ Banach spaces over $K$, for $j_i\in J_i$ and $i=1,\dots,n$. Let ${\goth{G}}_i$ be a topological group acting $K$-linearly and continously on $\cal{A}_i$, for $i=1,\dots, n$. Suppose moreover that for $i=1,\dots,n$ the diagram defining $\cal{A}_i=\rm{lim}_{j_i\in J_i}A_{j_i}$ is $\goth{G}_i$-equivariant.
    Then $$(\cal{A}_1\wotimes_K\cdots\wotimes_K\cal{A}_{n})^{\goth{G}_1\times\cdots\times \goth{G}_n}\cong \cal{A}^{\goth{G}_1}_1\wotimes_K\dots\wotimes_K\cal{A}^{\goth{G}_n}_{n}.$$  
\end{coro}
\begin{proof}
    The claim follows using Lemma \ref{inv and tensors} inductively, as 
    $$\cal{A}_i^{\goth{G}_i}=\rm{lim}_{j_i\in J_i}A_{j_i}^{\goth{G}_i},$$ and $A_{j_i}^{\goth{G}_i}$ is a closed subspace of $A_{j_i}$ (so in particular it is a Banach space over $K$), which implies that $\cal{A}_i^{\goth{G}_i}$ is Fréchet over $K$.
\end{proof}

\begin{rema}\label{invariants LF spaces}
    The claim of Corollary \ref{inv and tensors} holds also if $\cal{A}_i$ is a strict equivariant union of Fréchet spaces over $K$ (satisfying the hypotheses of Lemma \ref{inv and tensors}), see \cite[Subsection 1.12.2]{PieThesis}). In what follows we assume that the LF spaces we consider are of this type, unless stated otherwise. 
\end{rema}
Let $K$ be a finite extension of $\Qp$,
and let $M$ be a normed module over $K$. We define \emph{the unit ball} of $M$ as the submodule of elements of norm less or equal to $1$. We denote it by $\O_{M}$. 
\begin{lemm}\label{unit ball}
    Let $A$ be a Banach space over $K$ and let $\O_{A}$ be its unit ball. Let $M$ be a normed module over $K$, with unit ball $\O_{M}$. The the unit ball of $A\wotimes_{K}M$ is $\O_{A}\wotimes_{\O_K}\O_{M}$. 
\end{lemm}
\begin{proof}
    We know that choosing a basis $\cal{B}$ of $A$, we have an isometric isomorphism $A\simeq_{\cal{B}}\hat{\bigoplus}_{i\in I}K$. We remark that the unit ball of $\hat{\bigoplus}_{i\in I}K$ is $\hat{\bigoplus}_{i\in I}\O_K$.\\
    We have $A\wotimes_{K}M\cong_{\cal{B}}\hat{\bigoplus}_{i\in I}M$, whose unit ball is $\hat{\bigoplus}_{i\in I}\O_{M}$. But \begin{equation*}\hat{\bigoplus}_{i\in I}\O_{M}\cong_{\cal{B}}\hat{\bigoplus}_{i\in I}\O_K\wotimes_{\O_K}\O_{M},\end{equation*} by \cite[Proposition 2.1.7.8]{BoschAna}, and this gives us the claim.
\end{proof}
\begin{lemm}
    Let $A$ be a $K$-Banach space and let $M$ be a normed module over $K$, endowed with an isometric $K$-linear action of a topological group $\goth{G}$. Then
    \begin{equation*}
        (\O_{A}\wotimes_{\O_K}\O_{M})^{\goth{G}}\cong\O_{A}\wotimes_{\O_K}\O_{M}^{\goth{G}}
    \end{equation*}
\end{lemm}
\begin{proof}
    Since the action of $\goth{G}$ is isometric then it preserves the unit ball, so the claim makes sense. Using the same strategy as in the proof of Lemma \ref{inv and tensors}, it is easy to see that 
    \begin{equation*}
        (A\wotimes_{K}M)^{\goth{G}}\cong A\wotimes_{K}M^{\goth{G}},
    \end{equation*}
    but then taking the unit balls of these spaces we get the claim.
\end{proof}
Arguing similarly as in the proof of Corollary \ref{cor invariant of frechet factor by factor} we obtain the following
\begin{coro}
    Let $n\in\N$, and let $A_i$ be a Banach space over $K$, for $i=1,\dots, n$. And let $\goth{G}_i$ be a topological group acting isometrically, $K$-linearly and continously on $A_i$, for $i=1,\dots, n$. Then
    $$ (\O_{A_1}\wotimes_{K}\cdots\wotimes_{K} \O_{A_n})^{\goth{G}_1\times\cdots\times \goth{G}_n}\cong \O_{A_1}^{\goth{G}_1}\wotimes_{K}\dots\wotimes_{K} \O_{A_n}^{\goth{G}_n}.$$
\end{coro}

Let $S$ be a $\Qp$-Banach algebra. A family of $p$-adic representations of $\G_{\Kunder,\Delta}$ is a free $S$-module of finite type, endowed with a continuous linear action of $\G_{\Kunder,\Delta}$. We will assume that our families admit a Galois-stable integral subring, i.e. that there exists a free $\O_S$-module $T$ such that $V=S\otimes_{\Zp}T$. As in the one-variable case, if $S=E$ is a field, this condition is always satisfied:

\begin{lemm}
\label{lemma stable lattice field}
    If $V$ is an $E$-representation of dimension $d$, then there exists a free $\O_E$-module $T$ of dimension $d$, stable by $\G_{\Kunder,\Delta}$ such that $V=E\otimes_{\O_E}T=T[1/p]$.
\end{lemm}
\begin{proof}
    This is the same proof as the one of \cite[Lemm. 2.3.1]{BC08}, replacing $\G_K$ by $\G_{\Kunder,\Delta}$.
\end{proof}

We also recall the following, which is a slight variant of an étale descent result and proposition 2.2.1 of \cite{BC08}:

Let now $B$ be a $\Qp$-Banach algebra, endowed with a continuous action of a finite group $G$. Let $\B^{\natural}$ denote the ring $B$ endowed with trivial $G$-action. We assume that:
\begin{enumerate}
    \item the $B^G$-module $B$ is free of finite rank and faithfully flat;
    \item we have $B \otimes_{B^G}B^{\natural} \simeq \oplus_{g \in G}B^{\natural}\cdot e_g$ (where $e_ge_h=\delta_{gh}e_g$ and $g(e_h)=e_{gh}$).
\end{enumerate}
In this case, the following result holds.
\begin{prop}
\label{prop taking invariants gives proj and iso}
    If $S$ is a $\Qp$-Banach algebra, endowed with a trivial $G$-action, and if $M$ is a finitely generated free $S \hat{\otimes}_{\Qp}B$-module, endowed with a semilinear action of $G$, then:
    \begin{enumerate}
        \item $M^G$ is a finitely generated projective $S \hat{\otimes}_{\Qp}B^G$-module;
        \item the map $(S \hat{\otimes}_{\Qp}B) \otimes_{S \hat{\otimes}_{\Qp}B^G}M^G \ra M$ is an isomorphism.
    \end{enumerate}
\end{prop}
\begin{proof}
    Let $\pi_G = \frac{1}{|G|}\sum_{g \in G}g \in B[G]$. If $N$ is a $B[G]$-module, we get a decomposition $N = \pi_G N \oplus \ker \pi_G$, and $N^G = \pi_GN$. In particular, $M=M^G\oplus \ker \pi_G$ and thus $M^G$ is a direct factor of the free module $M$ and thus is projective. For item 2 the proof is the same as in \cite[Prop. 2.2.1]{BC08}.
\end{proof}

\section{Locally analytic vectors}
We now recall some of the theory of locally and pro-analytic vectors, following the presentation of Emerton in \cite{Emer} and of Berger in \cite{Ber14MultiLa}.

Let $G$ be a $p$-adic Lie group, and let $W$ be a $\Qp$-Banach representation of $G$. Let $H$ be an open subgroup of $G$ such that there exist coordinates $c_1,\cdots,c_d : H \to \Zp$ giving rise to an analytic bijection $\cbf : H \to \Z_p^d$. We say that $w \in W$ is an $H$-analytic vector if there exists a sequence $\left\{w_{\kbf}\right\}_{\kbf \in \N^d}$ such that $w_{\kbf} \rightarrow 0$ in $W$ (for the cofinite filter) and such that $g(w) = \sum_{\kbf \in \N^d}\cbf(g)^{\kbf}w_{\kbf}$ for all $g \in H$. We let $W^{H-\an}$ be the space of $H$-analytic vectors. This space injects into $\cal{C}^{\an}(H,W)$, the space of all analytic functions $f : H \to W$.  Note that $\cal{C}^{\an}(H,W)$ is a Banach space equipped with its usual Banach norm, so that we can endow $W^{H-\an}$ with the induced norm, that we will denote by $||\cdot ||_H$. With this definition, we have $||w||_H = \sup_{\kbf \in \N^d}||w_{\kbf}||$ and $(W^{H-\an},||\cdot||_H)$ is a Banach space.

The space $\cal{C}^{\an}(H,W)$ is endowed by an action of $H \times H \times H$, given by
\[
((g_1,g_2,g_3)\cdot f)(g) = g_1 \cdot f(g_2^{-1}gg_3)
\]
and one can recover $W^{H-\an}$ as the closed subspace of $\cal{C}^{\an}(H,W)$ of its $\Delta_{1,2}(H)$-invariants,  where $\Delta_{1,2} : H \to H \times H \times H$ denotes the map $g \mapsto (g,g,1)$ (we refer the reader to \cite[§3.3]{Emer} for more details).

We say that a vector $w$ of $W$ is locally analytic if there exists an open subgroup $H$ as above such that $w \in W^{H-\an}$. Let $W^{\la}$ be the space of such vectors, so that $W^{\la} = \bigcup_{H}W^{H-\an}$, where $H$ runs through a fundamental system of open subgroup neighborhoods of $G$. The space $W^{\la}$ is naturally endowed with the inductive limit topology, so that it is an LB space. 

\begin{lemm}
\label{ringla}
If $W$ is a ring  such that $||xy|| \leq ||x|| \cdot ||y||$ for $x,y \in W$, then
\begin{enumerate}
  \item $W^{H-\an}$ is a ring, and $||xy||_H \leq||x||_H \cdot ||y||_H$ if $x,y \in W^{H-\an}$;
  \item if $w \in W^\times \cap W^{\la}$, then $1/w \in W^{\la}$. In particular, if $W$ is a field, then  $W^{\la}$ is also a field.
\end{enumerate}
\end{lemm}
\begin{proof}
See \cite[Lemm. 2.5]{Ber14SenLa}.
\end{proof}

It is often useful to choose a specific fundamental system of open neighborhoods of $G$: let $G_0$ be a compact open subgroup of $G$ which is $p$-valued and saturated (see \cite[\S 26 and 27]{schneider2011p} for the definition and proof of existence), with coordinates $\c$, and set $G_{n}=G^{p^{n}}=\left\{ g^{p^{n}}:g\in G_{0}\right\}$ for $n\in\N$. We say that such a system $(G_n)_{n\in\N}$ is a system of coordinates for $G$.

The normalization is such that for $w\in W^{G_{n}-\an}$ we can write
\[
g(w)=\sum_{\mathbf{k\in}\mathbb{N}^{d}}c(g)^{\mathbf{k}}w_{\mathbf{k}}
\]
for $g\in G_{n}$ and $\left\{ w_{\mathbf{k}}\right\} _{\mathbf{k}\in\N^{d}}$
with $p^{n\left|\mathbf{k}\right|}w_{\mathbf{k}}\rightarrow0$, and
the Banach norm is given by
\[
|\!|w|\!|_{G_{n}-\an}=\sup_{\mathbf{k}}|\!|p^{n\mathbf{k}}w_{\mathbf{k}}|\!|.
\]
It is easy to check if $w\in W^{G_{n}-\an}$ then $|\!|w|\!|_{G_{m}-\an}\leq|\!|w|\!|_{G_{m+1}-\an}$
for $m\geq n$ and $|\!|w|\!|_{G_{m}-\an}=|\!|w|\!|$
for $m\gg n$ (see \cite[Lemm. 2.4]{Ber14SenLa}).

In the case where $G=\Gamma_K$ is the Galois group of the cyclotomic extension $K_\infty/K$, the map $\log \chi_{\cycl} : \Gamma_K \to \Z_p^\times$ induces isomorphisms $\Gamma_n \simeq p^n\Z_p$ for $n \gg 0$, where $\Gamma_n = \Gal(K_\infty/K(\mu_{p^n}))$, so that the groups $\Gamma_n$ form such a fundamental system of open neighborhoods of $\Gamma_K$ for $n$ big enough.

Let $W$ be a Fréchet space whose topology is defined by a sequence $\left\{p_i\right\}_{i \geq 1}$ of seminorms. Let $W_i$ be the Hausdorff completion of $W$ at $p_i$, so that $W = \varprojlim\limits_{i \geq 1}W_i$. The space $W^{\la}$ can be defined, but as stated in \cite{Ber14MultiLa} and explained in \S 7 of \cite{Poyetonlocanaperiods}, this space is too small in general for what we are interested in, and so we give the following definition, following \cite[Def. 2.3]{Ber14MultiLa}:

\begin{defi}
If $W = \varprojlim\limits_{i \geq 1}W_i$ is a Fréchet representation of $G$, then we say that a vector $w \in W$ is pro-analytic if its image $\pi_i(w)$ in $W_i$ is locally analytic for all $i$. We let $W^{\pa}$ denote the set of all pro-analytic vectors of $W$. 
\end{defi}

We extend the definition of $W^{\la}$ and $W^{\pa}$ for LB and LF spaces respectively in the obvious way. 

\begin{prop}
\label{lainla and painpa}
Let $G$ be a $p$-adic Lie group, let $B$ be a Banach $G$-ring and let $W$ be a free $B$-module of finite rank, equipped with a compatible $G$-action. If the $B$-module $W$ has a basis $w_1,\ldots,w_d$ in which $g \mapsto \Mat(g)$ is a globally analytic function $G \to \GL_d(B) \subset M_d(B)$, then
\begin{enumerate}
\item $W^{H-\an} = \bigoplus_{j=1}^dB^{H-\an}\cdot w_j$ if $H$ is a subgroup of $G$;
\item $W^{\la} = \bigoplus_{j=1}^dB^{\la}\cdot w_j$.
\end{enumerate}
Let $G$ be a $p$-adic Lie group, let $B$ be a Fréchet $G$-ring and let $W$ be a free $B$-module of finite rank, equipped with a compatible $G$-action. If the $B$-module $W$ has a basis $w_1,\ldots,w_d$ in which $g \mapsto \Mat(g)$ is a pro-analytic function $G \to \GL_d(B) \subset M_d(B)$, then
$$W^{\pa} = \bigoplus_{j=1}^dB^{\pa}\cdot w_j.$$
\end{prop}
\begin{proof}
The part for Banach ring is proven in \cite[Prop. 2.3]{Ber14SenLa} and the one for Fréchet rings is proven in \cite[Prop. 2.4]{Ber14MultiLa}.
\end{proof}

For $\alpha \in \Delta$, let $G_\alpha$ be a $p$-adic Lie group. Let $G = \prod_{\alpha \in \Delta}G_\alpha$. This is also a $p$-adic Lie group, and if $(G_{\alpha,n})_{\alpha \in \Delta, n\in\N}$ are systems of coordinates of $G_{\alpha}, \alpha \in \Delta$, it is easy to see that the subgroups of $G$ defined by $(G_n = \prod_{\alpha \in \Delta}G_{\alpha,n})_{n\in\N}$ form a system of coordinates of $G$.

\begin{lemm}
\label{lemma locanamulti}
For $\alpha \in \Delta$, let $G_\alpha$ be a $p$-adic Lie group. Let $G = \prod_{\alpha \in \Delta}G_\alpha$, and let $(G_{\alpha,n})_{\alpha \in \Delta, n\in\N}$ be systems of coordinates of $G_{\alpha}, \alpha \in \Delta$. We have 
$$\cal{C}^{\an}(G_n,\Qp) \simeq \hat{\otimes}_{\Qp}^{\alpha \in \Delta}\cal{C}^{\an}(G_{\alpha,n},\Qp).$$
\end{lemm}
\begin{proof}
Suppose that $\mathbf{c}_{\alpha}:G_{\alpha,n}\to \Z^{d_{\alpha}}_p$ are analytic bijections, for $\alpha\in\Delta$ and $n\in\N$. 
These induce analytic bijections $\mathbf{c}:G_{n}\to \Z^{d_{\alpha_1}+\dots+d_{\alpha_{\delta}}}_p$, for $n\in\N$.
We have then that 
\begin{dmath*}\hat{\otimes}_{\Qp}^{\alpha \in \Delta}\cal{C}^{\an}(G_{\alpha,n},\Qp)\simeq \hat{\otimes}_{\Qp}^{\alpha \in \Delta}\cal{C}^{\an}\left(\Z^{d_{\alpha}}_p,\Qp\right)\simeq\hat{\otimes}_{\Qp}^{\alpha \in \Delta}\tate{\Qp}{x^{\alpha}_1,\dots,x^{\alpha}_{d_\alpha}}\simeq\tate{\Qp}{x^{\alpha_1}_1,\dots,x^{\alpha_1}_{d_\alpha},\dots,x^{\alpha}_1,\dots,x^{\alpha_{\delta}}_{d_{\alpha_{\delta}}}}\simeq \cal{C}^{\an}\left(\Z^{d_{\alpha_1}+\dots+d_{\alpha_{\delta}}}_p,\Qp\right)\simeq\cal{C}^{\an}(G_n,\Qp). \end{dmath*}
\end{proof}

\begin{coro}
\label{coro computes locanamulti}
For $\alpha \in \Delta$, let $G_\alpha$ be a $p$-adic Lie group, and let $B_\alpha$ be a Banach (or LB, or Fréchet, or LF) space, endowed with a continuous $\Qp$-linear action of $G_\alpha$ \footnote{In the case of Fréchet or LF spaces, we also assume that the $B_\alpha$ satisfy the same assumptions as in Lemma \ref{inv and tensors} and Remark \ref{invariants LF spaces}.}. Let $G = \prod_{\alpha \in \Delta}G_\alpha$ and let $B = \hat{\otimes}_{\Qp}^{\alpha \in \Delta}B_\alpha$, where the action of $G_\alpha$ on the tensor product is trivial on $B_\beta$ if $\alpha \neq \beta$. We have 
$$B^{G_n-\an} = \hat{\otimes}_{\Qp}^{\alpha \in \Delta}(B_\alpha)^{G_{\alpha,n}-\an},$$
$$B^{G-\la} = \hat{\otimes}_{\Qp}^{\alpha \in \Delta}(B_\alpha)^{G_\alpha-\la},$$
and
$$B^{G-\pa} = \hat{\otimes}_{\Qp}^{\alpha \in \Delta}(B_\alpha)^{G_\alpha-\pa}.$$
\end{coro}
\begin{proof}
By lemma \ref{lemma locanamulti}, we have 
$$\cal{C}^{\an}(G_n,\Qp) \simeq \hat{\otimes}_{\Qp}^{\alpha \in \Delta}\cal{C}^{\an}(G_{\alpha,n},\Qp).$$
Since $B^{G_n-\an} = (\cal{C}^{\an}(G_n,\Qp)\hat{\otimes}_{\Qp}B)^{G_n}$ where the action on the tensor product is diagonal and the action on $\cal{C}^{\an}(G_n,\Qp)$ is given by the $\Delta_{1,2}$-action, we obtain that
$$B^{G_n-\an} \simeq (\hat{\otimes}_{\Qp}^{\alpha \in \Delta}\cal{C}^{\an}(G_{\alpha,n},\Qp))\hat{\otimes}_{\Qp}B)^{G_n}.$$
Using the fact that $G_n = \prod_{\alpha \in \Delta}G_{\alpha,n}$ and lemma \ref{inv and tensors} successively for each of the $G_{\alpha,n}$, we obtain that $B^{G_n-\an} = \hat{\otimes}_{\Qp}^{\alpha \in \Delta}(B_\alpha)^{G_{\alpha,n}-\an}.$ The result for locally analytic vectors follows by taking the inductive limit over $n$, and the result for pro-analytic vectors follows by taking the projective limit over the Hausdorff completions defining the Fréchet (or LF) topology.
\end{proof}

\section{Classical and multivariable rings of periods}

\subsection{Classical rings of periods and cyclotomic $(\phi,\Gamma)$-modules}
\label{subsection classical rings}

Let $p$ be a prime and let $k$ be a perfect field of characteristic $p$. Let $F=W(k)[1/p]$ the field of fractions of its ring of integers $\O_F=W(k)$. This is a complete discrete valuation field of characteristic $0$ for the $p$-adic valuation $v_p$, with residue field $k$. We let $K$ be a finite totally ramified extension of $F$. We fix $\Kbar$ an algebraic closure of $K$ and we let $C$ denote the $p$-adic completion of $\Kbar$. Let $v_p$ denote the $p$-adic valuation on $C$ normalized so that $v_p(p)=1$. We let $F_\infty=F(\mu_{p^\infty})$ be the cyclotomic extension of $F$. If $L$ is a finite extension of $F$, we let, for $n \geq 1$, $L_n = L(\mu_{p^n})$ be the extension of $L$ generated by the $p^n$-th roots of unity, and let $L_\infty= \bigcup_{n \geq 1}L(\mu_{p^n})=L \cdot F_\infty$ be the cyclotomic extension of $L$. We let $H_L = \Gal(\Kbar/L_\infty)$ and $\Gamma_L = \Gal(L_\infty/L)$. Recall that the cyclotomic character $\chi_\cycl : \G_L \to \Z_p^\times$ factors through $\Gamma_L$ and identifies it with an open subset of $\Z_p^\times$.

Let $C^\flat = \varprojlim\limits_{x \to x^p}C$ denote the tilt of $C$ and let $\O_C^\flat = \varprojlim\limits_{x \to x^p}\O_C$ be the tilt of $\O_C$ as defined in \cite{scholzeperf}. Recall that the ring $\O_C^\flat$ is the ring of integers of $C^\flat$ for the valuation $v_{\E}$ induced by $v$ on $C^\flat$, and that $C^\flat$ is an algebraically closed complete field of characteristic $p$ endowed with a continuous action of $\G_K$ coming from the one on $C$. We let $\Atplus=W(\O_C^\flat)$ (which is also classically denoted by $\A_{\inf}$). 

Recall that there is a surjective $\G_K$-equivariant ring homomorphism 
$$\theta : \Atplus \ra \O_C$$
given by $\theta(\sum_{k \geq 0}p^k[x_k]) = \sum_{k \geq 0}p^kx_k^{(0)}$. The kernel of $\theta$ is principal, generated by $\xi = [\tilde{p}]-p$ where $\tilde{p} = (p,p^{1/p},\ldots) \in \O_C^\flat$. This map extends to a surjective $\G_K$-equivariant ring homomorphism $\theta : W(\O_C^\flat)[1/p] \ra C$. Recall that $\Bdrplus$ is defined as the completion for the $\ker(\theta)$-adic topology of $W(\O_C^\flat)[1/p]$. The power series defining $\log[\epsilon]$ converges in $\Bdrplus$ to an element $t$ that generates the maximal ideal $\ker(\theta : \Bdrplus \to C)$ of $\Bdrplus$, so that $\Bdr = \Bdrplus[1/t]$. Note that the action of $\G_{F}$ on $t$ is given by $g(t) = \chi_{\cycl}(g)\cdot t$.

We choose a sequence $\epsilon \in \O_C^{\flat}$ of compatible $p^n$-th roots of unity (with $\epsilon^{(1)} \neq 1$). Let $\overline{u} = \epsilon-1 \in \O_C^\flat$ and let $\E_{F}:=k(\!(\overline{u})\!) \subset C^\flat$. Let $\E = \E_{F}^{sep}$ be the separable closure of $\E_{F}$ inside $C^\flat$. The field $\E_{F}$ is left invariant by the action of $H_{F}$ so that we have a morphism $H_{F} \to \Gal(\E/\E_{F})$. By \cite[Thm. 3.2.2]{Win83}, it is actually an isomorphism. We also let $\E_K=\E^{H_K}$. Note that $\Gamma_K$ acts on $\E_K$, and that the action of $\G_{F}$ on $\overline{u}$ is given by $g(\overline{u})=(1+\overline{u})^{\chi_\cycl(g)}-1$. 

Let $\At = W(C^\flat)$ and let $u = [\epsilon]-1 \in \Atplus$. Let $\A_{F}$ be the $p$-adic completion of $\O_F(\!(u)\!)$ inside $\At$. This is a discrete valuation ring with residue field $\E_{F}$. Since
$$\phi(u) = (1+u)^p-1 \quad \textrm{and} \quad g(u)=(1+u)^{\chi_\cycl(g)}-1 \textrm{ if } g \in \G_{F},$$
the ring $\A_{F}$ and its field of fractions $\B_{F} := \A_{F}[1/p]$ are both stable by $\phi$ and $\G_{F}$.
We let $\Bt=\At[1/p]$.
If $L$ is a finite extension of $F$, we let $\Bt_L = \Bt^{H_L}$ and $\At_L = \At^{H_L}$. 

For $r > 0$, we define $\Bt^{\dagger,r}$ the subset of overconvergent elements of ``radius'' $r$ of $\Bt$, by

$$\Bt^{\dagger,r}=\left\{x = \sum_{n \gg -\infty}p^n[x_n] \textrm{ such that } \lim\limits_{k \to +\infty}v_{\E}(x_k)+\frac{pr}{p-1}k =+\infty \right\}$$

and we let $\Bt^\dagger = \bigcup_{r > 0}\Bt^{\dagger,r}$ be the subset of all overconvergent elements of $\Bt$. 

Let $\B_{F}^{\dagger,r}$ be the subset of $\B_{F}$ given by
$$\B_{F}^{\dagger,r}=\left\{\sum_{i \in \Z}a_iu^i, a_i \in F \textrm{ such that the } a_i \textrm{ are bounded and } \lim\limits_{i \to - \infty}v_p(a_i)+i\frac{pr}{p-1} = +\infty \right\},$$
and note that $\B_{F}^{\dagger,r} = \B_{F} \cap \Bt^{\dagger,r}$.  

Let $\B_{F}^\dagger = \bigcup_{r > 0}\B_{F}^{\dagger,r}$. By \S 2 of \cite{matsuda1995local}, this is a Henselian field, and its residue ring is still $\E_{F}$. Since $\B_{F}^\dagger$ is Henselian, if $L$ is a finite extension of $F$, there exists a finite unramified extension $\B_L^\dagger/\B_{F}^\dagger$ inside $\Bt$, of degree $f=[L_\infty:F_\infty]$ and whose residue field is $\E_L$. Therefore, there exists $r(K) > 0$ and elements $x_1,\ldots,x_f$ in $\B_L^{\dagger,r(L)}$ such that $\B_L^{\dagger,s} = \oplus_{i=1}^f \B_{F}^{\dagger,s}\cdot x_i$ for all $s \geq r(L)$. We let $\B_L$ be the $p$-adic completion of $\B_L^\dagger$ and we let $\A_L$ be its ring of integers for the $p$-adic valuation. One can show that $\B_L$ is a subfield of $\Bt$ stable under the action of $\phi$ and $\Gamma_K$ (see for example \cite[Prop. 6.1]{colmez2008espaces}). Let $\A$ be the $p$-adic completion of $\bigcup_{L/F}\A_L$, taken over all the finite extensions $L/\Qp$. Let $\B = \A[1/p]$. Note that $\A$ is a complete discrete valuation ring whose field of fractions is $\B$ and with residue field $\E$. Once again, both $\A$ and $\B$ are stable by $\phi$ and $\G_{F}$. Moreover, we have $\A^{H_K}=\A_K$ and $\B_K=\B^{H_K}$, so that $\A_K$ is a complete discrete valuation ring with residue field $\E_K$ and fraction field $\B_K = \A_K[1/p]$. If $L$ is a finite extension of $K$, then $\B_L/\B_K$ is an unramified extension of degree $[L_\infty:K_\infty]$ and if $L/K$ is Galois then so is $\B_L/\B_K$, and we have the following isomorphisms: $\Gal(\Bt_L/\Bt_K) = \Gal(\B_L/\B_K) = \Gal(\E_L/\E_K) = \Gal(L_\infty/K_\infty) = H_K/H_L$.

For $r \geq 0$, we define a valuation $V(\cdot,r)$ on $\Btplus[\frac{1}{[\overline{u}]}]$ by setting
$$V(x,r) = \inf_{k \in \Z}(k+\frac{p-1}{pr}v_{\E}(x_k))$$
for $x = \sum_{k \gg - \infty}p^k[x_k]$. If $I$ is a closed subinterval of $[0;+\infty[$, we let $V(x,I) = \inf_{r \in I}V(x,r)$. We then define the ring $\Bt^I$ as the completion of $\Btplus[1/[\overline{u}]]$ for the valuation $V(\cdot,I)$ if $0 \not \in I$, and as the completion of $\Btplus$ for $V(\cdot,I)$ if $I=[0;r]$. We will write $\Bt_{\mathrm{rig}}^{\dagger,r}$ for $\Bt^{[r,+\infty[}$ and $\Bt_{\mathrm{rig}}^+$ for $\Bt^{[0,+\infty[}$. We also define $\Bt_{\mathrm{rig}}^\dagger = \bigcup_{r \geq 0}\Bt_{\mathrm{rig}}^{\dagger,r}$. We also let $\At_{\rig}^{\dagger,r}$ be the ring of integers of $\Bt_{\rig}^{\dagger,r}$ for the valuation $V(\cdot,r)$.

We define $\At^{\dagger,r}$ by 

$$\At^{\dagger,r}=\left\{x = \sum_{n \geq 0}p^n[x_n] \in \At \textrm{ such that } V(x,r) \geq 0 \textrm{ and } \lim\limits_{k \to +\infty}v_{\E}(x_k)+\frac{pr}{p-1}k \geq 0 \right\}.$$

We let $\rho(r) = p^{-1/r}$ and we let $\At^{(0,\rho(r)]} = \At^{\dagger,r}[1/[\overline{u}]]$. We also let $\At^{\dagger,r}_K = (\At^{\dagger,r})^{H_K} = \At_K \cap \At^{\dagger,r}$ and $\At^{(0,\rho(r)]}_K = (\At^{(0,\rho(r)]})^{H_K} = \At_K \cap \At^{(0,\rho(r)]}$. We also define $\A^{\dagger,r} = \A \cap \At^{\dagger,r}$, $\A^{\dagger,r}_K =\A \cap \At^{\dagger,r}_K$ and $\A^{(0,\rho(r)]} = \A \cap \At^{(0,\rho(r)]}$. All of these rings are complete for the valuation $V(\cdot,r)$.

Let $I$ be a subinterval of $]1,+\infty[$ or such that $0 \in I$. Let $f(Y) = \sum_{k \in \Z}a_kY^k$ be a power series with $a_k \in F$ and such that $v_p(a_k)+k/\rho \to +\infty$ when $|k| \to + \infty$ for all $\rho \in I$. The series $f(u)$ converges in $\Bt^I$ and we let $\B_{F}^I$ denote the set of all $f(\pi)$ with $f$ as above. It is a subring of $\Bt_{F}^I$. 

We also write $\B_{\mathrm{rig},F}^{\dagger,r}$ for $\B_{F}^{[r;+\infty[}$. It is a subring of $\B_{F}^{[r;s]}$ for all $s \geq r$ and note that the set of all $f(u) \in \B_{\mathrm{rig},F}^{\dagger,r}$ such that the sequence $(a_k)_{k \in \Z}$ is bounded is exactly the ring $\B_{F}^{\dagger,r}$. Let $\B_{F}^{\dagger}=\cup_{r \gg 0}\B_{F}^{\dagger,r}$ which we call the Robba ring over $F$. 

For $n \geq 0$ we let $r_n:=p^{n-1}(p-1)$. 

If $L$ is a finite extension of $F$ and if $r(L) \leq \min(I)$, we let $\B_{L}^I$ be the completion of $\B_L^{\dagger,r(L)}$ for $V(\cdot,I)$, so that $\B_L^I=\oplus_{i=1}^f\B_{F}^I\cdot x_i$. 

We actually have a better description of the rings $\B_{\rig,K}^{\dagger,r}$ in general:

\begin{prop}
Let $L$ be a finite extension of $F$ and let $e_L=[L_\infty:F_\infty]$. 
\begin{enumerate}
\item There exists $u_L \in \A_L^{\dagger,r(L)}$ whose image modulo $p$ is a uniformizer of $\E_L$ and such that, for $r \geq r(L)$, every element $x \in \B_L^{\dagger,r}$ can be written as $x = \sum_{k \in \Z}a_ku_L^k$, where $a_k \in F'=W(\overline{k})[1/p] \cap L_\infty$, and the power series $\sum_{k \in \Z}a_kT^k$ is bounded on $\left\{p^{-1/e_Lr} \leq |T| < 1 \right\}$.
\item Let $\mathcal{H}^{\alpha}_{F'}(T)$ be the set of power series $\sum_{k\in \Z}a_kT^k$ where $a_k \in F'$ and such that, for all $\rho \in [\alpha;1[, \lim\limits_{k \to \pm \infty}|a_k|\rho^k=0$ and let $\alpha_L^r = p^{-1/e_Lr}$. Then the map $\mathcal{H}^{\alpha}_{F'}(T) \rightarrow \B_{\rig,L}^{\dagger,r}$ sending $f$ to $f(u_L)$ is an isomorphism.
\end{enumerate}
\end{prop} 
\begin{proof}
The first item is proved in \cite[Prop. 7.5]{colmez2008espaces} and the second one in \cite[Prop. 7.6]{colmez2008espaces}. Be careful that the notations for the rings and the normalizations of the valuations used in Colmez's paper are a bit different than ours. 
\end{proof}

\begin{defi}
A $(\phi,\Gamma_K)$-module $\D$ on $\A_K$ (resp. $\B_K$) is an $\A_K$-module of finite rank (resp. a finite dimensional $\B_K$-vector space) endowed with semilinear actions of $\Gamma_K$ and $\phi$ that commute with each other.

It is said to be étale if $1 \otimes \phi: \phi^*\D \to \D$ is an isomorphism (resp. if there exists a basis of $\D$ such that $\Mat(\phi) \in \GL_d(\A_K)$).
\end{defi}

If $V$ is a $p$-adic representation of $\G_K$, we set 
$$\D(V) = (\B \otimes_{\Qp}V)^{H_K}.$$ 
Note that $\D(V)$ is a $(\phi,\Gamma_K)$-module. Moreover, if $V$ is a $p$-adic representation of $\G_K$, then $\D(V)$ is étale and $(\B \otimes_{\B_K}\D(V))^{\phi=1}$ is canonically isomorphic to $V$ (see \cite[Prop. 1.2.6]{Fon90}). The functors $V \mapsto \D(V)$ and $\D \mapsto (\B \otimes_{\B_K}D)^{\phi=1}$ then induce an equivalence of tannakian categories between $p$-adic representations of $\G_K$ and étale $(\phi,\Gamma_K)$-modules over $\B_K$.

The following theorem is the main result of \cite{cherbonnier1998representations} and shows that every étale $(\phi,\Gamma_K)$-module is the base change to $\B_K$ of an overconvergent module:
\begin{theo}
\label{theo cherbonniercolmez}
If $\D$ is an étale $(\phi,\Gamma_K)$-module, then the set of free sub-$\B_K^\dagger$-modules of finite type stable by $\phi$ and $\Gamma_K$ admits a bigger element $\D^\dagger$ and one has $\D= \B_K \otimes_{\B_K^{\dagger}}\D^\dagger$. 
\end{theo}

In particular, if $V$ is a $p$-adic representation of $\G_K$, then there exists an étale $(\phi,\Gamma_K)$-module over $\B_K^\dagger$ which we will denote by $\D^\dagger(V)$ and such that $\D(V) = \B_K \otimes_{\B_K^{\dagger}}\D^\dagger(V)$. We let $\D_{\rig}^\dagger(V) = \B_{\rig,K}^\dagger \otimes_{\B_K^{\dagger}}\D^\dagger(V)$.

\subsection{Multivariable setting and first multivariable constructions}

Let $\Delta$ be a finite set, and let $\delta$ denote its cardinal. For each $\alpha \in \Delta$, we fix a finite extension $K_\alpha$ of $F$. We define $\G_{\Kunder,\Delta}:= \prod_{\alpha \in \Delta}\G_{K_\alpha}$, $H_{\Kunder,\Delta}:=\prod_{\alpha \in \Delta}H_{K_\alpha}$ and $\Gamma_{\Kunder,\Delta}:=\prod_{\alpha \in \Delta}\Gamma_{K_\alpha}$. 

For $\alpha \in \Delta$, we let $\G_{\Kunder,\alpha}$ denote the image of $\G_{K_\alpha}$ by the group homomorphism $\iota_\alpha : \G_{K_\alpha} \ra \G_{\Kunder,\Delta}$ mapping $g$ to the element whose component of index $\alpha$ is $g$ and whose other components are $1$. We define groups $H_{\Kunder,\alpha}$ and $\Gamma_{\Kunder,\alpha}$ in the same fashion. We let $\chi_\Delta$ denote the $\Delta$-cyclotomic character of $\G_{\Kunder,\Delta}$ with values in $(\Z_p^\times)^\delta$ by $\chi_{\Delta}= \prod_{\alpha \in \Delta}\chi_\cycl$. It factors through $\Gamma_{\Kunder,\Delta}$ and identifies it with an open subgroup of $(\Z_p^\times)^\delta$.  

Our goal is to generalize the construction of classical rings of $p$-adic periods in order to study $p$-adic representations of $\G_{\Kunder,\Delta}$, following and expanding upon \cite{ZabMulti0,PalZab21,BriChiaMaz21,KedCarZab}. In those papers except the last one, the fields $K_\alpha$ are taken to be the same, but we see no reason to restrict ourselves to this setting, and thus will work with this level of generality. 

The main way to construct rings of $p$-adic periods for products of Galois groups is to take some completion of the tensor product of $\delta$ copies of the classical rings of periods one wishes to consider. One tricky question is which ring our tensor products have to be taken over. 

In \cite{PalZab21} the fields $K_\alpha$ are all equal to $\Qp$ so that the authors are taking the tensor products over either $\Qp$ or $\Z_p$ (depending on the situation). In \cite{BriChiaMaz21} the fields $K_\alpha$ are all equal to the same $K$ which is a finite extension of $\Qp$, with residue field $k$, and the tensor products are taken over either $W(k)$ or $W(k)[1/p]$, which simplify the cohomology computations (over taking the tensor products over either $\Zp$ or $\Qp$). However, this means that the multivariable rings of periods they construct depend on $k$, and thus one would get different rings of de Rham periods depending on the fields $K_\alpha$ we chose. 

Because of this, we make the choice here to do everything over $\Qp$ (or $\Zp$) so that the construction of the rings of multivariable crystalline, semi-stable and de Rham periods will not depend on our choice of the $K_\alpha$s. Moreover, it extends naturally to our setting where we allow ourselves to consider distinct fields $K_\alpha$, for $\alpha\in\Delta$. Note that this was also the point of view in \cite{KedCarZab} in order to generalize the results of \cite{ZabMulti0} and \cite{PalZab21} for finite extensions of $\Qp$.

In particular, this means that our rings $\O_{C_\Delta}$ and $\B_{\dR,\Delta}^+$ will correspond to the particular case $k=\F_p$ of \cite{BriChiaMaz21}, so that we can apply their results and proofs because the constructions and definitions we're using here are specializations of theirs, by only considering tensor product over $\Zp$ (resp. $\Qp$) instead of the more general $W(k)$ (resp. $W(k)[1/p])$.

If $(L_\alpha)_{\alpha \in \Delta}$ is a set of subfields of $C$ such that $L_\alpha \supset K_\alpha$, we let $\O_{\Lunder_\Delta}$ be the $p$-adic completion of the tensor product $\otimes_{\Zp}^{\alpha \in \Delta}\O_{L_\alpha}$ and we let $\Lunder_\Delta = \O_{\Lunder_\Delta}[1/p]$. In order to simplify the notations, in the case where all the $L_\alpha$ are equal to $C$, we write $\O_{C_\Delta}$ and $C_\Delta$ respectively for $\O_{\Lunder_\Delta}$ and $\Lunder_\Delta$.

Note that $\O_{C_\Delta}$ and $C_\Delta$ are naturally endowed with an action of $\G_{\Kunder,\Delta}$, and that $C_\Delta$ comes equipped with the $p$-adic valuation $v_p$ coming from the one on $C$. 

Since $\O_{C}^\Delta/(p) \simeq (\O_C/(p))^{\otimes \Delta}$ where the tensor product is taken over $\F_p$ and since the Frobenius map on $\O_C/(p)$ is surjective, the $\Zp$-algebra $\O_{C_\Delta}$ is perfectoid 
and we can define its tilt $\O_{C_\Delta}^\flat$ by 
$$\O_{C_\Delta}^\flat = \varprojlim\limits_{x \to x^p}\O_{C_\Delta} = \{(x^{(0)},x^{(1)},\dots) \in \O_{C_\Delta}^{\N}~: (x^{(n+1)})^p=x^{(n)}\}.$$
This is a perfect $\Fp$-algebra endowed with an action of $\G_{\Kunder,\Delta}$ coming from the one on $\O_{C_\Delta}$ and it is complete for the valuation coming from the one on $\O_{C_{\Delta}}$. 

Let $(\O_C^\flat)^{\otimes \Delta} = \O_C^\flat \otimes_{\Fp} \cdots \otimes_{\Fp}\O_C^\flat$ where the copies of $\O_C^\Delta$ are indexed by $\Delta$. For $\alpha \in \Delta$, we let 
$$\tilde{p}_\alpha = 1 \otimes_{\F_p} \cdots \otimes_{\F_p} 1 \otimes_{\F_p} \tilde{p} \otimes_{\F_p} 1 \otimes_{\F_p} \cdots \otimes_{\F_p} 1$$
where $\tilde{p}$ is the factor of index $\alpha$, and let
$$\epsilon_\alpha = 1 \otimes_{\F_p} \cdots \otimes_{\F_p} 1 \otimes_{\F_p} \epsilon \otimes_{\F_p} 1 \otimes_{\F_p}\cdots \otimes_{\F_p} 1.$$

We let $I_{\tilde{p}}$ denote the ideal of $(\O_C^\flat)^{\otimes \Delta}$ generated by $\{\tilde{p}_\alpha\}_{\alpha \in \Delta}$. 

\begin{prop}
\label{prop OCDeltaflat completion of OCflatDelta}
The ring $(\O_C^\flat)^{\otimes \Delta}$ is $I_{\tilde{p}}$-adically separated, and we have a natural injective morphism of $k$-algebras $(\O_C^\flat)^{\otimes \Delta} \ra \O_{C_\Delta}^\flat$ which induces an isomorphism $(\O_C^\flat)^{\otimes \Delta}/I_{\tilde{p}} \simeq \O_{C_\Delta}^\flat$. Moreover, $\O_{C_\Delta}^\flat$ is isomorphic to the $I_{\tilde{p}}$-adic completion of $(\O_C^\flat)^{\otimes \Delta}$. 
\end{prop}
\begin{proof}
This is \cite[Lemma 4.2 and Proposition 4.3]{BriChiaMaz21}.
\end{proof}

We still denote by $\tilde{p}_\alpha$ and $\epsilon_\alpha$ the images of those elements \textit{via} the embedding $(\O_C^\flat)^{\otimes \Delta} \ra \O_{C_\Delta}^\flat$. 

For $\alpha \in \Delta$, we let $\phi_\alpha$ denote the Frobenius map of index $\alpha$ on $(\O_C^\flat)^{\otimes \Delta}$, defined on pure tensors by
$$\phi_\alpha(x_1 \otimes_{\F_p} \cdots \otimes_{\F_p} x_\delta) = x_1 \otimes_{\F_p} \cdots \otimes_{\F_p} x_{\alpha-1} \otimes_{\F_p} x_\alpha^p \otimes_{\F_p} x_{\alpha+1} \otimes_{\F_p} \cdots \otimes_{\F_p} x_\delta$$
and extended by $k$-linearity on $(\O_C^\flat)^{\otimes \Delta}$. Proposition \ref{prop OCDeltaflat completion of OCflatDelta} shows that this extends naturally to a map $\phi_\alpha : \O_{C_\Delta}^\flat \ra \O_{C_\Delta}^\flat$. We let $\phi = \phi_{\alpha_1} \circ \cdots \circ \phi_{\alpha_\delta}$ on $\O_{C_\Delta}^\flat$. We have that $\phi : (\O_C^\flat)^{\otimes \Delta} \ra (\O_C^\flat)^{\otimes \Delta}$ is equal to $x \mapsto x^p$ on pure tensors, so that it is equal to the usual absolute Frobenius on $(\O_C^\flat)^{\otimes \Delta}$ and thus its image in $\O_{C_\Delta}^\flat$ also is the absolute Frobenius.

We let $\A_{\mathrm{inf},\Delta}=W(\O_{C_\Delta}^\flat)$. For $\alpha \in \Delta$, we let $\xi_\alpha = [\tilde{p_\alpha}]-p \in \A_{\mathrm{inf},\Delta}$ (so this is equal to $-\xi_\alpha$ in the notations of \cite{BriChiaMaz21}), $\varpi_\alpha = [\epsilon_\alpha]-1 \in \A_{\mathrm{inf},\Delta}$ and $\omega_\alpha = \frac{\varpi_\alpha}{\phi_\alpha^{-1}(\varpi_\alpha)}$. By functoriality of Witt vectors, the maps $\phi_\alpha$ on $\O_{C_\Delta}^\flat$ give rise to maps that we still denote by $\phi_\alpha : \A_{\mathrm{inf},\Delta} \ra \A_{\mathrm{inf},\Delta}$, and we also let $\phi = \phi_{\alpha_1} \circ \cdots \circ \phi_{\alpha_\delta}$ (which also corresponds to the usual Frobenius on Witt vectors). Also by functoriality of Witt vectors, the $\O_F$-algebra $\A_{\mathrm{inf},\Delta}$ is endowed by an action of $\G_{K,\Delta}$. 

The map $\theta_\Delta : \A_{\mathrm{inf},\Delta} \ra \O_{C_\Delta}$ defined by $\theta_\Delta(\sum_{k \geq 0}p^k[x_k]) = \sum_{k \geq 0}p^kx_k^{(0)}$ is a surjective $\G_{\Kunder,\Delta}$-equivariant morphism of $\O_F$-algebras, and induces a surjective $\G_{\Kunder,\Delta}$-equivariant morphism of $F$-algebras $\theta_\Delta : \A_{\mathrm{inf},\Delta} [1/p] \ra C_\Delta$. 

The ideal $\ker(\theta_\Delta)$ of $\A_{\mathrm{inf},\Delta}$ is generated by $\{\xi_\alpha\}_{\alpha \in \Delta}$ (this is \cite[Corollary 4.4]{BriChiaMaz21}) and also by $\{\omega_\alpha\}_{\alpha \in \Delta}$ (one checks that $\xi_\alpha/\omega_\alpha$ is invertible in $\A_{\mathrm{inf},\Delta}$).

Note that our definition of $\A_{\mathrm{inf},\Delta}$ is not exactly the same as the one given in \S 4 of \cite{BriChiaMaz21}, as they define $\A_{\mathrm{inf},\Delta}$ as the completion of $W(\O_{C_\Delta}^\flat)$ for the $(p,\ker(\theta_\Delta))$-adic topology. The following lemma shows that $W(\O_{C_\Delta}^\flat)$ is already complete for this topology, so that their definition and ours coincide, and so that we can use the results they prove about $\A_{\mathrm{inf},\Delta}$ directly.

\begin{lemm}
\label{lemm Afindel p+ker complete}
    The ring $\A_{\inf,\Delta}=W(\O_{C_\Delta}^\flat)$ is $(p,\ker(\theta_\Delta))$-adically complete.
\end{lemm}
\begin{proof}
    First note that since $\ker(\theta_\Delta)$ is generated by $\{\xi_\alpha\}_{\alpha \in \Delta}$, we have $(p,\ker(\theta_\Delta)) = (p,[\tilde{p}_\alpha])_{\alpha \in \Delta}$. In order to simplify the notations, let us write $J = ([\tilde{p}_\alpha])_{\alpha \in \Delta}$, so that the ideals $p^nW(\O_{C_\Delta}^\flat)+J^m$ form a basis of open neighborhood of $0$ for the $(p,\ker(\theta_\Delta))$-adic topology.

    We have to prove that, if $(x_n)$ is a Cauchy sequence in $W(\O_{C_\Delta}^\flat)$ for the $\ker(\theta_\Delta)$-adic topology, it converges in $W(\O_{C_\Delta}^\flat)$. Let us write $x_n = \sum_{m \geq 0}p^m[x_{n,m}]$. Let $m \geq 1$. Since $J$ induces the $I_{\tilde{p}}$-adic topology on $\O_{C_\Delta}^\flat$ modulo $p$, the sequence $(x_{n,m})_{n \geq 1}$ is Cauchy for the $I_{\tilde{p}}$-adic topology, and thus converges in $\O_{C_\Delta}^\flat$ to an element $x_m$ as it is $I_{\tilde{p}}$-adically complete. Now if we let $x=\sum_{m \geq 0}p^m[x_m]$, the sequence $(x_n)$ converges to $x$ for the $(p,J)$-adic topology, and so we are done.
\end{proof}

\begin{prop}
\label{Ainfdel = padiccomp of xiadic}
    We have a natural isomorphism between $\A_{\inf,\Delta}$ and the $p$-adic completion of the $(\varpi_1,\cdots,\varpi_{\delta})$-completion of $\Atplus \otimes_{\Zp} \Atplus \otimes_{\Zp} \cdots \otimes_{\Zp}\Atplus$.
\end{prop}
\begin{proof}
    Let $A$ denote the $p$-adic completion of the $(\varpi_1,\cdots,\varpi_{\delta})$-completion of $\Atplus \otimes_{\Zp} \Atplus \otimes_{\Zp} \cdots \otimes_{\Zp}\Atplus$. By definition, $A$ is $p$-adically complete, and $A/pA$ is equal to the $(\overline{\varpi_1},\cdots,\overline{\varpi_{\delta}})$-completion of $\O_C^\flat \otimes_{\F_p}\O_C^\flat \otimes_{\Fp}\cdots\otimes_{\Fp}\O_C^\flat$, which is isomorphic to $\O_{C_\Delta}^\flat$ by Proposition \ref{prop OCDeltaflat completion of OCflatDelta}. By the universal property of Witt vectors, we get that $A \simeq W(\O_{C_\Delta}^\flat)$.
\end{proof}

\subsection{Multivariable rings and $(\phi,\Gamma)$-modules}

For $\r = (r_\alpha)_{\alpha \in \Delta}$ a set of nonnegative real numbers indexed by $\Delta$, we let $\A_{\Kunder,\Delta}^{\dagger,\r}:=\hat{\otimes}_{\Zp}^{\alpha \in \Delta}\A_{K_\alpha}^{\dagger,r_\alpha}, \At_{\Kunder,\Delta}^{\dagger,\r}:=\hat{\otimes}_{\Zp}^{\alpha \in \Delta}\At_{K_\alpha}^{\dagger,r_\alpha}, \At_{\Delta}^{\dagger,\r}:=\hat{\otimes}_{\Zp}^{\alpha \in \Delta}\At^{\dagger,r_\alpha}$ and $\A_{\Delta}^{\dagger,\r}:=\hat{\otimes}_{\Zp}^{\alpha \in \Delta}\A^{\dagger,r_\alpha}$ to be the completions of the tensor product, where the factors with supscripts $r_\alpha$ of the tensor product are endowed with the valuation $v_{r_\alpha}:=V(\cdot,r_\alpha)$ defined in \S \ref{subsection classical rings}. We let $v_{\r}$ denote the resulting valuation on those completed tensor products.

Once again, those rings are naturally endowed with an action of $\G_{\Kunder,\Delta}$, and for each $\alpha \in \Delta$, we have a map $\phi_\alpha$ induced by the map 
$$1 \otimes_{\Zp} \cdots \otimes_{\Zp} 1 \otimes_{\Zp} \phi_\alpha \otimes_{\Zp} 1 \otimes_{\Zp} \cdots \otimes_{\Zp} 1 $$ 
from those rings corresponding to $\r$ to the rings corresponding to ${\r}'$ where $\r'$ has the same components as $\r$ for $\beta \neq \alpha$, and ${r_\alpha}'=pr_\alpha$. 

If $\underline{I} = \prod_{\alpha \in \Delta}I_\alpha$ is a product of subintervals of $[0,+\infty[$, we let $\Bt_{\Delta}^{\underline{I}}:= \hat{\otimes}_{\Qp}^{\alpha \in \Delta}\Bt^{I_\alpha}$, $\B_{\Delta}^{\underline{I}}:= \hat{\otimes}_{\Qp}^{\alpha \in \Delta}\B^{I_\alpha}$, $\Bt_{\Kunder,\Delta}^{\underline{I}}:= \hat{\otimes}_{\Qp}^{\alpha \in \Delta}\Bt_{K_\alpha}^{I_\alpha}$ and $\B_{\Kunder,\Delta}^{\underline{I}}:= \hat{\otimes}_{\Qp}^{\alpha \in \Delta}\B_{K_\alpha}^{I_\alpha}$. 

If $\r = (r_\alpha)_{\alpha \in \Delta}$ and $\s = (s_\alpha)_{\alpha \in \Delta}$ are two sets of nonnegative real numbers indexed by $\Delta$, we say that $\r \leq \s$ if for all $\alpha \in \Delta$, we have $r_\alpha \leq s_\alpha$. We also write $\r > 0$ when $r_\alpha > 0$ for each $\alpha \in \Delta$. 

If $\underline{I} = \prod_{\alpha \in \Delta}I_\alpha$ is such that all the $I_\alpha$s are of the form $[r_\alpha,+\infty[$, we write $\Bt_{\rig,\Delta}^{\dagger,\r}$, $\B_{\rig,\Delta}^{\dagger,\r}$, $\Bt_{\rig,\Kunder,\Delta}^{\dagger,\r}$ and $\B_{\rig,\Kunder,\Delta}^{\dagger,\r}$ for $\Bt_{\Delta}^{\underline{I}}$, $\B_{\Delta}^{\underline{I}}$, $\Bt_{\Kunder,\Delta}^{\underline{I}}$ and $\B_{\Kunder,\Delta}^{\underline{I}}$ respectively.

We let $\Bt_{\rig,\Delta}^{\dagger}:=\bigcup_{\r > 0}\Bt_{\rig,\Delta}^{\dagger,\r}$, $\B_{\rig,\Delta}^{\dagger}:=\bigcup_{\r > 0}\B_{\rig,\Delta}^{\dagger,\r}$, $\Bt_{\rig,\Kunder,\Delta}^{\dagger}:=\bigcup_{\r > 0}\Bt_{\rig,\Kunder,\Delta}^{\dagger,\r}$ and $\B_{\rig,\Kunder,\Delta}^{\dagger}:=\bigcup_{\r > 0}\B_{\rig,\Kunder,\Delta}^{\dagger,\r}$.

\begin{lemm}
Let $\underline{I}=\prod_{\alpha \in \Delta}I_\alpha$. We have 
$$(\Bt_{\Delta}^{\underline{I}})^{H_{\Kunder}} = \Bt_{\Kunder,\Delta}^{\underline{I}},$$
and if for all $\alpha \in \Delta$, $r(K_\alpha) \leq \min(I_\alpha)$, then 
$$(\B_{\Delta}^{\underline{I}})^{H_{\Kunder}} = \B_{\Kunder,\Delta}^{\underline{I}}.$$
\end{lemm}
\begin{proof}
    This is lemma \ref{inv and tensors}.
\end{proof}

Our definitions of multivariable rings involve completing the (projective) tensor products of rings endowed with a norm, and it may not be completely clear in some cases that the seminorms induced on the noncompleted tensor products are actually norms. It may also not be clear that the injections for classical rings give rise to injections for the completed tensor products. The following results show that this is actually the case for the rings we consider:

\begin{lemm}
\label{lemma BdelI included in BdelJ}
If $\underline{J}= \prod_{\alpha \in \Delta}J_\alpha\subset \underline{I}=\prod_{\alpha \in \Delta}I_\alpha$ then the inclusion $\Bt^{I_\alpha} \subset \Bt^{J_\alpha}$ (resp. $\B^{I_\alpha} \subset \B^{J_\alpha}$, resp. $\Bt_{K_{\alpha}}^{I_\alpha} \subset \Bt_{K_{\alpha}}^{J_\alpha}$, resp. $\B_{K_{\alpha}}^{I_\alpha} \subset \B_{\Kunder}^{J_\alpha}$) gives rise to an inclusion $\Bt_{\Delta}^{\underline{I}} \subset \Bt_{\Delta}^{\underline{J}}$ (resp. $\B_{\Delta}^{\underline{I}} \subset \B_{\Delta}^{\underline{J}}$, resp. $\Bt_{\Kunder,\Delta}^{\underline{I}} \subset \Bt_{\Kunder,\Delta}^{\underline{J}}$, resp. $\B_{\Kunder,\Delta}^{\underline{I}} \subset \B_{\Kunder,\Delta}^{\underline{J}}$).
\end{lemm}
\begin{proof}
    This is a direct consequence of proposition \ref{prop tens Fréchet inj}.
\end{proof}

\begin{lemm}
\label{lemma BdelI Hausdorff}
    Let $\underline{I} = \prod_{\alpha \in \Delta}I_\alpha$. Then the topology on $\otimes_{\Qp}^{\alpha \in \Delta}\Bt^{I_\alpha}$ induced by the valuations $V(\cdot,I_\alpha)$ is separated.
\end{lemm}
\begin{proof}
    This follows from \cite[Proposition 17.4 ii.]{Schnei}.
    
    


\end{proof}

\begin{lemm}
\label{lemma Adagdeltar inj BdelI}
    Let $\r = (r_\alpha)_{\alpha \in \Delta}$ and let $\underline{I} = \prod_{\alpha \in \Delta}I_\alpha$ be such that for all $\alpha \in \Delta$, $r_\alpha \in I_\alpha$. Then the injections $\At^{\dagger,r_\alpha} \ra \Bt^{I_\alpha}$ induce an injective map $\At_\Delta^{\dagger,\r} \ra \Bt^{\underline{I}}$ which is bounded by $1$. 
\end{lemm}
\begin{proof}
    First note that because of lemma \ref{lemma BdelI included in BdelJ}, it suffices to prove the result in the case where $I_{\alpha} = [r_\alpha,r_\alpha]$ for all $\alpha \in \Delta$, and we thus assume now that this is the case. The rings $\At^{\dagger,r_\alpha}$, $\At^{I_\alpha}$ and $\Bt^{I_\alpha}$ are flat $\Zp$-algebras (because torsion free) so that the injections $\At^{\dagger,r_\alpha} \ra \At^{I_\alpha} \ra \Bt^{I_\alpha}$ give rise to injections
    $$\otimes_{\Zp}^{\alpha \in \Delta}\At^{\dagger,r_\alpha} \ra \otimes_{\Zp}^{\alpha \in \Delta}\At^{I_\alpha}\ra \otimes_{\Zp}^{\alpha \in \Delta}\Bt^{I_\alpha} = \otimes_{\Qp}^{\alpha \in \Delta}\Bt^{I_\alpha}.$$
    Because of the definition of the topology on the tensor product, and because the valuations $V(\cdot,I_\alpha)$ contain the valuations $v_{r_\alpha}$, the induced maps on the tensor products are bounded by $1$. Since the topology on $\otimes_{\Qp}^{\alpha \in \Delta}\Bt^{I_\alpha}$ is separated by lemma \ref{lemma BdelI Hausdorff}, so is the one on $\otimes_{\Zp}^{\alpha \in \Delta}\At^{\dagger,r_\alpha}$.
    
    We therefore get map $\At_\Delta^{\dagger,\r} \ra \Bt^{\underline{I}}$ which is bounded by $1$, and it remains to see that it is injective. Let us write $A$ for $\otimes_{\Zp}^{\alpha \in \Delta}\At^{\dagger,r_\alpha}$ endowed with the topology induced by the valuations $v_{r_\alpha}$, and $B = A[1/p] = A \otimes_{\Zp}\Qp$ equipped with the topology induced by this tensor product. By Proposition 4 of \cite[\S 2.1.7]{BoschAna}, we have an isometric isomorphism $\hat{B} \simeq \hat{A} \wotimes_{\Zp}\Qp$ so that $\hat{A} \ra \hat{B}$ is injective. To conclude, it suffices to go back to the definition of $B$ to notice that $\hat{B} \simeq \Bt^{\underline{I}}$ by using once again Proposition 4 of \cite[\S 2.1.7]{BoschAna} since the completion of $\At^{\dagger,r}[1/p]$ for $v_r$ is exactly $\Bt^{[r,r]}$.
\end{proof}

Our aim is now calculating the invariants for the actions of $H_{K,\Delta}$ on $\At_{\Delta}^{\dagger, \r}$ and $\A_{\Delta}^{\dagger, \r}$. We firstly need to establish some lemmas.
\begin{lemm}\label{topology on Atdagger}
    The topology induced by $V(\cdot,r)$ on $\At^{\dagger,r}$ coincides with the $\teich{\overline{\varpi}}$-adic topology (see also \cite[p. 6]{porat2022overconvergence}).
\end{lemm}
\begin{proof}
     It is clear that the $V(\cdot,r)$-topology coincides with the $([\varpibar],p)$-adic topology. The claim follows from the fact that $[\varpibar]$ divides $p^m$ in $\At^{\dagger,r}$, if $m\in\N$ is such that $V(p^m/[\u],r)>0$.
\end{proof}
\begin{lemm}\label{lemm regular principal ideal is closed}
    Let $A$ be a ring endowed with the $I$-adic topology, for $I$ a finitely generated ideal of $A$. Let $a$ be an element of $a$ such that $a$ is regular in $A$ and in $\hat{A}$. Then the ideal generated by $a$ is closed in $\hat{A}$ for the $I$-adic topology.   
\end{lemm}
\begin{proof}
    The ideal generated by $a$ in $A$ is a rank one free module over $A$, so that the map $a\hat{A} \rightarrow \hat{(aA)}$ is surjective (by \cite[\href{https://stacks.math.columbia.edu/tag/0315}{Lemma 0315}]{stacks-project}) and injective since $\hat{A}$ has no $a$-torsion. We conclude by using item $(2)$ of \cite[\href{https://stacks.math.columbia.edu/tag/0ARZ}{Lemma 0ARZ}]{stacks-project} which shows that $\hat{(aA)}$ is closed in $\hat{A}$. 
\end{proof}
\begin{lemm}
\label{lemma1 invar Atdagdelta}
    We have that
    $$(\At_{\Delta}^{\dagger, \r})^{H_{\Kunder}}=\At^{\dagger, r}_{\Kunder,\Delta},$$
    and
    $$(\A_{\Delta}^{\dagger, \r})^{H_{\Kunder}}=\A^{\dagger, r}_{\Kunder,\Delta},$$
\end{lemm}
Given these lemmas we can finally calculate the invariants.
\begin{proof}
    We only prove the statement for $\At_{\Delta}^{\dagger, \r}$, as the proof for $\A_{\Delta}^{\dagger, \r}$ is the same.
    We start by showing that \begin{equation*}(\At^{\dagger,r_{\alpha_1}}\wotimes_{\Zp}\dots\wotimes_{\Zp}\At^{\dagger,r_{\alpha_\delta}})^{H_K}=\At^{\dagger,r_{\alpha}}\wotimes_{\Zp}\dots\wotimes_{\Zp}(\At^{\dagger,r_{\alpha_\delta}})^{H_{K}}\end{equation*}
    We notice that
\begin{dmath*}\At^{\dagger,r_{\alpha_1}}\wotimes_{\Zp}\dots\wotimes_{\Zp}\At^{\dagger,r_{\alpha_\delta}}\cong\rm{lim}_{m\in\N}(\At^{\dagger,r_{\alpha_1}}/([\varpibar_{\alpha_1}])^m)\otimes_{\Z/p^{l_m}\Z}\dots\otimes_{\Z/p^{l_m}\Z}(\At^{\dagger,r_{\alpha_\delta}}/([\varpibar_{\alpha_\delta}])^m)\cong\rm{lim}_{m\in\N}(\At^{\dagger,r_{\alpha_1}}/([\varpibar_{\alpha_1}])^m)\otimes_{\Z/p^{l_m}\Z}\dots\otimes_{\Z/p^{l_m}\Z}(\At^{\dagger,r_{\alpha_{\delta-1}}}/([\varpibar_{\alpha_{\delta-1}}])^m)\wotimes_{\Zp}\At^{\dagger,r_{\delta}},
\end{dmath*}
    where $l_m$ is an integer such that $[\varpibar_{\alpha}]^m$ divides $p^{l_m}$ in $\At^{\dagger, r_{\alpha}}$, for $\alpha\in\Delta$ and $m\in\N$ (cfr. the proof of Lemma \ref{topology on Atdagger}).\\
    Since, for $\alpha\in\Delta$, $(\At^{\dagger,r_{\alpha}}/([\varpibar_{\alpha}])^m$ is a discrete, $p^{l_m}$-torsion $\Zp$-module, it is isomorphic to a direct sum $\oplus_{j\in J}\Z/p^j\Z$, for $j\le l_m$, by \cite{kaplansky2018infinite}, Theorem 22. Thus it suffices to show that 
    $$(\Z/p^m\Z)\wotimes_{\Zp}(\At^{\dagger,r_{\alpha_\delta}})^{H_K}=((\Z/p^m\Z)\wotimes_{\Zp}\At^{\dagger,r_{\alpha_\delta}})^{H_K},$$
    for $m\in\N$, but this follows from the fact that $p^m$ is a regular element invariant by $H_{K}$ in $\At^{\dagger,r_{\delta}}$, and that the ideal generated by $p^m$ is closed in $\At^{\dagger,r_{\alpha_\delta}}$ for the $[\varpibar_{\alpha_\delta}]$-adic topology by Lemma \ref{lemm regular principal ideal is closed},
   
    so that 
    $$(\Z/p^m\Z)\wotimes_{\Zp}\At^{\dagger,r_{\delta}}\simeq \At^{\dagger,r_{\delta}}/p^m.$$ 
    Applying this method inductively we can conclude.
    
\end{proof}

\begin{lemm}
    \label{lemma Adagdeltar inj Adagdelts}
Let $\r = (r_\alpha)_{\alpha \in \Delta}$ and $\s = (s_\alpha)_{\alpha \in \Delta}$ be two sets of nonnegative real numbers indexed by $\Delta$ such that $\r \leq \s$.
The natural injective map $\A^{\dagger,r_\alpha} \ra \A^{\dagger,s_\alpha}$ (resp. $\A_{K_\alpha}^{\dagger,r_\alpha} \ra \A_{K_\alpha}^{\dagger,s_\alpha}$, resp. $\At^{\dagger,r_\alpha} \ra \At^{\dagger,s_\alpha}$, resp. $\At_{K_\alpha}^{\dagger,r_\alpha} \ra \At_{K_\alpha}^{\dagger,s_\alpha}$) induces an injective map $\A_{\Delta}^{\dagger,\r} \ra \A_{\Delta}^{\dagger,\s}$ (resp. $\At_{\Kunder,\Delta}^{\dagger,\r} \ra \At_{\Kunder,\Delta}^{\dagger,\s}$, resp. $\At_{\Delta}^{\dagger,\r} \ra \At_{\Delta}^{\dagger,\s}$, resp. $\At_{\Kunder,\Delta}^{\dagger,\r} \ra \At_{\Kunder,\Delta}^{\dagger,\s}$).
\end{lemm}
\begin{proof}
We will just prove that statement for the map $\At_{\Delta}^{\dagger,\r} \ra \At_{\Delta}^{\dagger,\s}$. For the other maps the proof is analogous. 

First note that if for $\alpha \in \Delta$ we let $I_\alpha := [r_\alpha,s_\alpha]$ and $J_\alpha:=[s_\alpha,s_\alpha]$, we have an injection bounded by $1$: $\At_{\Delta}^{\dagger,\r} \subset \Bt^{\underline{I}}$, and an injection of $\Qp$-Banach spaces $\Bt^{\underline{I}} \subset \Bt^{\underline{J}}$ by lemmas \ref{lemma Adagdeltar inj BdelI}, \ref{lemma BdelI Hausdorff}, and \ref{lemma BdelI included in BdelJ}. To prove the statement, it remains to see that the image of $\At_{\Delta}^{\dagger,\r}$ by this composite map is contained in (the image in $\Bt^{\underline{J}}$ of) $\At_\Delta^{\dagger,\s}$. 

We let $\beta = \frac{\max_{\alpha \in \Delta}(s_\alpha)}{\min_{\alpha \in \Delta}(r_\alpha)}$. Let $(x_n)$ be a sequence of elements of $\otimes_{\Zp}^{\alpha \in \Delta}\At^{\dagger,r_\alpha}$, which converges in $\At_\Delta^{\dagger,\r}$ for $v_{\r}$. If $x_n = \sum_{i=1}^dx_{n,i}$, with $x_{n,i} = \otimes_{\alpha \in \Delta}x_{n,i,\alpha}$, we have by definition of the normed tensor product that for all $i$, $v_{\r}(x_{n,i}) \leq \beta v_{\s}(x_{n,i})$ using the fact that $rv_r \leq sv_s$ if $r \leq s$ (this is checked directly from the definition of $V(\cdot,r)$). Therefore, $v_{\r}(x_n) \leq \beta v_{\s}(x_n)$ (as the sup is taken over every possible way to write $x_n$ this way, and since the map $\otimes_{\Zp}^{\alpha \in \Delta}\At^{\dagger,r_\alpha} \ra \otimes_{\Zp}^{\alpha \in \Delta}\At^{\dagger,s_\alpha}$ is injective because each of the rings involved is a flat $\Zp$-algebra, each way of writing $x_n$ as such in $\otimes_{\Zp}^{\alpha \in \Delta}\At^{\dagger,r_\alpha}$ is a way to write $x_n$ as such in $\otimes_{\Zp}^{\alpha \in \Delta}\At^{\dagger,s_\alpha}$). This implies that the sequence $(x_n)$ converges in $\At_{\Delta}^{\dagger,\s}$.
\end{proof}

Let $\boldsymbol{\varpi}:=\prod_{\alpha \in \Delta}\varpi_\alpha$. We let $\A_{\Kunder,\Delta}^{\dagger}:=\bigcup_{\r > 0}\A_{\Kunder,\Delta}^{\dagger,\r}[1/\boldsymbol{\varpi}]$, $\At_{\Kunder,\Delta}^{\dagger}:=\bigcup_{\r > 0}\At_{\Kunder,\Delta}^{\dagger,\r}[1/\boldsymbol{\varpi}]$, $\A_{\Delta}^{\dagger}:=\bigcup_{\r > 0}\A_{\Delta}^{\dagger,\r}[1/\boldsymbol{\varpi}]$, $\At_{\Delta}^{\dagger}:=\bigcup_{\r > 0}\At_{\Delta}^{\dagger,\r}[1/\boldsymbol{\varpi}]$. 

We now define $\A_{\Kunder,\Delta}$ (resp. $\At_{\Kunder,\Delta}$ resp. $\A_{\Delta}$ resp. $\At_{\Delta}$) to be the $p$-adic completion of $\A_{\Kunder,\Delta}^\dagger$ (resp. $\At_{\Kunder,\Delta}^\dagger$ resp. $\A_{\Delta}^\dagger$ resp. $\At_{\Delta}^\dagger$). 

We let $\Bt_{\Kunder,\Delta} = \At_{\Kunder,\Delta}[1/p]$ and $\Bt_\Delta=\At_\Delta[1/p]$. As above, all these rings are naturally endowed with commuting actions of $\G_{\Kunder,\Delta}$ and $(\phi_\alpha)_{\alpha \in \Delta}$.

\begin{lemm}
\label{lemma dagger padicsep}
    The rings $\A_{\Kunder,\Delta}^{\dagger,\r},\At_{\Kunder,\Delta}^{\dagger,\r}, \A_{\Delta}^{\dagger,\r}$ and $\At_{\Delta}^{\dagger,\r}$ are $p$-adically separated.
\end{lemm}
\begin{proof}
    The arguments are the same in all four cases. By construction, those rings are separated for $v_{\r}$. If $x$ is in any of those rings, $x \neq 0$, then $v_{\r}(x)$ is bounded above. Since $v_{\r}(p^n)=n$, this implies that there exists $N \geq 0$ such that for $n \geq N$, $p^n$ does not divide $x$ (because by construction $v_{\r}$ is $\geq 0$ on those completed tensor products). This finishes the proof.
\end{proof}

\begin{coro}
Let $\r = (r_\alpha)_{\alpha \in \Delta}$. The natural injective map $\A^{\dagger,r_\alpha} \ra \A$ (resp. $\A_{K_\alpha}^{\dagger,r_\alpha} \ra \A_{K_\alpha}$ resp. $\At^{\dagger,r_\alpha} \ra \At$ resp. $\At_{K_\alpha}^{\dagger,r_\alpha} \ra \At_{K_\alpha}$) induce an injective map $\A_{\Delta}^{\dagger,\r} \ra \A_{\Delta}$ (resp. $\At_{\Kunder,\Delta}^{\dagger,\r} \ra \At_{\Kunder,\Delta}$ resp. $\At_{\Delta}^{\dagger,\r} \ra \At_{\Delta}$ resp. $\At_{\Kunder,\Delta}^{\dagger,\r} \ra \At_{\Kunder,\Delta}$).    
\end{coro}
\begin{proof}
    This follows from Lemma \ref{lemma Adagdeltar inj Adagdelts}.
\end{proof}

In \cite{PalZab21}, Pal and Zabradi define multivariable overconvergent rings and multivariable Robba rings as follows: first they define $\cal{R}_\Delta:= \bigcup_{\underline{\rho} \in (0,1)^\Delta}\cal{R}_\Delta^{(\underline{\rho},\underline{1})}$ as  the ascending union of the rings of multivariable power series
$$\cal{R}_\Delta^{(\underline{\rho},\underline{1})} := \left\{\sum_{\i \in \Z^\Delta}a_{\i}\mathbf{X}^{\underline{i}}, a_{\underline{i}} \in \Qp, \textrm{ convergent on } B^{(\underline{\rho},\underline{1})} \right\},$$
where $\mathbf{X}^{\underline{i}} = \prod_{\alpha \in \Delta}X_\alpha^{i_\alpha}$ and the polyannulus $B^{(\underline{\rho},\underline{1})}$ for the tuple $\underline{\rho} = (\rho_\alpha)_{\alpha \in \Delta} \in (0,1)^{\Delta}$ is defined as
$$B^{(\underline{\rho},\underline{1})}:=\left\{\mathbf{X} = (X_\alpha)_{\alpha \in \Delta} \in C^\Delta \textrm{ such that } \rho_\alpha < |X_\alpha|_p < 1 \textrm{ for all } \alpha \in \Delta   \right\}.$$

For each tuple $\underline{1} > \underline{\sigma} > \underline{\rho}$, we can define the $\underline{\sigma}$-norm on $\cal{R}_\Delta^{(\underline{\rho},\underline{1})}$ by $|\sum_{\i \in \Z^\Delta}a_{\i}\mathbf{X}^{\underline{i}}|_{\underline{\sigma}} = \sup_{\underline{i}}|a_{\underline{i}}|_p\prod_{\alpha \in \Delta}\sigma_\alpha^{i_\alpha}.$ 

Then they define 
$$\cal{E}_{\Delta}^\dagger := \{f \in \cal{R}_\Delta, \limsup_{\underline{\sigma} \to \underline{1}}|f|_{\underline{\sigma}} < \infty \}$$
and 
$$\O_{\cal{E}_\Delta}^\dagger:=\{f \in \cal{R}_\Delta, \limsup_{\underline{\sigma} \to \underline{1}}|f|_{\underline{\sigma}} \leq 1 \}.$$

Note that their $X_\alpha$ correspond to our $\varpi_\alpha$ and we won't make a distinction in what follows between the two, as the following result holds. 

\begin{lemm}
\label{lemma PalZab=ours Robba}
We have a natural isomorphism between $\cal{R}_\Delta$ and $\B_{\rig,\underline{\Qp},\Delta}^\dagger$.
\end{lemm}
\begin{proof}
     We have that $$\cal{R}^{(\rho,1)}=\rm{lim}_{n\in\N}\tate{\Qp}{\frac{x^{-1}}{\rho+1/n},\frac{x}{1-1/n}}=\rm{lim}_{n\in\N}\tate{\Qp}{\frac{y}{\rho+1/n},\frac{x}{1-1/n}}/(xy-1).$$ Thus 
     \tiny
     $$\cal{R}^{(\rho_{1},1)}\wotimes_{\Qp}\dots\wotimes_{\Qp}\cal{R}^{(\rho_{\delta},1)}\simeq\rm{lim}_{n\in\N}\tate{\Qp}{\frac{y_1}{\rho_{1}+1/n},\frac{x_1}{1-1/n}}/(x_1y_1-1)\wotimes_{\Qp}\dots\wotimes_{\Qp}\tate{\Qp}{\frac{y_{\delta}}{\rho_{\delta}+1/n},\frac{x_{\delta}}{1-1/n}}/(x_{\delta}y_{\delta}-1).$$
     \normalsize
     But since ideals in Tate algebras are closed we get
     $$\cal{R}^{(\rho_{1},1)}\wotimes_{\Qp}\dots\wotimes_{\Qp}\cal{R}^{(\rho_{\delta},1)}\simeq\rm{lim}_{n\in\N}\tate{\Qp}{\frac{x_1^{-1}}{\rho+1/n},\frac{x_{1}}{1-1/n},\dots,\frac{x_{\delta}^{-1}}{\rho+1/n},\frac{x_{\delta}}{1-1/n}}$$
     that is $\cal{R}_{\Delta}^{(\underline{\rho},\mathbf{1})}$. From this the claim follows.
\end{proof}
\begin{lemm}\label{lemm power serie description of Atdagger}
    We have a canonical isomorphism
    \begin{equation}\label{eqn power serie description of Atdagger}\A_{\Qp,\Delta}^{\dagger,\underline{r}}\cong \cal{S},\end{equation}
    where $\cal{S}$ is the set 
    $$\left\{f=\sum_{\underline{i}\in\Z^{\Delta}}a_{\underline{i}}\mathbf{X}^{\underline{i}}:a_{\underline{i}}\in\Zp,f \ \rm{converges} \ \rm{in} \ B^{[\underline{\rho},\underline{1})}, |f|_{\underline{\rho}}\le 1\right\}.$$
\end{lemm}
\begin{proof}
    The one variable case is \cite[Proposition II.2.1]{cherbonnier1998representations}.\\
    We prove the multivariable case using the universal property of the completed tensor product (see \cite[ Proposition 2.1.7]{BoschAna}).
    Note that we have a bounded map
    $$\A_{\Qp,\Delta}^{\dagger,r_1}\times\dots\times\A_{\Qp,\Delta}^{\dagger,r_{\delta}}\xrightarrow{\phi}\cal{S},$$
    given by
    $$\phi:(\sum_{i_{1}}a_{i_1}X_1^{i_1}\dots, \sum_{i_{\delta}}a_{i_{\delta}}X_{\delta}^{i_\delta} )\mapsto \sum_{\underline{i}\in\Z^{\Delta}}a_{i_1}\dots a_{i_{\delta}}\mathbf{X}^{\underline{i}}.$$ Note that such a map is bounded because $|a_{i_1}\dots a_{i_{\delta}}\mathbf{X}^{\underline{i}}|_{\underline{\rho}}=|a_{i_1}X_1^{i_1}|_{\rho_1}\dots |a_{i_{\delta}}X_1^{i_{\delta}}|_{\rho_{\delta}}.$
    Let $M$ be a Banach module over $\Zp$ and let 
    $$\A_{\Qp,\Delta}^{\dagger,r_1}\times\dots\times\A_{\Qp,\Delta}^{\dagger,r_{\delta}}\xrightarrow{\psi} M$$ be a bounded multilinear map. Then $\psi$ factors uniquely through $\phi$ as $\psi'\circ\phi$,
    where 
    $$\psi':\sum_{\underline{i}\in\Z^{\Delta}}a_{\underline{i}}\mathbf{X}^{\underline{i}}\mapsto \sum_{\underline{i}\in\Z^{\Delta}}a_{\underline{i}}\psi(X_1^{i_1},\dots X_{\delta}^{i_{\delta}}).$$
    note that $\psi'$ is well defined and bounded because, since $\psi$ is bounded, there exists a constant $C\in\R$ such that $$\norm{a_{\underline{i}}\psi(X_1^{i_1},\dots X_{\delta}^{i_{\delta}})}_M\le C|a_{}|_{\Zp}|X_1^{i_1}|_{\rho_1}\dots |X_{\delta}^{i_{\delta}}|_{\rho_{\delta}}=C|a_{\underline{i}}\mathbf{X}^{\underline{i}}|_{\underline{\rho}},$$
    for each $\underline{i}\in\Z^{\Delta}$.
    Thus we obtain \eqref{eqn power serie description of Atdagger}.
\end{proof}
\begin{prop}
\label{prop rings of PalZab=ours}
    Through the identification between $\cal{R}_\Delta$ and $\B_{\rig,\underline{\Qp},\Delta}^\dagger$ given by Lemma \ref{lemma PalZab=ours Robba}, we have $\O_{\cal{E}_\Delta}^\dagger = \A_{\underline{\Qp},\Delta}^\dagger$ and $\O_{\cal{E}_\Delta} = \A_{\underline{\Qp},\Delta}$. 
\end{prop}
\begin{proof}
    Proposition 3.1.2 of \cite{PalZab21} shows that the $p$-adic completion of $\O_{\cal{E}_\Delta}^\dagger$ is isomorphic to $\O_{\cal{E}_\Delta}$ so that with our definition of $\A_{\underline{\Qp},\Delta}$ as the $p$-adic completion of $\A_{\underline{\Qp},\Delta}^\dagger$ it suffices to prove that $\O_{\cal{E}_\Delta}^\dagger = \A_{\underline{\Qp},\Delta}^\dagger$.

    Since we have an inclusion $\A_{\underline{\Qp},\Delta}^{\dagger,\r} \subset \B_{\rig,\underline{\Qp},\Delta}^{\dagger,\r}$  for any $\r \geq 0$ (this follows from Lemma \ref{lemm power serie description of Atdagger}), we obtain an injection (through the identification between $\cal{R}_\Delta$ and $\B_{\rig,\underline{\Qp},\Delta}^\dagger$) $\A_{\underline{\Qp},\Delta}^{\dagger,\r}[1/\boldsymbol{\varpi}] \ra \cal{R}_{\Delta}^{(\underline{\rho}(\r),\underline{1})}$. Moreover, we have $v_{\r}(a_{\i}\boldsymbol{\varpi}^{\i}) = v_p(a_{\i}) + \sum_{\alpha \in \Delta}i_\alpha v_{r_\alpha}(\varpi_\alpha)$ 
    so that if $x \in \A_{\underline{\Qp},\Delta}^{\dagger,\r}$ then $|x|_{\boldsymbol{\rho}} \leq 1$ for $\boldsymbol{\rho} = (p^{-v_{r_\alpha}(\varpi_\alpha)})_{\alpha \in \Delta}$.
    Using the fact that $\A_{\underline{\Qp},\Delta}^{\dagger,\r} \subset \A_{\underline{\Qp},\Delta}^{\dagger,\s}$ for $\s \geq \r$ by lemma \ref{lemma Adagdeltar inj Adagdelts}, and using the fact that for any $k \geq 0$, $\limsup_{\underline{\sigma} \to \underline{1}}|\boldsymbol{\varpi}^k|_{\underline{\sigma}} = 1$, we obtain that if $x \in \A_{\underline{\Qp},\Delta}^\dagger$ then its image in $\cal{R}_\Delta$ is in $\O_{\cal{E}_\Delta}^\dagger$.

    It remains to see that any element of $\O_{\cal{E}_\Delta}^\dagger$ can be written as an element of $\A_{\underline{\Qp},\Delta}^{\dagger,\r}[1/\boldsymbol{\varpi}]$ for some $\r \geq 0$. We fix an order on $\Delta$, and for $\underline{i} \in \Z^\Delta$, we let $h(\underline{i}) \in \Delta$ denote the smallest $\alpha \in \Delta$ such that $\min_{\alpha \in \Delta}i_\alpha = i_{h(\i)}$. We also let $j(\underline{i}) = \min_{\alpha \in \Delta}i_\alpha$. 
    Let $x \in \O_{\cal{E}_\Delta}^\dagger, x = \sum_{\underline{i} \in \Z^\Delta}a_{\underline{i}}\boldsymbol{\varpi}^{\underline{i}}$. Recall that by lemma 3.1.1 of \cite{PalZab21}, this implies that all the $a_{\i}$ belong to $\Zp$. 
    
    By definition of $\O_{\cal{E}_\Delta}^\dagger$, if $C > 0$ then there exists $\epsilon > 0$ such that for all $\underline{\rho} \in (1-\epsilon,1)^\Delta$, we have $x \in \cal{R}_\Delta^{(\underline{\rho},\underline{1})}$ and $|x|_{\underline{\rho}} < p^C$. Let $\underline{\rho}=(\rho)^\Delta \in (1-\epsilon,1)^\Delta$ and let $r > 0$ be such that $\rho = p^{-1/r}$. Since $\underline{\rho} \in (1-\epsilon,1)^\Delta$, we have that for all $\underline{i} \in \Z^\Delta$, $v(a_{\underline{i}})+\sum_{\alpha \in \Delta}i_\alpha r > -C$. We write $x = x^+ + \sum_{\alpha \in \Delta}x_\alpha$ where 
    $$x^+ = \sum_{\underline{i} \in \Z^\Delta, \underline{i} \geq 0}a_{\underline{i}}\boldsymbol{\varpi}^{\underline{i}}$$
    and
    $$x_\alpha = \sum_{\underline{i} \in \Z^\Delta, \underline{i} \not \geq 0, h(\underline{i}) =\alpha}a_{\underline{i}}\boldsymbol{\varpi}^{\underline{i}}.$$
    The fact that for all $\underline{i} \in \Z^\Delta$, $v(a_{\underline{i}})+\sum_{\alpha \in \Delta}i_\alpha r > -C$ translates into the fact that for all $\underline{i} \in \Z^\Delta$, $v(a_{\i})+\sum_{\alpha \in \Delta}i_\alpha r > -C$. If $\underline{i} \not \geq 0$, this translates into the fact that $v(a_{\underline{i}})+j(\underline{i})(\delta r) > -C$ (since $j(\underline{i})$ is negative in this case). In particular, if $\underline{k}$ is such that $v_{\delta r}(\varpi_\alpha^{k_\alpha}) > C$ then the $a_{\underline{i}}\boldsymbol{\varpi}^{\underline{i}+\underline{k}}$ belong to $\otimes_{\Zp}^{\alpha \in \Delta}\A_{\underline{\Qp}}^{\dagger,\delta r}$ for all $\underline{i} \in \Z^\Delta$. Moreover, since $x \in \cal{R}_\Delta^{(\underline{\rho},\underline{1})}$, this implies that the $a_{\underline{i}}\boldsymbol{\varpi}^{\underline{i}}$ go to zero for $v_{\underline{r}}$ (and thus for $v_{\delta\underline{r}}$) for the cofinite filter. Therefore, $x^+$ and each of the $x_\alpha$ belong to $\boldsymbol{\varpi}^{-\underline{k}}\A_{\underline{\Qp},\Delta}^{\dagger,\delta\underline{r}}$, and thus we have $x \in \boldsymbol{\varpi}^{-\underline{k}}\A_{\underline{\Qp},\Delta}^{\dagger,\delta\underline{r}}$, which is what we wanted.
\end{proof}

In particular, this shows that all our definition are compatible with those of Pal and Zabradi in \cite{PalZab21} when $\Kunder = \underline{\Qp}$ and naturally extend to any $\Kunder$ where each $K_\alpha$ is a finite extension of $\Qp$. Moreover, it follows directly from proposition \ref{prop rings of PalZab=ours} and from the definitions of the objects involved that our ring $\At_{\Delta}^\dagger$ contains the ring $\O_{\hat{\cal{E}_{\Delta}^{\mathrm{ur}}}}^\dagger$ defined in \S 3.2 of \cite{PalZab21}. 

Finally, in \cite[\S 2.2]{KedCarZab}, rings $\tilde{\O}_{\mathcal{E}_{\Delta}}$, $\O_{\mathcal{E}_{\Delta}}$, $\tilde{\O}_{\mathcal{E}_{\Delta}}^\dagger$ and $\O_{\mathcal{E}_{\Delta}}^{\dagger}$ are constructed. The definitions of those rings differ from the ones of \cite{PalZab21}, so we now prove that those rings coincide respectively with our rings $\At_{\Kunder,\Delta}, \A_{\Kunder,\Delta}, \At_{\Kunder,\Delta}^{\dagger}$ and $\A_{\Kunder,\Delta}^\dagger$ (in particular, this statement alongside proposition \ref{prop rings of PalZab=ours} shows that their rings and constructions do extend the ones of \cite{ZabMulti0} and \cite{PalZab21}).

\begin{prop}
\label{prop rings of KedCarZab=ours}
    The rings $\tilde{\O}_{\mathcal{E}_{\Delta}}$, $\O_{\mathcal{E}_{\Delta}}$, $\tilde{\O}_{\mathcal{E}_{\Delta}}^\dagger$ and $\O_{\mathcal{E}_{\Delta}}^{\dagger}$ defined in \cite[\S 2.2]{KedCarZab} coincide respectively with our rings $\At_{\Kunder,\Delta}, \A_{\Kunder,\Delta}, \At_{\Kunder,\Delta}^{\dagger}$ and $\A_{\Kunder,\Delta}^\dagger$. 
\end{prop}
\begin{proof}
    In \cite[\S 2.2]{KedCarZab}, Carter, Kedlaya and Zabradi show that their rings $\tilde{\O}_{\cal{E}_\Delta}$ and $\O_{\cal{E}_\Delta}$ satisfy the following:
    $$\tilde{\O}_{\cal{E}_\Delta} = \varprojlim_{m}(\At_{K_{\alpha_1}}/p^m\At_{K_{\alpha_1}})\hat{\otimes}_{\Zp}\cdots \wotimes_{\Zp}(\At_{K_{\alpha_\delta}}/p^m\At_{K_{\alpha_\delta}})$$
    and 
    $$\O_{\cal{E}_\Delta} = \varprojlim_{m}(\A_{K_{\alpha_1}}/p^m\A_{K_{\alpha_1}})\hat{\otimes}_{\Zp}\cdots \wotimes_{\Zp}(\A_{K_{\alpha_\delta}}/p^m\A_{K_{\alpha_\delta}})$$
    where the completed tensor products are taken for the $(\varpi_{\alpha})_{\alpha \in \Delta}$-adic topology.
    
    In order to prove that our ring $\At_{\Kunder,\Delta}$ (resp. $\A_{\Kunder,\Delta}$) is isomorphic to their ring $\tilde{\O}_{\cal{E}_\Delta}$ (resp. $\O_{\cal{E}_\Delta}$), it suffices to prove that for $m \geq 1$ and $\r \geq \underline{r_0}$, $(\At^{\dagger,\r}_{\Kunder,\Delta}[1/\boldsymbol{\varpi}])/p^m$ is isomorphic to $(\At_{K_{\alpha_1}}/p^m\At_{K_{\alpha_1}})\hat{\otimes}_{\Zp}\cdots \wotimes_{\Zp}(\At_{K_{\alpha_\delta}}/p^m\At_{K_{\alpha_\delta}})$ (the statement and the proof are the same in the non tilde case). First note that for any $m \geq 1$ and for any $r \geq \frac{p-1}{p}$, $\At_K^{\dagger,r}[1/\varpi]/p^m \simeq \At_K/p^m$. 

    Let us prove first that $(\At^{\dagger,\r}_{\Kunder,\Delta}[1/\boldsymbol{\varpi}])/p^m$ is isomorphic to the $I$-adic completion of $(\At_{K_{\alpha_1}}^{\dagger,r_1}/p^m\otimes_{\Zp} \cdots \otimes_{\Zp} \At_{K_{\alpha_\delta}}^{\dagger,r_\delta}/p^m)$, where $I = (\varpi_{\alpha})_{\alpha \in \Delta}$. It suffices to prove that the ideal generated by $p^m$ is closed in $\At^{\dagger,\r}_{\Kunder,\Delta}$, which follows from Lemma \ref{lemm regular principal ideal is closed}.

    Let $r_1,\ldots,r_\delta$ such that for all $\alpha \in \Delta$, $r_\alpha \geq \frac{p-1}{p}$ (this is to ensure that $\frac{\varpi_\alpha}{[\varpi_\alpha]}$ is a unit in $\At^{\dagger,r_\alpha}$) and let $m \geq 1$. By \cite[\S 2.1.7, Prop. 4]{BoschAna}, we have that the $v_{\r}$-completion of $(\At_{K_{\alpha_1}}^{\dagger,r_1}/p^m\otimes_{\Zp} \cdots \otimes_{\Zp} \At_{K_{\alpha_\delta}}^{\dagger,r_\delta}/p^m)$ is isomorphic to $(\At_{K_{\alpha_1}}^{\dagger,r_1}/p^m\wotimes_{\Zp} \cdots \wotimes_{\Zp} \At_{K_{\alpha_\delta}}^{\dagger,r_\delta}/p^m)$, where the completion is taken with respect to the induced valuations $v_{r_\alpha}$ on $\At_{K_{\alpha}}^{\dagger,r_\alpha}/p^m$ (note that we used that $(\At_{K_{\alpha_1}}^{\dagger,r_1}/p^m\otimes_{\Zp} \cdots \otimes_{\Zp} \At_{K_{\alpha_\delta}}^{\dagger,r_\delta}/p^m) \simeq (\At_{K_{\alpha_1}}^{\dagger,r_1}\otimes_{\Zp} \cdots \otimes_{\Zp} \At_{K_{\alpha_\delta}}^{\dagger,r_\delta})/p^m$ using for example \cite[Exercise 1.3, Chapter 1]{qing2006algebraic}). 

    Since localizations commute with quotients, it remains to check that the $(\varpi_\alpha)_{\alpha \in \Delta}$-adic topology on $(\At_{K_{\alpha_1}}^{\dagger,r_1}/p^m\otimes \cdots \otimes \At_{K_{\alpha_\delta}}^{\dagger,r_\delta}/p^m)$ coincides with the $v_{\r}$-topology, which is straightforward as the denominators are bounded below. 

    For the rings $\tilde{\O}_{\mathcal{E}_{\Delta}}^\dagger$ and $\O_{\mathcal{E}_{\Delta}}^{\dagger}$, we proceed as in the proof of Proposition \ref{prop rings of PalZab=ours}: thanks to the identification between $\tilde{\O}_{\cal{E}_\Delta}$ and $\At_{\Kunder,\Delta}$, we check that the condition defining $\tilde{\O}_{\cal{E}_\Delta}^\dagger$ in \cite[Notation 2.10]{KedCarZab} implies that if $x \in \tilde{\O}_{\cal{E}_\Delta}^\dagger$, then it belongs to $\At_{\Kunder,\Delta}^{\dagger}$ (this is exactly the same proof as in the proof of Proposition \ref{prop rings of PalZab=ours}). For the converse, we check that an element in $\At_{\Kunder,\Delta}^{\dagger,\r}$ defined as a limit of elements of the tensor product $\At_{K_{\alpha_1}}^{\dagger,r_{\alpha_1}} \otimes_{\Zp} \cdots \otimes_{\Zp} \At_{K_{\alpha_\delta}}^{\dagger,r_{\alpha_\delta}}$ that converges for $v_{\r}$ defines an element of $\tilde{\O}_{\cal{E}_{\Delta}}$ with bounded $\rho$-norm in the sense of \cite[Notation 2.6]{KedCarZab} for some $\rho > 0$, and thus defines an element of $\tilde{\O}_{\cal{E}_{\Delta}}^\dagger$.
\end{proof}

\begin{rema}\label{rmk AtDelta as Witt vectors}
Let $C^{\flat}_{\Delta}$ be $\O_{C,\Delta}^{\flat}[1/\overline{\boldsymbol{\varpi}}]$. We have that $\At_{\Delta}$ is isomorphic to $W\left( C^{\flat}_{\Delta}\right)$, as they are strict $p$-rings with the same residue perfect algebra (see \cite[Theorem 1.1.8]{kedlaya2015new}). This is because 
$$\At_{\Delta}/p\simeq \At_{\Delta}^{\dagger}[1/\boldsymbol{\varpi}]/p\simeq C^{\flat}_{\Delta}.$$
In particular every element in $\At_{\Delta}$ can be written as 
a sum
$$\sum_{i\in\N}\teich{c_i}p^i,$$
where $c_i\in C^{\flat}_{\Delta}$ (see also \cite[p. 1337]{KedCarZab}). Note that if $|\cdot|'$ is the norm on $C^{\flat}_{\Delta}$, if $|c_i|'$ is bounded in $i$, then 
$$\sum_{i\in\N}\teich{c_i}p^i\in \At_{\Delta}^{\dagger}.$$
Moreover we have that if $\phi_{\Delta}$ is the Frobenius on $C^{\flat}_{\Delta}$ satisfies $|\phi(\cdot)|'=(|\cdot|')^{p}$ by definition. 
\end{rema}

Following \cite{ZabMulti0} and \cite{KedCarZab}, we make the following definition:

\begin{defi}
Let $A$ be any of $\A_{\Kunder,\Delta}, \A_{\Kunder,\Delta}^\dagger, \At_{\Kunder,\Delta}, \At_{\Kunder,\Delta}^\dagger$ and $B=A[1/p]$.

A $(\phi_{\Delta},\Gamma_{\Kunder,\Delta})$-module $\D$ on $A$ (resp. $B$) is a finitely presented $A$-module endowed with semilinear actions of $\Gamma_{\Kunder,\Delta}$ and $(\phi_\alpha)_{\alpha \in \Delta}$ that commute with one another.   

It is said to be étale if the maps $1 \otimes_{\Zp} \phi_\alpha: \phi_\alpha^*\D \to \D$ are isomorphisms for $\alpha \in \Delta$ (resp. if it is of the form $\D_0[1/p]$ for some projective finitely generated étale $(\phi,\Gamma_{\Kunder,\Delta})$-module $\D_0$ over $A$).
\end{defi}

\begin{rema}
\label{remark CarKedZab bounded action}
    Contrary to what happens in the classical setting, the étaleness condition in the multivariable setting is also related to the $\Gamma_{\Kunder,\Delta}$-action. Thanks to \cite[Theorem 6.19]{KedCarZab} (and Remark 6.20 of ibid), this condition can be relaxed by asking that the action of $\Gamma_{\Kunder,\Delta}$ is bounded, in the sense that the action of $\Gamma_{\Kunder,\Delta}$ carries $\D_0$ in the previous definition into $p^{-m}\D_0$ for some nonnegative integer $m$, when $A$ is either $\A_{\Kunder,\Delta}$ or $\A_{\Kunder,\Delta}^\dagger$.
\end{rema}

The following is \cite[Thm. 6.15 and 6.16]{KedCarZab} in the case where the $K_\alpha$ are all equal, and Theorem 8.2 and 8.3 of ibid. when the $K_\alpha$ are any finite extensions of $\Qp$:
\begin{theo}
\label{theo overconvergence PalZab and KedCarZab}
    The category of continuous representations of $\G_{\Kunder,\Delta}$ on finite free $\Zp$-modules (resp. dimensional $\Qp$-vector spaces) is equivalent to the category of projective étale $(\phi_\Delta,\Gamma_{\Kunder,\Delta})$-modules over each of the rings $\A_{\Kunder,\Delta}, \A_{\Kunder,\Delta}^\dagger, \At_{\Kunder,\Delta}, \At_{\Kunder,\Delta}^\dagger$ (resp. $\B_{\Kunder,\Delta}, \B_{\Kunder,\Delta}^\dagger, \Bt_{\Kunder,\Delta}, \Bt_{\Kunder,\Delta}^\dagger$).
\end{theo}

\subsection{Multivariable crystalline and semistable rings}
Recall (\cite{fontaine1994corps}) that classically the ring $\A_{\crys}$ is defined as the $p$-adic completion of the divided power envelope of $\A_\inf$ with respect to $\ker(\theta)$ and then one defines $\B_{\crys}^+ := \A_\crys^+[1/p]$ and $\B_\crys:=\B_\crys^+[1/t]$ (as $t$ belongs to $\B_\crys$). One could thus mimick this construction in the multivariable setting, by taking the $p$-adic completion of the divided power envelope of $\A_{\inf,\Delta}$ with respect to $\ker(\theta_\Delta)$. Here however we decide to follow our philosophy behind the constructions of the various multivariable rings of periods we have defined. First, recall \cite[\S III.2]{Col98} that one can replace the classical ring $\B_\crys^+$ by the ring $\B_\max^+ = \Bt^{[0,r_0]}$ and $\B_\crys$ by $\B_\max := \B_\max^+[1/t]$, which does not change anything for the study of $p$-adic representations because of the fact that $\phi(\B_\max) \subset \B_\crys \subset \B_\max$. One of the main reason behind this change is that the topology on $\B_\max$ is much nicer than the one on $\B_\crys$ (namely, as explain in section III.2 of \cite{Col98}, the topology on $\B_\crys^+$ induced by the one on $\B_\crys$ is not the natural topology on $\B_\crys^+$). Moreover, $\B_\max^+$ is a $p$-adic Banach space and thus is completely adapted to our point of view and constructions. Recall also that $\cap_{n\in\N} \phi^n(\B_\max^+) = \cap_{n\in\N} \phi^n(\B_\crys^+) = \Btrigplus$ \cite[\S III.2]{Col98}. 

We define 
\[\B_{\max,\Delta}^+:=\B_\max^+ \wotimes_{\Qp} \cdots \wotimes_{\Qp} \B_\max^+\]
as the $p$-adic completion of the tensor product of $\delta$ copies of $\B_\max^+$ (note that this is the same as taking the completed tensor product in the category of $p$-adic Banach spaces). For $\alpha \in \Delta$ we let 
\[t_\alpha = 1 \otimes_{\Qp} \cdots \otimes_{\Qp} 1 \otimes_{\Qp} t \otimes_{\Qp} 1 \otimes_{\Qp} \cdots \otimes_{\Qp} 1 \in \B_\max^+ \otimes_{\Qp} \cdots \otimes_{\Qp} \B_\max^+\]
where $t$ is the factor of index $\alpha$, and we stille denote by $t_\alpha$ its image in $\B_{\max,\Delta}^+$. We let $t_\Delta = \prod_{\alpha \in \Delta}t_\alpha \in \B_{\max,\Delta}^+$ and we let $\B_{\max,\Delta} = \B_{\max,\Delta}^+[1/t_\Delta]$. 

Recall that we have defined Fréchet rings $\Bt_{\rig,\Delta}^{\dagger,\r}$, $\Bt_{\rig,\Kunder,\Delta}^{\dagger,\r}$ and $\B_{\rig,\Kunder,\Delta}^{\dagger,\r}$, and LF rings $\Bt_{\rig,\Delta}^{\dagger}$, $\Bt_{\rig,\Kunder,\Delta}^{\dagger}$ and $\B_{\rig,\Kunder,\Delta}^{\dagger}$.

Recall \cite[\S 2.4]{Ber02} that classically one defines rings $\Bt_{\log}^{\dagger,r}$ by $\Bt_{\log}^{\dagger,r}:=\Bt_{\rig}^{\dagger,r}[Y]$ and $\Bt_\log^+$ by $\Bt_{\log}^+:=\Bt_{\rig}^+[Y]$, endowed with an action of $\G_K$ and of the Frobenius given by
\[g(Y) = Y+\log([\frac{g(\tilde{p})}{\tilde{p}}]) \quad \textrm{ and } \phi(Y)=pY\]
using the fact that if $x \in \O_C^\flat$ is such that $v(x^{(0)}-1) \geq 1$ then the series 
$$\log[x]:= \sum_{n \geq 1}(-1)^{n-1}\frac{([x]-1)^n}{n}$$
converges in $\Btrigplus$ (cf proposition 2.23 of \cite{Ber02}). Intuitively, one has to think of $Y$ as $\log[\tilde{p}]$, and this intuition can be made explicit as follows: if we fix the $p$-adic logarithm by putting $\log p=0$, then the power series 
\[\log([\tilde{p}])= \log(\frac{[\tilde{p}]}{p}) = -\sum_{n \geq 1}\frac{(1-[\tilde{p}]/p)^n}{n}\]
converges in $\Bdrplus$ and so we get an injective map $\iota_n : \Bt_{\rig}^{\dagger,r_n}[Y] \ra \Bdrplus$ extending $\iota_n : \Bt_{\rig}^{\dagger,r_n} \ra \Bdrplus$ by $\iota_n(Y) = p^{-n}\log[\tilde{p}]$ (cf proposition 2.25 of \cite{Ber02}). 

Mimicking this construction, for $\alpha \in \Delta$ we let $Y_\alpha$ be a variable, and we define 
\begin{equation*}\Bt^{\dagger}_{\mathrm{log},\Delta}=\mathbf{B}^{\dagger}_{\mathrm{rig},\underline{K}}[Y_\alpha]_{\alpha \in \Delta}\end{equation*}
and
\begin{equation*}\mathbf{B}^{\dagger}_{\mathrm{log},\Delta,\underline{K}}=\mathbf{B}^{\dagger}_{\mathrm{rig},\underline{K}}[Y_\alpha]_{\alpha \in \Delta}.\end{equation*}
endowed with an action of $\G_{\Kunder,\Delta}$ and $(\phi_\alpha)_{\alpha \in \Delta}$ given by
\[g(Y_\alpha) = Y_\alpha + \log([\frac{g(\tilde{p}_\alpha)}{\tilde{p}_\alpha}] \quad \textrm{ and } \phi_\beta(Y_\alpha) = p^{\delta_{\alpha\beta}}Y_\alpha\]
where the series defining $\log([\frac{g(\tilde{p}_\alpha)}{\tilde{p}_\alpha}]$ converges in $\Bt_{\rig,\alpha}^+$. 

Once again we have injective maps $\iota_{\n}: \Bt_{\log,\Delta}^{\dagger,\r_{\n}} \ra \BdrplusDel$ sending $Y_\alpha$ to $p^{-n_\alpha}\log([\tilde{p_\alpha}])$ for $\alpha \in \Delta$. 

Finally, we also define $\Bt_{\log,\Delta}^+:=\Bt_{\rig,\Delta}^+[Y_\alpha]_{\alpha \in \Delta}$, $\B_{\st,\Delta}^+:=\B_{\max,\Delta}^+[Y_\alpha]_{\alpha \in \Delta}$ and $\B_{\st,\Delta}:=\B_{\st,\Delta}^+[1/t_\Delta]$. 

For $V$ a $p$-adic representation of $\G_{\Kunder,\Delta}$, we say that $V$ is crystalline if the map
$$(V\otimes_{\Qp}\B_{\max,\Delta})^{\G_{\Kunder,\Delta}}\wotimes_{K_{\Delta}}\B_{\max,\Delta}\to V\wotimes_{\Qp} \B_{\max,\Delta},$$
induced by the inclusion
$$(V\otimes_{\Qp}\B_{\max,\Delta})^{\G_{\Kunder,\Delta}}\subset V\wotimes_{\Qp} \B_{\max,\Delta}$$
is an isomorphism, and is semistable if the map
$$(V\otimes_{\Qp}\B_{\st,\Delta})^{\G_{\Kunder,\Delta}}\wotimes_{K_{\Delta}}\B_{\st,\Delta}\to V\wotimes_{\Qp} \B_{\st,\Delta},$$
induced by the inclusion
$$(V\otimes_{\Qp}\B_{\st,\Delta})^{\G_{\Kunder,\Delta}}\subset V\wotimes_{\Qp} \B_{\st,\Delta}$$
is an isomorphism.

Note that we have the following result:

\begin{pro}\label{rigplus in rig}
    The morphisms 
    \begin{equation*}
        \BtlogplusDel\rightarrow\BtlogDel
    \end{equation*}
    and
    \begin{equation*}
        \BtlogplusDel\rightarrow \B_{\st,\Delta}^+
    \end{equation*}
   induced by the inclusion factor by factor is a continous equivariant injection.
\end{pro}
\begin{proof}
    This follows easily from Proposition 1.1.26 of \cite{Emer}.
\end{proof}
Let $K$ be a finite extension over $\Qp$ and let $F$ be its maximal unramified subextension.\\
Using Lemma \ref{inv and tensors} and the fact that \begin{equation*}F=(\Btrigplus)^{\G_{K}}\subset (\B_{\st})^{\G_{K}}=F\end{equation*} the following corollary is immediate.
\begin{cor}
    We have \begin{equation*}(\BtlogplusDel)^{\G_{\underline{K}}}=(\B_{\st,\Delta})^{\G_{\underline{K}}}=F_1\otimes_{\Qp}\dots\otimes_{\Qp}F_{\delta}.\end{equation*}
\end{cor}

\section{Multivariable Tate-Sen descent}

\subsection{The Tate-Sen formalism}
Let $\delta \geq 1$ be an integer, $\Delta = \{1,\ldots,\delta\}$ and for each $i \in \Delta$, $G_i$ is a profinite group, admitting a continuous character $\chi_i : G_i \ra \Z_p^\times$ with open image and kernel $H_i$. If $G_i'$ is an open subgroup of $G_i$, and if $H_i' = G_i' \cap H_i$, then we let $G_{H_i'}$ denote the normalizer of $H_i'$ in $G_i$. We also let $\tilde{\Gamma}_{H_i'} = G_{H_i'}/H_i'$, and we let $C_{H_i'}$ denote the center of $\tilde{\Gamma}_{H_i'}$. By \cite[Lemm. 3.1.1]{BC08} $C_{H_i'}$ is an open subgroup of $\tilde{\Gamma}_{H_i'}$. We let $n_1(H_i')$ denote the smallest integer $n \geq 1$ such that $\chi_i(C_{H_i'})$ contains $1+p^n\Zp$, which is therefore $< +\infty$. When $\Delta = \{1\}$, we omit the subscript in the notations. We let $H = H_1 \times \ldots \times H_\delta$, and to avoid additional notations, we also denote by $H_i$ the subgroup of $H$ generated by the elements $(\id_{H_1},\ldots,h,\ldots,\id_{H_\delta})$ where $h \in H_i$ and the other components are the identity of $H_j$, $i \neq j$.

Let $S$ be a Banach algebra over $\Qp$ and let $\tilde{\Lambda}$ be an $\O_S$-algebra, equipped with a map $\vall : \tilde{\Lambda} \to \R \cup \{+\infty\}$ satisfying the following conditions:
\begin{enumerate}
\item $\vall(x) = +\infty$ if and only if $x=0$;
\item $\vall(xy) \geq \vall(x)+\vall(y)$;
\item $\vall(x+y) \geq \inf(\vall(x),\vall(y))$;
\item $\vall(p) > 0$.
\end{enumerate}

\begin{rema}
    In \cite{BC08} the fourth item asks that $\vall(px) = \vall(p)+\vall(x)$ but this is actually never used.
\end{rema}

Following \cite[Def. 3.1.3]{BC08}, we recall the classical Colmez-Tate-Sen conditions in the case where $\Delta = \{1\}$:

(CTS1) There exists $c_1 > 0$ such that for any open subgroups $H' \subset H''$ of $H$, there exists $\alpha \in \tilde{\Lambda}^{H'}$ such that $\vall(\alpha) > -c_1$ and $\sum_{\tau \in H''/H'}\tau(\alpha)=1$.

(CTS2) There exist $c_2 > 0$ and for each open subgroup $H'$ of $H$ an integer $n(H') \in \N$, a nondecreasing sequence $(\Lambda_{H',n})$ of closed $\O_S$-subalgebras of $\Lambda^{H'}$, and for $n \geq n(H')$, an $\O_S$-linear application $R_{H,n}: \Lambda^{H',n} \to \Lambda_{H',n}$ satisfying:
\begin{enumerate}
\item if $H' \subset H''$ then $\Lambda_{H'',n} \subset \Lambda_{H',n}$, and the restriction of $R_{H',n}$ to $\tilde{\Lambda}^{H''}$ coincides with $R_{H'',n}$;
\item the maps $R_{H',n}$ are $\Lambda_{H',n}$-linear, and $R_{H',n}(x) = x$ if $x \in \Lambda_{H',n}$;
\item for all $g \in G$, we have $g(\Lambda_{H',n}) = \Lambda_{gH'g^{-1},n}$, and $g(R_{H',n}(x)) = R_{gH'g^{-1},n}(g(x))$;
\item if $n \geq n(H')$, and if $x \in \tilde{\Lambda}^{H'}$, then $\vall(R_{H',n}(x) \geq \vall(x)-c_2$;
\item if $x \in \tilde{\Lambda}^{H'}$, then $\lim\limits_{n \ra +\infty}R_{H',n}(x)=x$.
\end{enumerate}

(CTS3) There exist $c_3 > 0$ and for any open subgroup $G'$ of $G$, an integer $n(G) \geq n_1(H')$, where $H' = G' \cap H$, such that if $n \geq n(G')$, if $\gamma \in \tilde{\Gamma}_{H'}$ is such that $n(\gamma) \leq n$, then $\gamma-1$ is invertible on $X_{H',n} = (1-R_{H',n})(\tilde{\Lambda}^{H'})$ and we have $\vall((\gamma-1)^{-1}(x)) \geq \vall(x)-c_3$ for $x \in X_{H',n}$.

The following proposition, which is proposition 3.1.4 of \cite{BC08}, is straightforward:

\begin{prop}
\label{prop CST stable by tensor prod}
If $\tilde{\Lambda}$ is a $\Zp$-algebra satisfying the Colmez-Tate-Sen conditions, and if $S$ is a Banach algebra equipped with the trivial $G$-action, then $\O_S \hat{\otimes}_{\Zp}\tilde{\Lambda}$ satisfies the Colmez-Sen-Tate conditions with the same constants $c_1, c_2$ and $c_3$. 
\end{prop}

We also recall the following lemma:

\begin{lemm}
\label{lemma classical CST decompletion stable}
Let $\tilde{\Lambda}$ be a $\Zp$-algebra satisfying the Colmez-Tate-Sen conditions, with constants $c_1, c_2$ and $c_3$. Let $H'$ be an open subgroup of $H$, $n \geq n(H')$, $\gamma \in \tilde{\Gamma}_H$ such that $n(\gamma) \leq n$ and $B \in \GL_d(\tilde{\Lambda}^{H'})$. Assume that there exists $V_1, V_2 \in \GL_d(\Lambda_{H',n})$ with $\vall(V_1-1) > c_3$ and $\vall(V_2-1) > c_3$ such that $\gamma(B) = V_1BV_2$. Then $B \in \GL_d(\Lambda_{H',n})$. 
\end{lemm}
\begin{proof}
This is \cite[Lemm. 3.2.5]{BC08}. 
\end{proof}

\begin{theo}
\label{theo classical TS}
Let $\tilde{\Lambda}$ a $\Zp$-algebra satisfying the Colmez-Tate-Sen conditions, with constants $c_1, c_2$ and $c_3$. Let $\sigma \mapsto U_\sigma$ be a continuous cocycle from $G$ to $\GL_d(\tilde{\Lambda})$. If $G'$ is an open normal subgroup of $G$ such that $U_\sigma-1 \in p^k\M_d(\tilde{\Lambda})$, and $\vall(U_\sigma-1) > c_1+2c_2+2c_3$ for all $\sigma \in G'$, and if $H' = G' \cap H$, then there exists $M \in 1+p^k\M_d(\tilde{\Lambda})$ such that $\vall(M-1) > c_2+c_3$ and such that the cocycle $\sigma \mapsto V_\sigma = M^{-1}U_\sigma\sigma(M)$ is trivial on $H'$ with values in $\GL_d(\Lambda_{H',n(G')})$. 
\end{theo}
\begin{proof}
This is \cite[Prop. 3.2.6]{BC08}.
\end{proof}

Note that this is meant to be applied to $\tilde{\Lambda}$ representations: if $T$ is a free continuous $\tilde{\Lambda}$-representation of rank $d$ of $G$, and if $\beta$ is a basis of $T$ over $\tilde{\Lambda}$, then the map $g \in G \mapsto \mathrm{Mat}_\beta(g)$ defines a continuous cocycle from $G$ to $\GL_d(\tilde{\Lambda})$. Conversely, the data of such a cocycle endows $\tilde{\Lambda}$ with the structure of a free continuous $\tilde{\Lambda}$-representation of rank $d$, and cohomologous cocycles are exactly obtained by base change over $\tilde{\Lambda}$, which means that isomorphism classes of free $\tilde{\Lambda}$-representations of rank $d$ are in bijection with the continuous cohomology set $H^1(G,\GL_d(\tilde{\Lambda}))$.
\begin{rema}
    In \cite[Section 4]{porat2022overconvergence}, Porat develops a variant of the Tate-Sen formalism for Tate rings, in which the prime $p$ is replaced with a pseudouniformizer $f$, we will be interested in defining a multivariable version of this method.
\end{rema}
\subsection{Multivariable Tate-Sen descent for Tate rings}

Let $(\Lambda,\Lambda^+)$ be a Tate ring, with $\Lambda^+$ ring of definition and pseudouniformizer $f$. We assume that $\Lambda^+$ is $f$-adically complete. Recall that $f$ defines a valuation $\rm{val}_{\Lambda}$ that induces the $f$-adic topology on $\Lambda$.\\
We will use the notations of the previous subsection. Moreover if $(G'_{\alpha})_{\alpha\in\Delta}$ is a $\delta$-uple of subgroups of $G$, we will write $\underline{G'}_{\Delta}$ for the product $\times^{\alpha\in\Delta}G_{\alpha}$. Moreover if $\underline{H'}_{\Delta}=(H\cap G'_{\alpha})_{\alpha\in\Delta}$, we will write 
$\underline{\Gamma}_{\underline{H'},\Delta}$ for the quotient $\underline{G'}_{\Delta}/\underline{H'}_{\Delta}$. We let $\underline{C}_{\underline{H'},\Delta}$ be the product
$\times^{\alpha\in\Delta}C_{H'_{\alpha},\Delta}$, that is the center of $\underline{\Gamma}_{\underline{H'},\Delta}$, as group center commutes with direct products.\\
Using \cite[Lemme 3.1.1]{BC08} we obtain the following
\begin{lemm}
The group $\underline{C}_{\underline{H'},\Delta}$ is open in $\underline{\Gamma}_{\underline{H'},\Delta}$.
\end{lemm}
Suppose $\Lambda$ endowed with a continuous action of $G_{\Delta}$.
We have that the topology in $\Lambda$ is given by a valuation $\rm{val}_{\Lambda}$ satisfying analogous conditions to the ones of the previous subsection:
\begin{enumerate}
\item $\Lambda$ is separated with respect to $\rm{val}_{\Lambda}$.
\item $\rm{val}_{\Lambda}(xy)\ge \rm{val}_{\Lambda}(x)+\rm{val}_{\Lambda}(y)$, for any $x,y\in \Lambda$.
\item $\rm{val}_{\Lambda}(x+y)\ge \rm{inf}(\rm{val}_{\Lambda}(x),\rm{val}_{\Lambda}(y))$ for any $x,y\in \Lambda$.
\item $\rm{val}_{\Lambda}(f)>0$ and $\rm{val}_{\Lambda}(fx)=\rm{val}_{\Lambda}(f)+\rm{val}_{\Lambda}(x)$ if $x\in \Lambda$.
\end{enumerate}
We also assume that $G_{\Delta}$ acts isometrically on $\Lambda$.

Let $\Lambt_i$ be a Tate ring over $\Zp$, with ring of definition $\Lambt^{+}_i$, and pseudo-uniformizer $\beta_i$, such that $\beta_i$ divides $p$ in $\Lambda^{+}$, for $i=1,\dots,\delta$. Suppose that $\Lambt_i$ is endowed with a continous action of a profinite group $G$ admitting a continous character $\chi:G\to\Zp$. Let $S$ be an affinoid algebra. We set $\Lambt^{+}_{\Delta, S}$ to be the completed tensor product
$$\O_S\wotimes_{\Zp}(\wotimes_{\Zp}\Lambt^{+}).$$
We set 
$$\Lambt_{\Delta, S}=\Lambt^{+}_{\Delta, S}[\frac{1}{\beta_1\dots\beta_{\delta}}],$$
and we endow it with the $\beta_1\dots\beta_{\delta}$-adic topology. Note then that $\Lambt_{\Delta, S}$ is a Tate ring with ring of definition $\Lambt^{+}_{\Delta, S}$ and pseudo-uniformizer $\beta_1\dots\beta_{\delta}.$ In this subsection we will denote by $\rm{val}$ the valuation of $\Lambt_{\Delta, S}$.
We remark that $\Lambt_{\Delta, S}$ is endowed with an action of $G_{\Delta}$ induced by the action of $G_i$ on $\Lambt_i$ for $i=1,\dots,\delta$.
We have a continous map
$$\O_S\to\Lambt^{+}_{\Delta, S}.$$ 
From now on we assume that $\beta_d$ is invariant by $H$ and that $$(\Lambt^+_{\Delta})^{\underline{H}'_{\Delta}}=\Lambt_{\underline{H}_{\Delta}'},$$
where $\Lambt^+_{\underline{H}'_{\Delta}}=\Lambt^+_{H'_{\alpha_1}}\wotimes_{\Zp}\dots\wotimes_{\Zp}\Lambt^+_{H'_\delta}$.
We moreover assume that $\Lambt_i$ satisfies the Tate-Sen conditions of \cite[Subsection 4.1]{porat2022overconvergence}. In particular for each subgroup $H'_{\alpha}$ of $H_{\alpha}$ there exist maps
$$R^{\alpha}_{H'_{\alpha},n}:\Lambt^{H'_{\alpha}}\to \Lambda_{H'_{\alpha},n}$$ that are projections to some subalgebras $\Lambda_{H_{\alpha},n}$ for $\alpha\in\Delta$ and $n\ge n(G_{\alpha})$.
We let $$\Lambda^{+}_{\underline{H}'_{\Delta},n}=\Lambda^+_{H'_{\alpha_1},n}\wotimes_{\Zp}\dots\wotimes_{\Zp}\Lambda^+_{H'_{\alpha_{\delta}},n},$$
and $$\Lambda_{\underline{H}'_{\Delta},n}=\Lambda^+_{\underline{H}'_{\Delta},n}[\frac{1}{\beta_1\dots\beta_{\delta}}].$$
If we add the subscript $S$ to one of this rings that means that we tensor by $\O_{S}$ and complete for the $(\beta_1,\dots,\beta_{\delta})$-adic topology for the $+$-rings, for the non $+$ ones we do the same thing to the $+$ ring and we invert $p$.
\begin{theo}
\label{theo multivariable TS tate rings}
Let $c_1, c_2, c_3 > 0$. Suppose that for each $i \in \Delta$, $\tilde{\Lambda_i}$ satisfies the Colmez-Tate-Sen conditions for Tate rings (see \cite[Subsection 4.1]{porat2022overconvergence}) with constants $c_1, c_2$ and $c_3$ for the action of $G_{\Delta}$. Let $\sigma \mapsto U_\sigma$ be a continuous cocycle from $G_{\Delta}$ to $\GL_d(\Lambt^+_{\Delta, S})$. If $G'_{\alpha}$ is an open normal subgroup of $G_{\alpha}$ for $\alpha\in\Delta$ such that $U_\sigma-1 \in (\beta_1\dots\beta)^k\mathrm{M}_d(\Lambt^+_{\Delta, S})$ and $\val(U_\sigma-1) > \delta c_1+(\delta+1)(c_2+c_3)$, for all $\sigma \in \underline{G}'_{\Delta}$, and if $\underline{H}'_{\Delta} = \underline{G}'_{\Delta} \cap H_{\Delta}$, then there exists $M \in \rm{GL}_d(\Lambt_{\Delta, S})$ such that $\val(M-1) > c_2+c_3$ and such that the cocycle $\sigma \mapsto V_\sigma = M^{-1}U_\sigma\sigma(M)$ is trivial on $\underline{H}'_{\Delta}$ and with values in $\GL_d(\Lambda_{\underline{H}'_{\Delta},n(G'),S})$.
\end{theo}
\begin{proof}
 
    Since $\Lambt_i$ satisfies the Tate Sen conditions, for each open subgroup $H_{\alpha}''$ there exists an $\alpha_i$ in $\Lambt_i^+$ such that 
    $\vall(\alpha_i)>-c_1$ and $\sum_{\tau\in H_{\alpha}''/H_{\alpha}'}\tau(\alpha)=1$ for $i=1,\dots,\delta$. Thus the element $\alpha_{\Delta}=\alpha_1\otimes_{\Qp}\dots\otimes_{\Qp}\alpha_\delta$ 
    satisfies $\vall(\alpha_{\Delta})>-\delta c_1$
    and
    $$\sum_{\tau\in H''_{\Delta}/H'_{\Delta}}\tau(\alpha)=1.$$
    Thus in particular the axiom (TS1) of \cite[Subsection 4.1]{porat2022overconvergence} is satisfied. This implies that there exists $M'\in\rm{GL}_d(\Lambt_{\Delta, S})$, 
    such that the cocycle $\sigma \mapsto {M'}^{-1}U_\sigma\sigma(M')$ is trivial on $\underline{H}'_{\Delta}$, and $\vall(M'-1)> (\delta+1)(c_2+c_3)$.
    Note that $\Lambt_{\Delta, S}[1/\beta_d]$ is a complete Tate ring with pseudouniformizer $\beta_d$.
    Moreover we can extend the Tate traces 
    $$R_{H_{\alpha_\delta},n}:\Lambt^{H'_{\alpha}}\to \Lambda_{H'_{\alpha},n}$$
    to the completed tensor product $\Lambt_{\underline{H}_{\Delta}'}$, that is again a complete Tate ring with pseudouniformizer $\beta_d$.
    Then using \cite[Proposition 4.8]{porat2022overconvergence} we can find a matrix $M''\in\GL_{d}(\Lambt_{\Delta,S})$ such that the cocycle
    $\sigma \mapsto V_\sigma = {M''}^{-1}U_\sigma\sigma(M'')$ is trivial on $\underline{H}_{\Delta}$ and with values in $$\GL_d(\O_S\wotimes_{\Zp}\Lambt^+_{H'_{\alpha_1}}\wotimes_{\Zp}\dots\wotimes_{\Zp}\Lambda^+_{H'_{\alpha_{\delta},n(G'_{\alpha_{\delta}})}}).$$
    Note that $M''$ is constructed as an infinite product 
    $$\prod_{i=1}^{+\infty}M''_i,$$
    where the first factor is $M_1=(\gamma_{\delta}-1)^{-1}(1-R_{H'_{\delta},n(G'_{\alpha_{\delta}})})(U'_{\gamma_{\delta}})$, the second is $M_2=(\gamma_{\delta}-1)^{-1}(1-R_{H'_{\delta},,n(G'_{\alpha_{\delta}})})(M^{-1}_1U'_{\gamma_{\delta}}\gamma_{\delta}(M_1))$ and the others are constructed inductively like this. Thus since $(\gamma_{\delta}-1)^{-1}$ sends $\beta^{k}_{\delta}$ to $\beta_{\delta}^{k-c_3}$, and $\beta_i$ to $\beta_i$ for $i=1,\dots,\delta-1$, 
    and $R_{H'_{\delta},,n(G'_{\alpha_{\delta}})}$ sends $\beta^{-k}_{\delta}$ to $\beta_{\delta}^{k-c_2}$, and $\beta_i$ to $\beta_i$ for $i=1,\dots,\delta-1$, we get that $\vall(M''-1)>\delta(c_2+c_3)$. We can proceed inductively in this way for each variable to find $M$ as in the statement.  
\end{proof}

Given this result we can prove the following proposition.
\begin{prop}
\label{prop CST gives descended module}
Let $T$ be an $\O_S$-representation of dimension $d$ of $G$, let $V = S \otimes_{\O_S}T$ and let $k$ be an integer such that $\vall(p^k) > \delta c_1+(\delta+1)(c_2+c_3)$. Let $G'$ be an open normal subgroup of $G$ acting trivially on $T/p^kT$, let $H' = G' \cap H$ and let $\n \geq \n(G')$. Then $\tilde{\Lambda}_{\Delta, S}^+\otimes_{\O_S}T$ contains a unique sub-$\Lambda_{\underline{H}'_{\Delta},\n,S}^+$-module $\D_{H',\n}^+(T)$, free of rank $d$ satisfying:
\begin{enumerate}
\item $\D_{\underline{H}'_{\Delta},\n}^+(T)$ is fixed by $\underline{H}'$ and stable by $\underline{G}$;
\item the natural map $\tilde{\Lambda}^+_{\Delta,S} \otimes_{\Lambda_{\underline{H}'_{\Delta},\n}^+}\D_{\underline{H}',\n,S}^+(T) \to \tilde{\Lambda}^+_{\Delta,S} \otimes_{\O_S}T$ is an isomorphism;
\item $\D_{\underline{H}'_{\Delta},\n}^+(T)$ admits a basis over $\Lambda_{\underline{H}'_{\Delta},\n,S}^+$ which is $c_3$-fixed by $\underline{G}'/\underline{H}'$.
\end{enumerate}
\end{prop}
\begin{proof}
The existence of $\D_{\underline{H}'_{\Delta},\n}^+(T)$ follows from the fact that the $\Lambda_{\underline{H}',\n,S}^+$-module whose basis is the one provided by the base change given by the matrix $M$ of theorem \ref{theo multivariable TS tate rings} satisfies those conditions. The fact that there is only one such sub-module is a consequence of a repeated use of lemma \cite[Proposition 4.9]{porat2022overconvergence}. Indeed, assume that there are two such submodules, with respective bases $e_1,\ldots,e_d$ and $e_1',\ldots,e_d'$. Let $B \in \GL_d(\tilde{\Lambda}^+)$ the matrix of the $e_j'$ in the basis $e_1,\ldots,e_d$. Then $B$ is invariant under the action of $H'$, and if $W$, $W'$ denote respectively the matrices of a generator $\gamma$ of $H_{\delta}'$ in the bases $e_1,\ldots,e_d$ and $e_1',\ldots,e_d'$. We have $\val(W-1) > c_3$, $\val(W'-1) > c_3$ by construction, and moreover $W$ and $W'$ belong to $\GL_d(\Lambda_{\underline{H}',\n,S}^+)$. We have $W'=B^{-1}W\gamma(B)$ so that we can apply Lemma \ref{lemma classical CST decompletion stable}. Which means that $B$ has its coefficients in $\O_S\wotimes_{\Zp}\Lambda^+_1\wotimes_{\Zp}\dots\wotimes_{\Zp}\Lambda_{H'_{\delta},n_{\delta}}$. Applying the same reasoning with generators of $H_i$, for every $i \in \Delta$ shows that $B \in \GL_d(\Lambda_{\underline{H}',\n,S}^+)$, so that the $\Lambda_{\underline{H}',\n,S}^+$-modules generated respectively by $e_1,\ldots,e_d$ and $e_1',\ldots,e_d'$ are equal.
\end{proof}

\subsection{Multivariable Sen theory}
\label{subs multi Sen}
We now explain how to recover the results from \cite[\S 3]{BriChiaMaz21} as a consequence of the multivariable Colmez-Tate-Sen formalism. As stated before, this formalism was already underlined in \cite{BriChiaMaz21} and allows us to extend their results to families. In what follows, we let $\Kunder = (K_\alpha)_{\alpha \in \Delta}$ be $\delta$ finite extensions of $\Qp$, and we let $K_{\alpha,\infty}$ denote the cyclotomic extension of $K_\alpha$ for $\alpha \in \Delta$. We let $\underline{K}_{\infty} = (K_{\alpha,\infty})_{\alpha \in \Delta}$ and $K_{\Delta,\infty} = \otimes_{\Qp}^{\alpha \in \Delta}K_{\alpha,\infty}$. We also write $\hat{K_{\Delta,\infty}}=(\wotimes_{\Zp}^{\alpha \in \Delta}\O_{K_{\alpha,\infty}})[1/p]$, and we have $\hat{K_{\Delta,\infty}}= C_{\Delta}^{H_{\Kunder,\Delta}}$ by \cite[Theorem. 3.3]{BriChiaMaz21}.  

Recall the following proposition:

\begin{prop}
\label{prop Cp satisfies CTS}
The ring $\tilde{\Lambda}=\Cp=\O_{\Cp}[1/p]$ satisfies the conditions (CTS1), (CTS2) and (CTS3), with $\tilde{\Lambda}^{H_L} = \hat{L}_\infty = \O_{\hat{L_\infty}}[1/p]$, $\Lambda_{H_L,n} = L_n = \O_{L_n}[1/p]$, $R_{H_L,n} = R_{L,n}$ and $\vall = v_p$, for any $c_1 > 0, c_2 > 0$ and $c_3 > \frac{1}{p-1}$. 
\end{prop}
\begin{proof}
See \cite[Prop. 4.1.1]{BC08}.
\end{proof}

In what follows, we fix constants $c_1 > 0, c_2 > 0$ and $c_3 > \frac{1}{p-1}$ such that $\delta c_1+(\delta+1)(c_2+c_3) < v_p(12p^{\delta})$.

\begin{theo}
\label{theo Sen theory from CST}
Let $S$ be a $\Qp$-Banach algebra, let $T$ be an $\O_S$-representation of dimension $d$ of $\G_{\underline{K},\Delta}$, and let $V = S \otimes_{\O_S}T$. Let $\underline{L} = (L_1,\ldots,L_\delta)$ be such that for all $i \in \Delta$, $L_i/K_i$ is Galois and such that $\G_{\underline{L},\Delta}$ acts trivially on $T/12p^{\delta}$, and let $\n \geq \n(\underline{L})$. Then $(S\hat{\otimes}\C_{\Delta})\otimes_S V$ contains a unique sub-$(S \otimes L_{\n,\Delta})$-module $\D_\Sen^{L_{\Delta,\n}}(V)$, free of rank $d$, such that:
\begin{enumerate}
\item $\D_\Sen^{L_{\Delta,\n}}(V)$ is fixed by $H_{\underline{L},\Delta}$ and stable by $\G_{\Kunder,\Delta}$;
\item $\D_\Sen^{L_{\Delta,\n}}(V)$ contains a basis over $S \otimes L_{\n,\Delta}$ which is $c_3$-fixed by $\Gamma_{\underline{L},\Delta}$;
\item the natural map $(S\hat{\otimes}\C_{\Delta}) \otimes_{S \otimes L_{\n,\Delta}}\D_\Sen^{L_{\Delta,\n}}(V) \ra (S\hat{\otimes}\C_{\Delta})\otimes_S V$ is an isomorphism.
\end{enumerate} 
Moreover, we have $S/\mathfrak{m}_x\otimes_S \D_\Sen^{L_{\Delta,\n}}(V) \simeq \D_\Sen^{L_{\Delta,\n}}(V_x)$.
\end{theo}
\begin{proof}
This follows from Proposition \ref{prop Cp satisfies CTS} so that we can apply Proposition \ref{prop CST gives descended module}. The last point follows from the Proposition applied to $S/\mathfrak{m}_x$ and the fact that the image of $S/\mathfrak{m}_x\otimes_S \D_\Sen^{L_{\Delta,\n}}(V)$ in $(E_x \otimes \C_{\Delta})\otimes_{E_x}V_x$ satisfies the three conditions of the theorem, where $E_x = S/\mathfrak{m}_x$. 
\end{proof}

In particular, we recover \cite[Theorem 3.19, Corollary 3.20 and 3.21]{BriChiaMaz21} (after a Galois descent argument as in \cite[Theorem 3.19]{BriChiaMaz21}) and the Tate-Sen formalism allows us to extend those to families of representations. 

The multivariable Sen theory developed in \cite{BriChiaMaz21}, just as in the classical case (\cite{sen1980continuous}), makes use of ``$K_\Delta$-finite vectors'': if $W$ is a $\hat{K_{\Delta,\infty}}$-representation of $\Gamma_{\Kunder,\Delta}$, we say that an element of $W$ is $K_{\Delta}$-finite if $x$ lives in a $K_\Delta$-submodule of finite type of $W$ which is stable by $\Gamma_{\Kunder,\Delta}$, and we let $W^{\fin}$ denote the set of $K_\Delta$-finite vectors of $W$. Note that, by Proposition 3.22 of \cite{BriChiaMaz21}, we have $\hat{K_{\Delta,\infty}}=K_{\Delta,\infty}$ so that $W^{\fin}$ can naturally be viewed as a $K_{\Delta,\infty}$-submodule of $W$. 

If $V$ is a $p$-adic representation of $\G_{\Kunder,\Delta}$, we can define its Sen module by $\D_{\Sen,\Delta}(V) := ((C_\Delta \otimes_{\Qp}V)^{H_{\Kunder,\Delta}})^{\fin}$ as in \cite[\S 3]{BriChiaMaz21} by analogy with the classical case (\cite{sen1980continuous}). 

Note that the notation of $\D_{\Sen,\Delta}(V)$ is compatible with the definition of the module $\D_\Sen^{L_{\Delta,\n}}(V)$ in theorem \ref{theo Sen theory from CST}, since we have $\D_{\Sen,\Delta}(V) = (L_{\infty,\Delta}\otimes_{L_{\Delta,\n}}\D_\Sen^{L_{\Delta,\n}}(V))^{H_{\Kunder,\Delta}/H_{\Lunder,\Delta}}$ (the proof of \cite[Corollary 3.24]{BriChiaMaz21} shows that $L_{\infty,\Delta}\otimes_{L_{\Delta,\n}}\D_\Sen^{L_{\Delta,\n}}(V)$ is the set of $L_{\Delta}$-finite vectors of $(C_{\Delta}\otimes_{\Qp}V)^{H_{\Lunder,\Delta}}$ as $\Gamma_{\Lunder,\Delta}$-representations, and the result follows by taking the invariants by $H_{\Kunder,\Delta}/H_{\Lunder,\Delta}$ by étale descent using Proposition 3.25 and Corollary 3.27 of ibid.). 

We now explain how one can recover the Sen operators and Sen theory developed in \cite[\S 3]{BriChiaMaz21} in the spirit of \cite{Ber14SenLa}, using locally analytic vectors. 

\begin{lemm}
    We have $\hat{K_{\infty,\Delta}}^{\Gamma_{\Kunder,\Delta}-\la} = K_{\infty,\Delta}$
\end{lemm}
\begin{proof}
    If $G = \Gamma_{\Kunder,\Delta}$, then $\hat{K_{\infty,\Delta}}^{G-\la} = \cup_{n \geq 1}\hat{K_{\infty,\Delta}}^{G_n-\la}$, and Corollary \ref{coro computes locanamulti} implies directly that $\hat{K_{\infty,\Delta}}^{G_n-\la} = K_{\Delta,\n}$. 
\end{proof}

\begin{prop}
\label{prop la recovers fin Sen CDelta}
    Let $W$ be a free $\hat{K_{\infty,\Delta}}$-semilinear representation of $\Gamma_{\Kunder,\Delta}$ of finite type. Then $W^{\la} \simeq W^{\fin}$.
\end{prop}
\begin{proof}
    This is the same proof as the one of \cite[Thm. 3.2]{Ber14SenLa}. If $x \in W^{\fin}$ then it lies in some finitely generated $K_\Delta$-module stable by $\Gamma_{\Kunder,\Delta}$ and therefore inside some finite dimensional $\Qp$-vector space endowed with a $\Gamma_{\Kunder,\Delta}$-action. By Cartan's theorem \cite[Prop. 3.6.10]{Emer}, this implies that $x \in W^{\la}$. 

    For the converse, by theorem 3.28 of \cite{BriChiaMaz21}, there exists a basis $e_1,\cdots,e_d$ of $W^{\fin}$ over $K_{\Delta,\infty}$ which is also a basis of $W$ over $\hat{K_{\Delta,\infty}}$. Then by Proposition \ref{lainla and painpa}, $W^{\la} = \oplus_{i=1}^d(\hat{K_{\Delta,\infty}})^{\la}e_i = \oplus_{i=1}^dK_{\Delta,\infty}\cdot e_i = W^{\fin}$. 
\end{proof}

If $W$ is a locally-analytic representation of $\Gamma_{\Kunder,\Delta}$, the Lie algebra of $\Gamma_{\Kunder,\Delta}$ naturally acts on $W$, so that we get operators $(\nabla_\alpha)_{\alpha \in \Delta}$ defined by:

$$\nabla_\alpha = \lim\limits_{\gamma \in \Gamma_{K_\alpha}, \gamma \ra 1}\frac{\gamma-1}{\log(\chi_{\cycl}(\gamma))}.$$

The operator $\nabla_\alpha$ is also equal to $\log \gamma$ for $\gamma \in \Gamma_{K_\alpha}$ close enough to $1$ (in the cyclotomic case these are the \emph{generalized Sen operators of \cite[Subsection 3.30]{BriChiaMaz21}}). 

We define the multivariable HT (Hodge-Tate) weights of $V$ as the elements $(\lambda_\alpha)_{\alpha \in \Delta}$ such that $\cap_{\alpha \in \Delta}\ker(\nabla_\alpha-\lambda_\alpha)$ is a nontrivial sub-$K_{\infty,\Delta}$-module of $\D_{\Sen,\Delta}(V)$.

\subsection{Overconvergent families of multivariable $(\phi,\Gamma)$-modules}
We now explain how to attach to a family of representations a family of $(\phi_{\Delta},\Gamma_{\Kunder,\Delta})$-modules using these multivariable Colmez-Sen-Tate conditions. When working over a base which is a finite extension of $\Qp$, these constructions coincide with the ones of \cite{KedCarZab} and of \cite{PalZab21}. 

\begin{prop}
\label{prop At satisfies CTS}
The rings $\tilde{\Lambda}=\At^{\dagger,r}[1/\varpi]$ satisfy the conditions (CTS1), (CTS2) and (CTS3) for any $r \geq 1$, with $\tilde{\Lambda}^{H_L} = \At_L^{\dagger,r}[1/\varpi]$, $\Lambda_{H_L,n} = \phi^{-n}(\A^{\dagger,p^nr}[1/\varpi])$, $R_{H_L,n} = R_{L,n}$ and $\vall = v_r$, for any $c_1 > 0, c_2 > 0$ and $c_3 > \frac{1}{p-1}$. 
\end{prop}
\begin{proof}
See \cite[Prop. 4.2.1]{BC08}.
\end{proof}

If $V$ is a (family of) representation(s) of $\G_{\Kunder,\Delta}$, admitting a Galois-stable integral lattice $T$ such that $V=S \otimes_{\O_S}T$, and if $(L_{\alpha})_{\alpha \in \Delta}$ is such that for all $i \in \Delta$, $L_i/K_i$ is Galois and such that $\G_{\underline{L},\Delta}$ acts trivially on $T/12p^{\delta}$, we let $\s(V) = (\max(r_{n(L_\alpha)},s(L_\alpha/K_\alpha)))_{\alpha \in \Delta}$, and up to increasing the $s(V)_\alpha$, we can assume that there exists $\n(V) = (n(V)_\alpha)_{\alpha \in \Delta}$ such that $\s(V) = (r_{n(V)_\alpha})_{\alpha \in \Delta}$. If $\s \geq \s(V)$, we let $\D_{\Lunder,\Delta}^{\dagger,\s}(T)= ((\O_S\hat{\otimes}_{\Zp}\A_{\Delta}^{\dagger,\s})\otimes_{\Zp}T)^{H_{\Lunder,\Delta}}$ which is by lemma \ref{lemma1 invar Atdagdelta} an $\O_S\hat{\otimes}_{\Zp}\A_{\Lunder,\Delta}^{\dagger,\s}$-module, which is endowed with an action of $\Gamma_{\Lunder,\Delta}$. For $\n \geq 0$, we let $\D_{\Lunder,\Delta,\n}^{\dagger,\s}(T)=\phi^{-\n}(\D_{\Lunder,\Delta}^{\dagger,p^{\n}\s}(T))$ which is an $\O_S\hat{\otimes}_{\Zp}\A_{\Lunder,\Delta,\n}^{\dagger,\s}= \phi^{-\n}(\O_S\hat{\otimes}_{\Zp}\A_{\Lunder,\Delta}^{\dagger,p^{\n}\s})$-module. 

For $\s \geq \s(V)$, we let $\D_{\Kunder,\Delta}^{\dagger,\s}(T)=(\D_{\Lunder,\Delta}^{\dagger,\s}(T))^{\G_{\Kunder,\Delta}/\G_{\Lunder,\Delta}}$ and $\D_{\Kunder,\Delta}^{\dagger,\s}(V) = \D_{\Kunder,\Delta}^{\dagger,\s}(T)[1/p]$. 

\begin{prop}
\label{prop overconv in fam}
    Let $S$ be a $\Qp$-Banach algebra, let $T$ be an $\O_S$-representation of dimension $d$ of $\G_{\Kunder,\Delta}$ and let $V=S \otimes_{\O_S}T$. Let $\underline{L} = (L_1,\ldots,L_\delta)$ be such that for all $i \in \Delta$, $L_i/K_i$ is Galois and such that $\G_{\underline{L},\Delta}$ acts trivially on $T/12p^{\delta}$, and let $\n \geq \n(\underline{L})$. Then $(\O_S \hat{\otimes}\At_{\Delta}^{\dagger,\underline{r_0}})\otimes_{\O_S} T$ admits a unique sub-$\O_S \hat{\otimes}\A_{\Lunder,\Delta,\n}^{\dagger,\underline{r_0}}$-module $\D_{\Lunder,\Delta,\n}^{\dagger,\underline{r_0}}(T)$ free of rank $d$, fixed by $H_{\Lunder,\Delta}$, stable by $\G_{\Kunder,\Delta}$, having a basis $c_3$-fixed by $\Gamma_{\Lunder,\Delta}$ and such that:
    $$(\O_S \hat{\otimes}\At_{\Delta}^{\dagger,\underline{r_0}})\otimes_{\O_S \hat{\otimes}\A_{\Lunder,\Delta,\n}^{\dagger,\underline{r_0}}}\D_{\Lunder,\Delta,\n}^{\dagger,\underline{r_0}}(T) \simeq (\O_S \hat{\otimes}\At_{\Delta}^{\dagger,\underline{r_0}})\otimes_{\O_S}T.$$
\end{prop}
\begin{proof}
    This is Proposition \ref{prop CST gives descended module} which we can use thanks to Theorem \ref{theo multivariable TS tate rings} and Proposition \ref{prop At satisfies CTS}.
\end{proof}

If $V$ is an $S$-family of representations of $\G_{\Kunder,\Delta}$ admitting a Galois-stable integral lattice $T$, we let 
$$\D_{\Kunder,\Delta}^{\dagger,\s}(V):=((S\hat{\otimes}_{\Qp}\B_{\Lunder,\Delta}^{\dagger,\s})\otimes_{S\hat{\otimes}_{\Qp}\B_{\Lunder,\Delta}^{\dagger,\s(V)}}\phi^{\n(V)}(\D_{\Lunder,\Delta,\n(V)}^{\dagger,\underline{r_0}}(V)))^{H_{\Kunder,\Delta}} ,$$
where $\D_{\Lunder,\Delta,\n(V)}^{\dagger,\underline{r_0}}(V) = \D_{\Lunder,\Delta,\n(V)}^{\dagger,\underline{r_0}}(T)[1/p]$.

\begin{theo}
\label{theo overconv in fam}
    If $V$ is an $S$-family of representations of $\G_{\Kunder,\Delta}$, free of dimension $d$, admitting a Galois-stable integral lattice, and if $\s \geq \s(V)$, then:
    \begin{enumerate}
        \item $\D_{\Kunder,\Delta}^{\dagger,\s}(V)$ is a projective $S \hat{\otimes}_{\Qp}\B_{\Kunder,\Delta}^{\dagger,\s}$-module of rank $d$;
        \item the map $(S \hat{\otimes}_{\Qp}\Bt_{\Kunder,\Delta}^{\dagger,\s})\otimes_{S \hat{\otimes}_{\Qp}\B_{\Kunder,\Delta}^{\dagger,\s}}\D_{\Kunder,\Delta}^{\dagger,\s}(V) \ra (S \hat{\otimes}_{\Qp}\Bt_{\Kunder,\Delta}^{\dagger,\s})\otimes_SV$ is an isomorphism;
        \item if $x \in \cal{X}$, the map $S/\mathfrak{m}_x \otimes_S\D_{\Kunder,\Delta}^{\dagger,\s}(V) \ra \D_{\Kunder,\Delta}^{\dagger,\s}(V_x)$ is an isomorphism.
    \end{enumerate}
\end{theo}
\begin{proof}
    The proof is the same as the one of \cite[Théorème 4.2.9]{BC08}: Proposition \ref{prop overconv in fam} implies that $\D_{\Lunder,\Delta}^{\dagger,\s}(V)$ is free of rank $d$ over $S\hat{\otimes}_{\Qp}\B_{\Lunder,\Delta}^{\dagger,\s}$ and that the map $(S \hat{\otimes}_{\Qp}\Bt_{\Delta}^{\dagger,\s})\otimes_{S \hat{\otimes}_{\Qp}\B_{\Lunder,\Delta}^{\dagger,\s}}\D_{\Lunder,\Delta}^{\dagger,\s}(V) \ra (S \hat{\otimes}_{\Qp}\Bt_{\Delta}^{\dagger,\s})\otimes_SV$ is an isomorphism. Proposition \ref{prop taking invariants gives proj and iso} then implies the first two points of the Theorem. For the last point, the same argument as in \cite[Théorème 4.2.9]{BC08} can be followed verbatim.
\end{proof}

It is clear by construction that after tensoring the modules over either $\O_S \wotimes_{\Zp}\A_{\Lunder,\Delta}^{\dagger}$ (before taking the invariants by $H_{\Kunder,\Delta}$) or $\O_S \wotimes_{\Zp}\A_{\Kunder,\Delta}^{\dagger}$ (after having taken the invariants by $H_{\Kunder}$) what we obtain are étale $(\phi_\Delta,\Gamma_{\Lunder,\Delta})$- or $(\phi_\Delta,\Gamma_{\Kunder,\Delta})$-modules, as the $\phi_\alpha$ act trivially on $T$. 

We now specialize our constructions to the case where $S$ is a finite extension of $\Qp$. Using Lemma \ref{lemma stable lattice field}, any $S$-representation $V$ in this case can be written as $V=S \otimes_{\O_S}T$ for some Galois-stable lattice $T$ of $V$.

\begin{prop}
\label{prop overconv over Qp}
    Let $T$ be a free $\Zp$-representation of $\G_{\Kunder,\Delta}$. Let $\underline{L} = (L_1,\ldots,L_\delta)$ be such that for all $i \in \Delta$, $L_i/K_i$ is Galois and such that $\G_{\underline{L},\Delta}$ acts trivially on $T/12p^{\delta}T$, and let $\n \geq \n(\underline{L})$. Then $\D_{\Lunder,\Delta,\n}^{\dagger,\underline{r_0}}(T)=\phi^{-\n}((\A_{\Delta}^{\dagger,\r_{\n}} \otimes_{\Zp}T)^{H_{\Lunder,\Delta}})$ is the unique sub-$\A_{\Lunder,\Delta,\n}^{\dagger,\underline{r_0}}$-module, free of rank $d$ of $\At_{\Delta}^{\dagger,\underline{r_0}}\otimes_{\Zp}T$ satisfying the following:
    \begin{enumerate}
        \item $\D_{\Lunder,\Delta,\n}^{\dagger,\underline{r_0}}(T)$ is fixed by $H_{\Lunder,\Delta}$ and stable by $\G_{\Kunder,\Delta}$;
        \item the natural map $\At_{\Delta}^{\dagger,\underline{r_0}}\otimes_{\A_{\Lunder,\Delta,\n}^{\dagger,\underline{r_0}}}\D_{\Lunder,\Delta,\n}^{\dagger,\underline{r_0}}(T) \ra \At_{\Delta}^{\dagger,\underline{r_0}}\otimes_{\Zp}T$ is an isomorphism;
        \item the $\A_{\Lunder,\Delta,\n}^{\dagger,\underline{r_0}}$-module $\D_{\Lunder,\Delta,\n}^{\dagger,\underline{r_0}}(T)$ admits a basis in which if $\gamma \in \Gamma_{\Lunder,\Delta}$ then the matrix $W_\gamma$ of $\gamma$ in this basis satisfies $v_{\underline{r_0}}(W_\gamma-1) > c_3$.
    \end{enumerate}
\end{prop}
\begin{proof}
This is similar to \cite[Proposition 4.2.3]{BC08}: the existence of a unique such sub-$\A_{\Lunder,\Delta,\n}^{\dagger,\underline{r_0}}$-module follows directly from the application of our Colmez-Sen-Tate method, using Proposition \ref{prop overconv in fam}. It remains to see that this module is indeed $\D_{\Lunder,\Delta,\n}^{\dagger,\underline{r_0}}(T)=\phi^{-\n}((\A_{\Delta}^{\dagger,\r_{\n}} \otimes_{\Zp}T)^{H_{\Lunder,\Delta}})$. 

As in the proof of \cite[Proposition 4.2.3]{BC08}, we are left to compare the module obtained by our method with the one coming from Theorem \ref{theo overconvergence PalZab and KedCarZab} (which is one of the main theorem of \cite{KedCarZab}). This is done in \cite{BC08} using Lemma 4.2.2 of ibid. and here we just need the analogue for our multivariable setting, which is given by \cite[Lemma 6.13]{KedCarZab}. 
\end{proof}

In particular this implies that if $\s \geq r_{\n(\Lunder)}$ then $\D_{\Lunder,\Delta}^{\dagger,\s}(T)$ is free of rank $d$ and the natural map $\A_{\Delta}^{\dagger,\s} \otimes_{\A_{\Lunder,\Delta}^{\dagger,\s}}\D_{\Lunder,\Delta}^{\dagger,\s}(T) \ra \A_{\Delta}^{\dagger,\s} \otimes_{\Zp}T$ is an isomorphism. 

By the unicity property in proposition \ref{prop overconv over Qp}, it follows that $\D_{\Kunder,\Delta}^{\dagger}(T):=\cup_{\s \geq \s(V)}\D_{\Kunder,\Delta}^{\dagger,\s}(T)$, where $\D_{\Kunder,\Delta}^{\dagger,\s}(T)$ is defined by

$$\D_{\Kunder,\Delta}^{\dagger,\s}(T):=((\O_S\hat{\otimes}_{\Zp}\A_{\Lunder,\Delta}^{\dagger,\s})\otimes_{\O_S\hat{\otimes}_{\Qp}\A_{\Lunder,\Delta}^{\dagger,\s(V)}}\phi^{\n(V)}(\D_{\Lunder,\Delta,\n(T)}^{\dagger,\underline{r_0}}(T)))^{H_{\Kunder,\Delta}} ,$$

is equal to the étale $(\phi_\Delta,\Gamma_{\Kunder,\Delta})$-module over $\A_{\Kunder,\Delta}^{\dagger}$ attached to $T$ by Theorem \ref{theo overconvergence PalZab and KedCarZab}. In particular, this allows us to recover the constructions of \cite{PalZab21} and \cite{KedCarZab} of overconvergent $(\phi_\Delta,\Gamma_{\Kunder,\Delta})$-modules attached to $p$-adic representations of $\G_{\Kunder,\Delta}$. Moreover, these constructions extend to families of representations. As in the classical case, the functor $V \mapsto \D^{\dagger}(V)$ is no longer an equivalence of categories between $S$-representations of $\G_{\Kunder,\Delta}$ and étale $(\phi_\Delta,\Gamma_{\Kunder,\Delta})$-modules over $S \wotimes_{\Qp}\B_{\Kunder,\Delta}^\dagger$ if $S$ is no longer a finite extension of $\Qp$ (cf \cite[Remarque 4.2.10]{BC08}).

Finally, we are able to apply our constructions to generalize \cite[Theorem 6.19]{KedCarZab} (which does not seem to be easy using the methods developed in ibid as stated in their Remark 6.20). Before doing so, we need several intermediate results, basically following \cite[\S 3]{porat2022overconvergence} :

\begin{lemm}
\label{Lemm approx for overconv Porat}
    Let $R$ be a commutative ring, endowed with a nonarchimedean valuation $v_R$, and an invertible morphism $\phi:R \ra R$ such that $v_R(\phi(x)) = pv_R(x)$. Let $X \in \GL_d(R)$. Then for any $c < \frac{p}{p-1}(v_R(X)+v_R(X^{-1}))$, and for any $Y \in \mathrm{M}_d(R)$, there exist $U,V \in \mathrm{M}_d(R)$ such that $v_R(V) \geq c$ and 
    \[X^{-1}\phi(U)X-U=Y-V \]
\end{lemm}
\begin{proof}
    This is Lemma 3.3 of \cite{porat2022overconvergence}. 
\end{proof}

\begin{lemm}
\label{lemm phi mod descends to overconv with perfect coeffs}
    Let $M$ be a free étale $\phi_{\Delta}$-module over $\At_{\Kunder,\Delta}$. Then there exists a unique free étale $\phi_{\Delta}$-module $M^\dagger$ over $\At_{\Kunder,\Delta}^\dagger$, contained in $M$ such that the natural map
    \[ M^\dagger \otimes_{\At_{\Kunder,\Delta}^\dagger}\At_{\Kunder,\Delta} \ra M \]
    is an isomorphim.
\end{lemm}
\begin{proof}
    The idea is the same as for Proposition 3.4 of \cite{porat2022overconvergence}: start with a basis of $M$, and let $X$ denote the matrix of the Frobenius $\phi_{\Delta} = \phi_{\alpha_1}\circ \cdots \circ \phi_{\alpha_\delta}$. Write $X = \sum_{n \geq 0}p^n[X_n]$ (which is possible because $\At_{\Kunder,\Delta} = W(C_\Delta^\flat)$ by Remark \ref{rmk AtDelta as Witt vectors}). We want to construct a matrix $U = \sum_{n \geq 0}p^n[U_n] \in \GL_d(\At_{\Kunder,\Delta})$ such that $C = U^{-1}X\phi_{\Delta}(U) \in \GL_d(\At_{\Kunder,\Delta}^\dagger)$. Let us take $U_0 = \Id$. Suppose now that $U_0,\cdots,U_{n-1}$ have been defined, and let us write $U' = \sum_{i=0}^{n-1}p^i[U_i]$. We can write $(U')^{-1}X\phi_{\Delta}(U') = \sum_{i=0}^{n-1}p^i[C_i] \mod p^n\mathrm{M}_d(\At_{\Kunder,\Delta})$, where $C_i$ depends only on $U_0,\cdots,U_i$, so that there exists $Y \in \mathrm{M}_d(\At_{\Kunder,\Delta})$ such that
    \[(U')^{-1}X\phi_{\Delta}(U') - \sum_{i=0}^{n-1}p^i[C_i]=p^nY.\]

    We now want to find $U_n \in \mathrm{M}_d(C_\Delta^\flat)$ such that 
    \[(U'+p^n[U_n])^{-1}X\phi_{\Delta}(U'+p^n[U_n]) = \sum_{i=0}^{n}p^i[C_i] \mod p^{n+1}\mathrm{M}_d(\At_{\Kunder,\Delta})\]
    satisfies that $v_\alpha(C_n)$ is bounded below for all $\alpha \in \Delta$.

    It suffices to prove that the mod $p$ equation
    \[U_n - X_0\phi_{\Delta}(U_n)X_0^{-1} = \overline{Y}X_0^{-1}-C_nX_0^{-1}\]
    can be solved in $C_\Delta^\flat$ in both $U_n$ and $C_n$, with each $v_\alpha(p^n[C_n])$ bounded independantly of $n$. But this is possible by Lemma \ref{Lemm approx for overconv Porat}, using the $(\varpi_{\alpha_1}\cdots \varpi_{\alpha_\delta})$-adic valuation on $C_\Delta^\flat$. 

    By induction, we obtain a matrix $U = \sum_{n \geq 0}p^n[U_n] \in \GL_d(\At_{\Kunder,\Delta})$ such that $C = U^{-1}X\phi_{\Delta}(U) \in \GL_d(\At_{\Kunder,\Delta}^\dagger)$, which is what we wanted.

    To finish the proof, it remains to prove that the overconvergent module $M^\dagger$ obtained is unique. This follows from the fact that if two such modules existed, then the base change matrix $V \in \GL_d(\At_{\Kunder,\Delta})$ between the two, would satisfy a relation of the form $\phi_\alpha(V)=A_\alpha VB_\alpha$, with $A_\alpha,B_\alpha \in \GL_d(\At_{\Kunder,\Delta}^\dagger)$ for each $\alpha \in \Delta$. The same proof as in \cite[Prop. 3.5]{porat2022overconvergence} shows that the coefficients of $V$ belong to $\tilde{\O}_{\mathcal{E}_\Delta}^{\dagger,j_\alpha}$ for each $\alpha \in \Delta$ (following Notation 2.10 of \cite{KedCarZab}), and thus $V \in \GL_d(\At_{\Kunder,\Delta}^\dagger)$.
\end{proof}

\begin{prop}
    Let $\D$ be a $(\phi_\Delta,\Gamma_{\Kunder,\Delta})$-module over $B=A[1/p]$, where $A$ is either $\At_{\Kunder,\Delta}^{\dagger}$, $\A_{\Kunder,\Delta}^{\dagger}$, $\At_{\Kunder,\Delta}$ or $\A_{\Kunder,\Delta}$. Assume that $\D$ satisfy the following:
    \begin{enumerate}
        \item The underlying $\phi_\Delta$-module of $\D$ is the base extension of an étale $\phi_\Delta$-module over $A$.
        \item The action of $\Gamma_{\Kunder,\Delta}$ is bounded: for some finitely generated $A$-module $\D_0$ generating $\D$ over $B$, the action of $\Gamma_{\Kunder,\Delta}$ carries $\D_0$ into $p^{-m}\D_0$ for some $m \geq 0$.
    \end{enumerate}
    Then $\D$ is an étale $(\phi_\Delta,\Gamma_{\Kunder,\Delta})$ over $B$.
\end{prop}
\begin{proof}
    The case where $A= \A_{\Kunder,\Delta}^{\dagger}$ or $A=\A_{\Kunder,\Delta}$ were treated by \cite[Theorem 6.19]{KedCarZab}, so it remains to prove the last two cases. We let $M_0 = \At_{\Delta} \otimes_A \D_0$ and $M = \Bt_{\Delta} \otimes_B \D = M_0[1/p]$ (note that if $A = \At_{\Kunder,\Delta}^{\dagger}$, then we could just tensor by $\At_{\Delta}^\dagger$ over $A$ directly and skip the next step). By assumption, the action of $\Gamma_{\Kunder,\Delta}$ carries $\D_0$ into $p^{-m}\D_0$ for some $m \geq 0$, so that it carries $M_0$ into $p^{-m}M_0$, and $M_0$ is an étale $\At_\Delta$-module. 

    By lemma \ref{lemm phi mod descends to overconv with perfect coeffs}, there exists $M_0^\dagger \subset M_0$ a sub-$\phi_\Delta$-module over $\At_\Delta^\dagger$ of $M_0$ which is étale, and such that the action of $\Gamma_{\Kunder,\Delta}$ carries $M_0^\dagger$ into $p^{-m}M_0^\dagger$ (since it is given by a base change matrix with coefficients in $\GL_d(\At_{\Kunder,\Delta})$). Now theorem \ref{theo multivariable TS tate rings} can be applied (because of theorem \ref{theo multivariable TS tate rings} and proposition \ref{prop At satisfies CTS}) and gives us an étale $(\phi_\Delta,\Gamma_{\Lunder,\Delta})$-module over $\A_{\Lunder,\Delta}^{\dagger}$ for some $\Lunder = (L_\alpha)_{\alpha \in \Delta}$ such that the $L_\alpha/K_\alpha$ are finite extensions, and which satisfy the same boundness condition for the $\Gamma_{\Lunder,\Delta}$. Inverting $p$ and taking the invariants by $H_{\Kunder}$, we obtain a $(\phi_\Delta,\Gamma_{\Kunder})$-module over $\B_{\Kunder,\Delta}^{\dagger}$ satisfying the conditions of the Proposition, and we can now apply \cite[Theorem 6.19]{KedCarZab} to conclude.
\end{proof}
\section{Multivariable de Rham representations}
Recall that in \cite[\S 4]{BriChiaMaz21}, a ring $\A_{\inf,\Delta}$ is defined, and that by Proposition \ref{Ainfdel = padiccomp of xiadic}, this ring is isomorphic to the $p$-adic completion of the completion of $\At^+ \otimes_{\Zp}\cdots\otimes_{\Zp} \At^+$ for the $(\xi_{\alpha_1},\ldots,\xi_{\alpha_\delta})$-topology. Also recall that there is a map $\theta_\Delta : \A_{\mathrm{inf},\Delta} \ra \O_{C_\Delta}$ which is a surjective $\G_{\Kunder,\Delta}$-equivariant morphism of $\Qp$-algebras $\theta_\Delta : \A_{\mathrm{inf},\Delta}\ra \O_{C_\Delta}$.

\begin{lemm}
\label{lemm Afindel = pkertheta adic completion}
    The ring $\A_{\inf,\Delta}$ is isomorphic to the $(p,\ker(\theta_\Delta))$-adic completion of $\At^+ \otimes_{\Zp}\cdots\otimes_{\Zp} \At^+$.
\end{lemm}
\begin{proof}
    Let $\Atplus_{\Delta}$ denote the $(p,\ker(\theta_\Delta))$-adic completion of $\At^+ \otimes_{\Zp}\cdots\otimes_{\Zp} \At^+$. Since $\At^+ \otimes_{\Zp}\cdots\otimes_{\Zp} \At^+$ is a subring of $\A_{\inf,\Delta}$, and since the latter is $(p,\ker(\theta_\Delta))$-adically complete, we deduce that we have an injection $\Atplus_{\Delta} \subset \A_{\inf,\Delta}$. The fact that $\Atplus_{\Delta}$ is $(p,\ker(\theta_\Delta))$-adically complete means that it is $p$-adically complete because $(p) \subset (p,\ker(\theta_\Delta))$ is finitely generated (see for example \cite[\href{https://stacks.math.columbia.edu/tag/090T}{Tag 090T}]{stacks-project}.
    Therefore we have $\At_{\Delta}^+ = \varprojlim_n(\Atplus_\Delta)/p^n$, and by construction the topology induced on $(\Atplus_\Delta)/p^n$ is the $(p,\ker(\theta_\Delta)) = (p,[\varpi_\alpha])_{\alpha \in \Delta}$-adic topology, and thus the $([\varpi_\alpha])_{\alpha \in \Delta}$-adic topology. This means that $\At_{\Delta}^+$ is the $p$-adic completion of the $([\varpi_\alpha])_{\alpha \in \Delta}$-completion of $\At^+ \otimes_{\Zp}\cdots\otimes_{\Zp} \At^+$, and we conclude by Proposition \ref{Ainfdel = padiccomp of xiadic}.
\end{proof}

The ideal $\ker(\theta_\Delta)$ of $\A_{\mathrm{inf},\Delta}$ is generated by $\{\xi_\alpha\}_{\alpha \in \Delta}$ (this is \cite[Corollary 4.4]{BriChiaMaz21}) and also by $\{\omega_\alpha\}_{\alpha \in \Delta}$ (one checks that $\xi_\alpha/\omega_\alpha$ is invertible in $\A_{\mathrm{inf},\Delta}$).

As in \cite[\S 4]{BriChiaMaz21}, we define $\BdrplusDel$ as the $\ker(\theta_\Delta)$-adic completion of $\A_{\inf,\Delta}[1/p]$, and we endow it with the so called \textit{canonical topology}, which means that each quotient $\BdrplusDel/\ker(\theta_\Delta)^m \simeq \A_{\inf,\Delta}[1/p]/\ker(\theta_\Delta)^m$ is endowed with the $p$-adic topology coming from the one on $ \A_{\inf,\Delta}$, giving the quotient a $p$-adic Banach space structure. This means that $\BdrplusDel$ equipped with the canonical topology is a Fréchet space. The following proposition shows that we can replace the definition of $\BdrplusDel$ as in \cite{BriChiaMaz21} by one which is more compatible with all our constructions so far:
\begin{prop}
\label{prop Bdrdel=prod of Bdr}
    We have an isomorphism of topological rings $\BdrplusDel \simeq \wotimes_{\Qp}^{\alpha \in \Delta}\Bdrplus$, where each $\Bdrplus$ is a copy of the classical ring of periods of Fontaine endowed with its canonical topology which makes it a Fréchet space.
\end{prop}
\begin{proof}
    By definition, we have $\BdrplusDel = \varprojlim_n\A_{\inf,\Delta}[1/p]/\ker(\theta_\Delta)^n$, and thus
    $$\BdrplusDel = \varprojlim_n((\A_{\inf,\Delta}/\ker(\theta_\Delta)^n)[1/p]),$$
    using the fact that localization and quotients commute. 
    Since $\ker(\theta_\Delta) = (\xi_\alpha)_{\alpha \in \Delta}$, we have that for every $n \geq 1$, $\ker(\theta_\Delta)^{n\delta} \subset (\xi_\alpha^n)_{\alpha \in \Delta} \subset \ker(\theta_\Delta)^n$, so that 
    $$\BdrplusDel = \varprojlim_n((\A_{\inf,\Delta}/(\xi_\alpha^n)_{\alpha \in \Delta})[1/p]).$$
    Moreover, the inclusions $\ker(\theta_\Delta)^{n\delta} \subset (\xi_\alpha^n)_{\alpha \in \Delta} \subset \ker(\theta_\Delta)^n$ imply that $\A_{\inf,\Delta}/(\xi_\alpha^n)_{\alpha \in \Delta}$ is the $p$-adic completion of $(\Atplus \otimes_{\Zp} \cdots \otimes_{\Zp}\Atplus)/(\xi_\alpha^n)_{\alpha \in \Delta}$ by lemma \ref{lemm Afindel = pkertheta adic completion}.
    By \cite[Exercise 1.3, Chapter 1]{qing2006algebraic}, we have an isomorphism
    $$(\Atplus \otimes_{\Zp} \cdots \otimes_{\Zp}\Atplus)/(\xi_\alpha^n)_{\alpha \in \Delta} \simeq \Atplus/\xi_{\alpha_1}^n \otimes_{\Zp} \cdots \otimes_{\Zp}\Atplus/\xi_{\alpha_\delta}^n,$$
    so that 
    $$\BdrplusDel = \varprojlim_n((\Atplus/\xi_{\alpha_1}^n \wotimes_{\Zp} \cdots \wotimes_{\Zp}\Atplus/\xi_{\alpha_\delta}^n)[1/p]),$$
    where the completion is taken with respect to the $p$-adic topology. Now we just note that 
    $$\varprojlim_n((\Atplus/\xi_{\alpha_1}^n \wotimes_{\Zp} \cdots \wotimes_{\Zp}\Atplus/\xi_{\alpha_\delta}^n)[1/p])=\varprojlim_n(\Atplus/\xi_{\alpha_1}^n[1/p] \wotimes_{\Qp} \cdots \wotimes_{\Qp}\Atplus/\xi_{\alpha_\delta}^n[1/p])$$
    by lemma \ref{unit ball}, and that for all $\alpha \in \Delta$ and $n \geq 1$, $(\Bdrplus)_{\alpha}/\ker(\theta_\alpha)^n \simeq \Atplus/\xi_{\alpha}^n[1/p]$, so that 
    $$\varprojlim_n(\Atplus/\xi_{\alpha_1}^n[1/p] \wotimes_{\Qp} \cdots \wotimes_{\Qp}\Atplus/\xi_{\alpha_\delta}^n[1/p]) = \wotimes_{\Qp}^{\alpha \in \Delta}\Bdrplus$$
    in the category of Fréchet spaces.
\end{proof}

We quickly recap some properties and definitions attached to the ring $\BdrplusDel$, coming from \S 4 of \cite{BriChiaMaz21}. We let $t_\Delta:=\prod_{\alpha \in \Delta}t_\alpha$ and we let $\BdrDel:=\BdrplusDel[1/t_\Delta]$. Since $\G_{\Kunder,\Delta}$ acts on $t_\Delta$ by $g(t) = \chi_{\Delta}(g)\cdot t_\Delta$, the action of $\G_{\Kunder,\Delta}$ on $\BdrplusDel$ extends to an action on $\BdrDel$. We endow $\BdrDel = \varinjlim_nt_\Delta^{-n}\BdrplusDel$ with the injective limit topology. The ring $\BdrDel$ is endowed with a filtration indexed by $\Z$ and defined by
$$\Fil^j\BdrDel = \varinjlim_{i \geq j}t_\Delta^{-i}\Fil^{i\delta+j}\BdrplusDel$$
where $\Fil^i\BdrplusDel = \ker(\theta_\Delta)^i$. The filtration thus defined on $\BdrDel$ is decreasing, separated, exhaustive, and stable by Galois. We let $\B_{\HT,\Delta}:= \gr \BdrDel$.

If $V$ is a $p$-adic representation of $\G_{\Kunder,\Delta}$, we let $\D_{\dR,\Delta}(V):= (\BdrDel \otimes_{\Qp}V)^{\G_{\Kunder,\Delta}}$, which is a $K_\Delta$-module, endowed with a filtration coming from the one on $\BdrDel$. By $\BdrDel$-linearity, the inclusion $\D_{\dR,\Delta}(V) \subset \BdrDel \otimes_{\Qp}V$ induces a $\G_{\Kunder,\Delta}$, $\BdrDel$-linear map
$$\alpha_{\dR,\Delta}(V): \BdrDel \otimes_{K_\Delta} \D_{\dR,\Delta}(V) \ra \BdrDel \otimes_{\Qp}V.$$
We define similarly a map $\alpha_{\HT,\Delta}(V)$, where $\D_{\HT,\Delta}(V) = (\B_{\HT,\Delta} \otimes_{\Qp}V)^{\G_{\Kunder,\Delta}}$. 

\begin{prop}
We have:
\begin{itemize}
    \item  $\gr^r\BdrDel \simeq \bigoplus_{\n = (n_\alpha)_{\alpha \in \Delta} \in \Z^\Delta , \sum_\alpha n_\alpha = r}C_\Delta t_\Delta^{\n}$;
    \item $\B_{\HT,\Delta}\simeq C_\Delta[t_\alpha,t_\alpha^{-1}]_{\alpha \in \Delta} \simeq \oplus_{\n \in \Z^\Delta}C_\Delta(\n)$ as $\G_{\Kunder,\Delta}$-modules.
\end{itemize}
\end{prop}
\begin{proof}
    See \cite[Prop. 4.14]{BriChiaMaz21}.
\end{proof}

Note that \cite[Proposition 4.17]{BriChiaMaz21} shows that the $K_\Delta$-modules $\D_{\dR,\Delta}(V)$ and $\D_{\HT,\Delta}(V)$ are of finite type, and that the maps $\alpha_{\dR,\Delta}(V)$ and $\alpha_{\HT,\Delta}(V)$ are injective. A representation $V$ of $\G_{\Kunder,\Delta}$ is said to be de Rham (resp. Hodge-Tate) when the map $\alpha_{\dR,\Delta}(V)$ (resp. $\alpha_{\HT,\Delta}(V)$) is bijective. 

The following proposition generalizes \cite[Proposition 4.17]{BriChiaMaz21} for Noetherian families admitting a Galois stable lattice:

\begin{prop}\label{pro de rham relative module finite}
    Let $S$ be a Noetherian $\Qp$-Banach algebra, with unit ball $\O_S$, let $T$ be a free $\O_S$-module of rank $d$, endowed with a continuous $\O_S$-linear action of $\G_{\Kunder,\Delta}$, and let $V = T[1/p]$. Then $\D_{\dR,\Delta}(V)$ and $\D_{\HT,\Delta}(V)$ are finite $S \otimes_{\Qp}K_\Delta$-modules.
\end{prop}
\begin{proof}
    We follow the proof of \cite[Prop. 4.3.2]{Bellovinsheaves}. We start by noticing that
    $$\D_{\HT,\Delta}(V) = \oplus_{\j \in \Z^\Delta}(\D_{\Sen,\Delta}(V)\cdot t_\Delta^{\j})^{\Gamma_{\Kunder,\Delta}}=\oplus_{\j \in \Z^\Delta}(\D_{\Sen,\Delta}(V))^{\Gamma_{\Kunder,\Delta}=\chi_{\Delta}^{-\j}},$$
    so that $\D_{\HT,\Delta}(V) \subset \D_{\Sen,\Delta}(V)$ which is a finite module over $S \otimes_{\Qp}K_\Delta$, which is Noetherian since $S$ is a Noetherian $\Qp$-algebra and $K_\Delta$ is of finite type over $\Qp$. This means that $\D_{\HT,\Delta}(V)$ is a finite $S \otimes_{\Qp}K_\Delta$-module.

    For the proof regarding $\D_{\dR,\Delta}(V)$ we can follow the proof of \cite[Prop. 4.3.2]{Bellovinsheaves} directly, as the arguments are the same.
\end{proof}

Given a $p$-adic representation $V$ of $\G_{\Kunder,\Delta}$, the authors of \cite{BriChiaMaz21} define in \S 5 of ibid a module attached to $V$ which we will denote by $\D_{\dif,\Delta}^+(V)$ as follows: $\D_{\dif,\Delta}^+(V)=\varprojlim_r((\BdrplusDel \otimes_{\Qp}V)^{H_{\Kunder,\Delta}})/\Fil^r)^{\fin}$, which means that we're taking the inverse limit with respect to the filtration induced by $\BdrplusDel$ of the finite vectors of the quotients by this filtration. Moreover, this is a $K_{\infty,\Delta}[\![t_\alpha]\!]_{\alpha \in \Delta}$-module, and if $(\BdrplusDel \otimes_{\Qp}V)^{H_{\Kunder,\Delta}}$ is free of rank $d$ over $(\BdrplusDel)^{H_{\Kunder,\Delta}}$ then it is free of rank $d$ over $K_{\infty,\Delta}[\![t_\alpha]\!]_{\alpha \in \Delta}$. 

More generally, given a free $\BdrplusDel$-representation $W$ of $\G_{\Kunder,\Delta}$, we can attach to $W$ a $K_{\infty,\Delta}[\![t_\alpha]\!]_{\alpha \in \Delta}$-module 
$$W_f:=\varprojlim_r(W^{H_{\Kunder,\Delta}})/\Fil^r)^{\fin}$$
which is free of rank $d$ over $K_{\infty,\Delta}[\![t_\alpha]\!]_{\alpha \in \Delta}$.

Theorem 5.11 of ibid shows that the functor $W \mapsto W_f$ is an equivalence of categories between free $\BdrplusDel$-representations of $\Gamma_{\Kunder,\Delta}$ of finite rank and free $K_{\infty,\Delta}[\![t_\alpha]\!]_{\alpha \in \Delta}$-representations of $\Gamma_{\Kunder,\Delta}$ of finite rank. 

As in \S \ref{subs multi Sen}, we can recover those constructions using locally and pro-analytic vectors.

\begin{lemm}
    We have $((\BdrplusDel)^{H_{\Kunder,\Delta}})^{\Gamma_{\Kunder,\Delta}-\pa} = K_{\infty,\Delta}[\![t_\alpha]\!]_{\alpha \in \Delta}$.
\end{lemm}
\begin{proof}
    By proposition \ref{prop Bdrdel=prod of Bdr}, lemma \ref{inv and tensors} and corollary \ref{coro computes locanamulti}, we have that
    \[((\BdrplusDel)^{H_{\Kunder,\Delta}})^{\Gamma_{\Kunder,\Delta}-\pa}=(\Bdrplus)=\wotimes_{\Qp}^{\alpha \in \Delta}(\Bdrplus)^{H_{K_\alpha}})^{\Gamma_{K_\alpha}-\pa}. \]
    By \cite[Prop. 2.6]{PoratHilbert}, we have that $((\Bdrplus)^{H_K})^{\Gamma_K-\pa} = K_\infty[\![t]\!]$ which finishes the proof.
\end{proof}

\begin{lemm}
    If $W$ is a free $(\BdrplusDel)^{H_{\Kunder,\Delta}}$-representation of $\Gamma_{\Kunder,\Delta}$ of finite rank, then $W_f = W^{\Gamma_{\Kunder,\Delta}-\pa}$.
\end{lemm}
\begin{proof}
    Since $\BdrplusDel$ is a Fréchet space whose Fréchet topology is given by the Banach topology on each $\BdrplusDel/\Fil^r\BdrplusDel$, we have that $W^{\Gamma_{\Kunder,\Delta}-\pa} = \varprojlim_r(W/\Fil^rW)^{\la}$ by definition of pro-analytic vectors. By Proposition \ref{prop la recovers fin Sen CDelta}, each $(W/\Fil^rW)^{\la}$ is equal to $(W/\Fil^rW)^{\fin}$ so that we recover the definition of $W_f$.
\end{proof}

As a direct consequence of this lemma, we obtain the following corollary, which shows that pro-analytic vectors recover Sen theory for $\BdrplusDel$-representations:
\begin{coro}
    If $V$ is a $p$-adic representation of $\G_{\Kunder,\Delta}$ then $\D_{\dif,\Delta}^+(V)=((\BdrplusDel \otimes_{\Qp}V)^{H_{\Kunder,\Delta}})^{\Gamma_{\Kunder,\Delta}-\pa}$.
\end{coro}

\section{Applications to crystalline and semistable representations}

Let now $A$ be an affinoid algebra over $\Qp$ and let $\O_A$ be the valuation ring of $A$. Let $V$ be a free representation of $\G_{\underline{K},\Delta}$ over $A$, such that $V$ contains an $\O_A$-lattice of rank $d$, stable for the action of $\G_{\Kunder,\Delta}$. Then by Theorem \ref{theo overconv in fam} we can associate to $V$ an overconvergent $(\phi_{\Delta},\Gamma_{\underline{K},\Delta})$-module over $A\wotimes_{\Qp}\B^{\dagger,\r}_{\underline{K}}$, for some $\r\in\N[1/p]^{\Delta}$, that we will call $\mathbf{D}^{\dagger,\r}_{\Delta}(V)$. We define
$$\mathbf{D}^{\dagger}_{\Delta}(V)=\bigcup_{\r\in\N[1/p]^{\delta}}\mathbf{D}^{\dagger,\r}_{\Delta}(V).$$
We set
$$\mathbf{D}_{\rm{rig},\underline{K},\Delta}^{\dagger,\r}(V)=(A\wotimes_{\Qp}\B^{\dagger,\r}_{\rm{rig},\underline{K},\Delta})\wotimes_{\mathbf{B}^{\dagger,\underline{r}}_{\underline{K},\Delta}}\mathbf{D}^{\dagger,\r}_{\Delta}(V),$$
$$\mathbf{D}_{\rm{rig},\underline{K},\Delta}^{\dagger}(V)=\bigcup_{\r\in\N[1/p]^{\delta}}\mathbf{D}_{\rm{rig},\underline{K},\Delta}^{\dagger,\r}(V)$$
$$\mathbf{D}_{\rm{log},\underline{K},\Delta}^{\dagger,\r}(V)=(A\wotimes_{\Qp}\B^{\dagger,\r}_{\rm{log},\underline{K}})\wotimes_{\mathbf{B}^{\dagger,\underline{r}}_{\underline{K}}}\mathbf{D}^{\dagger,\r}_{\Delta}(V),$$
and
$$\mathbf{D}_{\rm{log},\underline{K},\Delta}^{\dagger}(V)=\bigcup_{\r\in\N[1/p]^{\delta}}\mathbf{D}_{\rm{log},\underline{K},\Delta}^{\dagger,\r}(V).$$

In this section we will prove that 
\begin{equation}\label{comp crys phigamma}
    \mathbf{D}_{\mathrm{crys},\Delta}(V)=(\mathbf{D}^{\dagger}_{\mathrm{rig},\Delta}(V)[1/t_{\Delta}])^{\Gamma_{\Kunder,\Delta}},
\end{equation}
and 
\begin{equation}\label{comp crys phigamma}
    \mathbf{D}_{\mathrm{st},\Delta}(V)=(\mathbf{D}^{\dagger}_{\mathrm{log},\Delta}(V)[1/t_{\Delta}])^{\Gamma_{\Kunder,\Delta}},
\end{equation}
This is a multivariable generalization of the classical work of \cite[Théorème 0.2]{Ber02} and \cite[Theorem 4.2.9]{Bellovinsheaves}. It is analogous to the result obtained in \cite[Theorem 5.23]{BriChiaMaz21}, that relates multivariable $(\phi,\Gamma)$-modules and multivariable de Rham representations.\\

To compare $\DcrysDel$ and $\DrigDel$ we will pass from the module $(V\otimes_{\Qp}(A\wotimes_{\Qp}\BtrigDel))^{\G_{\Kunder,\Delta}}$. To do so one can use the following

\begin{prop}
\label{prop frobreg Arig}
Let $A$ be a $\Qp$-Banach algebra, and let $\O_A$ be its valuation ring. Then $\cap_{k=0}^{+\infty}p^{-hk}(\O_A\hat{\otimes}_{\Zp}\At_{\rig}^{\dagger,p^{-k}s}) \subset A \hat{\otimes}_{\Qp}\Btrigplus$. 
\end{prop}
\begin{proof}
This is \cite[Coro. 4.2.12]{Bellovinsheaves}.
\end{proof}

As in \cite[Prop. 3.2]{Ber02} and \cite[Prop. 4.2.13]{Bellovinsheaves}, the following ``Frobenius regularization'' is a direct consequence from Proposition \ref{prop frobreg Arig}, though we cannot apply either of \cite[Prop. 3.2]{Ber02} nor \cite[Prop. 4.2.13]{Bellovinsheaves} directly.

\begin{prop}\label{prop Frobenius reg}
Let $d_1$ and $d_2$ be two positive integers. Let $A$ be a $\Qp$-Banach algebra. Let $M \in \M_{d_2\times d_1}(A \hat{\otimes}_{\Qp}\Bt_{\log,\Delta}^\dagger)$ and suppose that for all $i \in \Delta$, there exists $h_i \geq 1$ and $P_i \in \GL_{d_1}(A \otimes_{\Qp}\underline{F}_{\Delta})$ such that $M=\phi_i^{-h_i}(M)P_i$. Then $M \in \M_{d_2\times d_1}(A\hat{\otimes}_{\Qp}\Bt_{\log,\Delta}^+)$. 
\end{prop}
\begin{proof}
We use exactly the same strategy as in the proof of \cite[Prop. 3.2]{Ber02} but since we are in the multivariable case we need to use it $\delta$ times. 

Let us write $P_\delta=(p_{ij})$, $M = (m_{ij})$, and $m_{ij} = \sum_{\underline{k}\ll\infty} m_{ij,\underline{k}}\boldsymbol{Y}^{\underline{k}}$. Let $\r=(r_1,\ldots,r_{\delta})$ be such that all the $m_{ij,\underline{k}}$ belong to $A\hat{\otimes}_{\Qp}\Bt_{\rig,\Delta}^{\dagger,\r}$. Let $I_1,\ldots,I_{\delta-1}$ be respectively compact subintervals of $[r_i,+\infty[$, with $\min(I_i)=r_i$, for $i=1,\dots,\delta$. We consider the $m_{ij,\underline{k}}$ as elements of $A\hat{\otimes}_{\Qp}\Bt^{I_1}\hat{\otimes}_{\Qp}\ldots\hat{\otimes}_{\Qp}\Bt^{I_{\delta-1}}\hat{\otimes}_{\Qp}\Bt_{\rig}^{\dagger,r_\delta}$. Note that $A_\delta:=A\hat{\otimes}_{\Qp}\Bt^{I_1}\hat{\otimes}_{\Qp}\ldots\hat{\otimes}_{\Qp}\Bt^{I_{\delta-1}}$ is a $\Qp$-Banach algebra as every algebra considered in the tensor product is a Banach algebra. We let $\O_{A_\delta}$ denote its valuation ring.

Let $f_0 \in \Z$ be such that $p^{h_0}P \in \M_{d_1}(\O_{A_\delta}\otimes_{\Zp} \O_{\underline{F}_\Delta})$. Let $s_\delta\in\Z$ be such that the degree of the $m_{ij}$ in $Y_\delta$ are bounded above by $s_\delta$, and let $f = f_0+h_\delta\cdot s_\delta$. The relation between $M$ and $P_\delta$ translates into
$$\phi_\delta^{-h_\delta}(m_{i1})p_{1j}+\ldots+\phi_\delta^{-h_\delta}(m_{id_1})p_{d_2j}=m_{ij} $$
for all $i \leq d_1, j \leq d_2$. Since $\phi_\delta^{-h_\delta}(\mathbf{Y}^{\underline{k}}) = p^{-h_\delta\cdot k_\delta}\mathbf{Y}^{\underline{k}}$, we obtain that if $m_{ij} \in p^{-c}\O_{A_{\delta}}\hat{\otimes}_{\Zp}\At_{\rig}^{\dagger,r_\delta}$, then we have $p^{f_0}p_{ij}  \in \O_{A_{\delta}}\otimes_{\Zp} \O_{\underline{F}_\Delta}$, and $\phi_\delta^{-h_\delta}(m_{ij,\underline{k}}) \in p^{-c}\O_{A_{\delta}} \hat{\otimes}_{\Zp}\At_{\rig}^{\dagger,r_\delta/p^{h_\delta}}$ so that $m_{ij,\underline{k}} \in p^{-f-c}\O_{A_{\delta}} \hat{\otimes}_{\Zp}\At_{\rig}^{\dagger,r_\delta/p^{h_\delta}}$. By iterating, we find that the $m_{ij,\underline{k}}$ belong to $\cap_{k=0}^{+\infty}p^{-fk-c}\O_{A_\delta} \hat{\otimes}_{\Zp}\At_{\rig}^{\dagger,r_\delta p^{-h_\delta k}}$. By proposition \ref{prop frobreg Arig}, this means that the $m_{ij,\underline{k}}$ actually belong to $A\hat{\otimes}_{\Qp}\Bt^{I_1}\hat{\otimes}_{\Qp}\ldots\hat{\otimes}_{\Qp}\Bt^{I_{\delta-1}}\hat{\otimes}\Btrigplus$. Since this is true for all $I_i$ compact subintervals of $[r_i,+\infty[$, with $\min(I_i)=r_i$, $i=1,\dots,\delta$, this means by \cite[Proposition 1.1.29]{Emer} that the $m_{ij,\underline{k}}$ belong to $A\hat{\otimes}_{\Qp}\Bt_{\rig,\Delta}^{\dagger,\r}\hat{\otimes}_{\Qp}\Btrigplus$. 

We apply then the same reasoning to each of the components, so that in the end the $m_{ij,\k}$ belong to $A\hat{\otimes}_{\Qp}\Bt_{\rig,\Delta}^+$, and thus this concludes the proof.
\end{proof}
One can then use Frobenius regularity to compare periods over $A\wotimes_{\Qp}\BtlogplusDel$ to periods over $A\wotimes_{\Qp}\BtlogDel$.
\begin{lemm}
    Let $A$ be a complete DVF over $\Qp$, with perfect residue field. We have an equality
    \begin{equation*}
        ((A\wotimes_{\Qp}\BtlogplusDel)\wotimes_{A}V)^{\G_{\underline{K},\Delta}}=((A\wotimes_{\Qp}\BtlogDel)\wotimes_{A}V)^{\G_{\underline{K},\Delta}}.
    \end{equation*}
\end{lemm}
\begin{proof}

   Using Frobenius regularization (Proposition \ref{prop Frobenius reg}), the proof is the same as the one of \cite[Proposition 4.2.16]{Bellovinsheaves}
\end{proof}
Using the lemma above one can prove the following lemma as in \cite[Corollary 4.2.18]{Bellovinsheaves}.
\begin{prop}\label{comparison invariants Blogplus Blogrig}
    Let $A$ be an affinoid algebra over $\Qp$. We have an isomorphism
    \begin{equation*}
        ((A\wotimes_{\Qp}\BtlogplusDel)\wotimes_{A}V)^{\G_{\underline{K},\Delta}}=((A\wotimes_{\Qp}\BtlogDel)\wotimes_{A}V)^{\G_{\underline{K},\Delta}}.
    \end{equation*}
\end{prop}
The following theorem explicates the relation between multivariable crystalline representations and multivariable $(\phi,\Gamma)$-modules. It is an analogue of \cite[Proposition 4.2.19]{Bellovinsheaves}. 
\begin{thm}
    In the situation above the natural map
    $$(\mathbf{D}_{\rm{log},\underline{K},\Delta}^{\dagger}(V)[1/t_{\Delta}])^{\Gamma_{\underline{K},\Delta}}\rightarrow((A\wotimes_{\Qp}\Bt_{\log,\Delta}^\dagger)\wotimes_AV[1/t_{\Delta}])$$ is an isomorphism.
\end{thm}
\begin{proof}
    Replacing $K_{\alpha}$, for $\alpha\in\Delta$, with finite extensions we may assume that $\mathbf{D}_{\rm{log},\underline{K},\Delta}^{\dagger}$ is free. After multiplying $V$ by a suitable power of $t_{\Delta}$ we may assume that 
    $$(\mathbf{D}_{\rm{log},\underline{K},\Delta}^{\dagger}(V)[1/t_{\Delta}])^{\Gamma_{\underline{K},\Delta}}=(\mathbf{D}_{\rm{log},\underline{K},\Delta}^{\dagger}(V))^{\Gamma_{\underline{K},\Delta}}.$$
    Note moreover that
    $$((A\wotimes_{\Qp}\Bt_{\log,\Delta}^\dagger)\wotimes_AV)^{H_{\underline{K},\Delta}}=((A\wotimes_{\Qp}\Bt_{\log,\Delta}^\dagger)\wotimes_{A\wotimes_{\Qp}\B^{\dagger}_{\rm{log},\underline{K},\Delta}}\mathbf{D}_{\rm{log},\underline{K},\Delta}^{\dagger}(V))$$
    by overconvergence.
    If we pove that 
    $$((A\wotimes_{\Qp}\Bt_{\log,\Delta}^\dagger)\wotimes_{A\wotimes_{\Qp}\B^{\dagger}_{\rm{log},\underline{K},\Delta}}\mathbf{D}_{\rm{log},\underline{K},\Delta}^{\dagger}(V))^{\Gamma_{\underline{K},\Delta}}=(\mathbf{D}_{\rm{log},\underline{K},\Delta}^{\dagger}(V))^{\Gamma_{\underline{K},\Delta}}$$
    thus we get the claim.
    Taking $\Gamma_{\underline{K},\Delta}$-pro-analytic vectors, and using Proposition 5.1 of \cite{Poyetonlocanaperiods} and Corollary \ref{coro computes locanamulti}, we get
    \begin{dmath*}((A\wotimes_{\Qp}\Bt_{\log,\Delta}^{\dagger,\r})\wotimes_{A\wotimes_{\Qp}\B^{\dagger,\r}_{\rm{log},\underline{K},\Delta}}\mathbf{D}_{\rm{log},\underline{K},\Delta}^{\dagger}(V))^{\Gamma_{\underline{K},\Delta}}=(\bigcup_{n\in\N}\phi_{\Delta}^{-n}(A\wotimes_{\Qp}\B_{\log,\underline{K},\Delta}^{\dagger, p^{n\delta}\r})\wotimes_{A\wotimes_{\Qp}\B^{\dagger,\r}_{\rm{log},\underline{K},\Delta}}\mathbf{D}_{\rm{log},\underline{K},\Delta}^{\dagger,\r}(V))^{\Gamma_{\underline{K},\Delta}}.\end{dmath*}
     Since $((A\wotimes_{\Qp}\Bt_{\log,\Delta}^\dagger)\wotimes_{A\wotimes_{\Qp}\B^{\dagger}_{\rm{log},\underline{K},\Delta}}\mathbf{D}_{\rm{log},\underline{K},\Delta}^{\dagger}(V))^{\Gamma_{\underline{K},\Delta}}$ is a finite $A$-module (as $\D_{\rm{dR}(V)}$ is finite by Proposition \ref{pro de rham relative module finite}), it will be contained in $$(\phi_{\Delta}^{-n}(A\wotimes_{\Qp}\B_{\log,\underline{K},\Delta}^{\dagger, p^{n\delta}\s})\wotimes_{A\wotimes_{\Qp}\B^{\dagger,\r}_{\rm{log},\underline{K},\Delta}}\mathbf{D}_{\rm{log},\underline{K},\Delta}^{\dagger,\s}(V))^{\Gamma_{\underline{K},\Delta}}$$ for some $\s\in\N[1/p]^{\delta}$. But $\phi^{n}_{\Delta}$ induces an isomorphism
     \begin{dmath*}
     (\phi_{\Delta}^{-n}(A\wotimes_{\Qp}\B_{\log,\underline{K},\Delta}^{\dagger, p^{n\delta}\s})\wotimes_{A\wotimes_{\Qp}\B^{\dagger,\r}_{\rm{log},\underline{K},\Delta}}\mathbf{D}_{\rm{log},\underline{K},\Delta}^{\dagger,\s}(V))^{\Gamma_{\underline{K},\Delta}}=(\phi_{\Delta}^{-n}(A\wotimes_{\Qp}\B_{\log,\underline{K},\Delta}^{\dagger, p^{n\delta}\s})\wotimes_{A\wotimes_{\Qp}\B^{\dagger,\r}_{\rm{log},\underline{K},\Delta}}\mathbf{D}_{\rm{log},\underline{K},\Delta}^{\dagger,\s}(V))^{\Gamma_{\underline{K},\Delta}}.\end{dmath*}
     As $((A\wotimes_{\Qp}\Bt_{\log,\Delta}^\dagger)\wotimes_{A\wotimes_{\Qp}\B^{\dagger}_{\rm{log},\underline{K},\Delta}}\mathbf{D}_{\rm{log},\underline{K},\Delta}^{\dagger}(V))^{\Gamma_{\underline{K},\Delta}}$ is stable by $\phi_{\Delta}$ we can conclude. 
\end{proof}
Combining the theorem above with Propostition \ref{comparison invariants Blogplus Blogrig} we obtain what we wanted to prove.
\begin{cor}
    Let $V$ be a representation of $\G_{\underline{K},\Delta}$ over $A$ admitting an invariant $\O_A$-lattice. Then 
    \begin{equation*}\label{comp crys phigamma}
    \mathbf{D}_{\mathrm{crys},\Delta}(V)\cong(\mathbf{D}^{\dagger}_{\mathrm{rig},\Delta}(V)[1/t_{\Delta}])^{\Gamma_{\underline{K},\Delta}},
\end{equation*}
and 
\begin{equation*}\label{comp crys phigamma}
    \mathbf{D}_{\mathrm{st},\Delta}(V)\cong(\mathbf{D}^{\dagger}_{\mathrm{log},\Delta}(V)[1/t_{\Delta}])^{\Gamma_{\underline{K},\Delta}}.
\end{equation*}
\end{cor}
\begin{proof}
    The isomorphism for $\mathbf{D}_{\mathrm{st},\Delta}(V)$ is clear. The one for $\mathbf{D}_{\mathrm{crys},\Delta}(V)$ follows taking the elements annihilated by the monodromy operators induced variable by variable.
\end{proof}
\section{Multivariable rings of periods and admissibility}
In this section we discuss the link between admissibility of representations of $\G_{\Kunder,\Delta}$ for multivariable rings of periods, and admissibility for the restriction of the representations of $\G_{\Kunder,\Delta}$ to $\G_{K_\alpha}$ for classical rings of periods, and give some applications.
Let $V$ be a $p$-adic representation of $\G_{\Kunder,\Delta}$ and let $B$ be one of the period rings $\B_{\max}, \B_{\mathrm{st}}$ or $\B_{\dR}$. We use $\B_{\max}$ here instead of the ring $\B_{\crys}$ to treat crystalline representations as the topology of $\B_{\crys}$ is \say{bad} (cfr. \cite[p. 24]{Col98}). Instead $\B_{\max}$ is a strict LF-space, so its completed tensor product is \say{well-behaved}. We can call $\B_{\max}$-admissible representations \emph{crystalline} as the notion of $\B_{\max}$-admissibility and $\B_{\crys}$-admissibility coincide (see ibid.). We denote by $B_{\Delta}$ the completed tensor product over $\Qp$ of $\delta$ copies of $B$. Moreover we set $L_{\alpha}=B^{G_{K_\alpha}}$, for $\alpha\in\Delta$.
We say that $V$ is \emph{admissible} if the map
$$(V\otimes_{\Qp}B_{\Delta})^{\G_{\Kunder,\Delta}}\wotimes_{K_{\Delta}}B_{\Delta}\to V\wotimes_{\Qp} B_{\Delta},$$
induced by the inclusion
$$(V\otimes_{\Qp}B_{\Delta})^{\G_{\Kunder,\Delta}}\subset V\wotimes_{\Qp} B_{\Delta}$$
is an isomorphism.

We recall the following result, which is \cite[Proposition 2.3]{KedlayadeRhamprop}:

\begin{theo}
    Let $B$ be one of the rings $\B_{\mathrm{max}}, \B_{\mathrm{st}}$ or $\B_{\dR}$. Let $V$ be a $p$-adic representation of $\G_{\Kunder,\Delta}$ such that, for each $\alpha \in \Delta$, the restriction of $V$ to $\G_{K_\alpha}$ is $B$-admissible. Then $V$ is $B \otimes_{\Qp} \cdots \otimes_{\Qp}B$-admissible.
\end{theo}

In the other direction, we can prove the following:

\begin{prop}
    Let $B$ be a classical ring of $p$-adic periods (either $\B_{\mathrm{max}}, \B_{\mathrm{st}}$ or $\B_{\dR}$), and let $B_{\Delta}$ denote the completion (for the canonical topology on $B$) of $B \wotimes_{\Qp} \cdots \wotimes_{\Qp}B$. Let $V$ be a $p$-adic representation of $\G_{\Kunder,\Delta}$, such that the $(B_{\Delta})^{\G_{\Kunder,\Delta}}$-module $(V \otimes_{\Qp}B_\Delta)^{\G_{\Kunder,\Delta}}$ is free of rank $\dim_{\Qp}(V)$. Then for each $\alpha \in \Delta$, the restriction of $V$ to $\G_{K_\alpha}$ is $B$-admissible.
\end{prop}
\begin{proof}
    Up to permutation of the factors in $B_{\Delta}$ and $G_{\Kunder,\Delta}$, we can assume $\alpha=1$.
    We have that $$(V\otimes_{\Qp}B_{\Delta})^{\G_{\Kunder,\Delta}}\simeq ((((V\otimes_{\Qp} B)^{\G_{K_{\alpha_1}}}\wotimes_{\Qp}B)^{\G_{K_{\alpha_2}}}\dots)\wotimes_{\Qp}B)^{\G_{K_{\alpha_\delta}}}.$$
    We know that if $W$ is a $\G_{K_\alpha}$ representation of rank $n$, for $\alpha\in\Delta$, we have that 
    $(W\otimes_{\Qp}B)^{\G_{K_\alpha}}$ is finite dimensional over $L_\alpha$ of rank less or equal than $n$, and the equality holds if and only if $W$ is $B$-admissible.
    This implies that 
    $$(V\otimes_{\Qp}B_{\Delta})^{\G_{\Kunder,\Delta}}\simeq ((((V\otimes_{\Qp} B)^{\G_{K_{\alpha_1}}}\otimes_{\Qp}B)^{\G_{K_{\alpha_2}}}\dots)\otimes_{\Qp}B)^{\G_{K_{\alpha_\delta}}}.$$
    Suppose that the restriction of $V$ to $\G_{K_{\alpha_1}}$ is not $B$-admissible. This implies that 
    $(V\otimes_{\Qp}B)^{\G_{K_{\alpha_1}}}$ is a vector space over $K_{\alpha_1}$ of rank less than $d$.
    Thus $((V\otimes_{\Qp} B)^{\G_{K_{\alpha_1}}}\otimes_{\Qp}B)^{\G_{K_{\alpha_2}}}$ is a $L_{\alpha_2}$-vector space of rank less than $[L_{\alpha_1}:\Qp]d$. Inductively we obtain that $(V\otimes_{\Qp}B_{\Delta})^{\G_{\Kunder,\Delta}}$ is a $\Qp$-vector space of dimension less than $d[L_{\alpha_1}:\Qp]\dots[L_{\alpha_\delta}:\Qp]$. But this contradicts the hypothesis that $(V\otimes_{\Qp}B_{\Delta})^{\G_{\Kunder,\Delta}}$ is free of rank $d$ over $L_{\alpha_1}\otimes_{\Qp}\dots\otimes_{\Qp}L_{\alpha_\delta}$.
\end{proof}

In particular, combining both results above with proposition \ref{prop Bdrdel=prod of Bdr}, we deduce the following corollary, which in the case $B = \Bdr$ answers the question asked at the end of \cite[\S 4]{BriChiaMaz21}: 

\begin{coro}
    Let $V$ be a $p$-adic representation of $\G_{\Kunder,\Delta}$, such that the $(B_{\Delta})^{\G_{\Kunder,\Delta}}$-module $(V \otimes_{\Qp}B_\Delta)^{\G_{\Kunder,\Delta}}$ is free of rank $\dim_{\Qp}(V)$. Then $V$ is $B$-admissible.
\end{coro}

Note that, if $V$ is a $p$-adic representation of $\G_{\Kunder,\Delta}$, then the $F_1 \otimes_{\Qp} \cdots \otimes_{\Qp} F_\delta$-module $\D_{\st,\Delta}(V):=(\B_{\st,\Delta} \otimes_{\Qp}V)^{\G_{\Kunder,\Delta}}$ is equipped with a filtration indexed by $\Z^\delta$ induced by the one on $\D_{\dR,\Delta}(V)$ and which can be split into partial filtration $(\Fil_\alpha)_{\alpha \in \Delta}$ and with partial Frobenii $(\phi_\alpha)_{\alpha \in \Delta}$ and partial monodromy operators $(N_\alpha)_{\alpha \in \Delta}$. 

In the classical setting, those objects are called filtered $(\phi,N)$-modules, and one can attach to them a Hodge polygon (coming from the filtration) and a Newton polygon (coming from the Frobenius structure). A filtered $(\phi,N)$-module $\D$ is said to be admissible if there exists a semistable $p$-adic representation $V$ such that $\D = \D_{\st}(V)$, and is said to be weakly admissible if some technical condition regarding the way the Frobenius structure and the filtration is met (if for any submodule stable by $\phi$ and $N$, the Newton polygon of the submodule lies above the Hodge polygon attached to the induced filtration). The main theorem of \cite{CF00} shows that weakly admissibility and admissibility are equivalent.

Here, we can define admissibility the same way: a filtered $(\phi_\alpha,N_\alpha)_{\alpha \in \Delta}$-module $\D$ over $F_1 \otimes_{\Qp} \cdots \otimes_{\Qp}F_\delta$ is admissible if and only if there exists a semistable representation $V$ of $\G_{\Kunder,\Delta}$ such that $\D = \D_{\st,\Delta}(V)$. Because of our corollary, it then makes sense to define weakly admissibility as follows: a filtered $(\phi_\alpha,N_\alpha)_{\alpha \in \Delta}$-module $\D$ over $F_1 \otimes_{\Qp} \cdots \otimes_{\Qp}F_\delta$ is weakly admissible if and only if for each submodule stable by the operators $(\phi_\beta)_{\beta \in \Delta}$ and $(N_\beta)_{\beta \in \Delta}$, and for each $\alpha \in \Delta$, the Hodge polygon of index $\alpha$ induced by the global filtration on this submodule lies below the Newton polygon of index $\alpha$. In this setting however, it does not seem clear wether or not admissibility and weakly admissibility are still equivalent.

\bibliographystyle{alpha}
\bibliography{bibli}

\begin{thebibliography}{{Sta}24}

\bibitem[BC08]{BC08}
Laurent Berger and Pierre Colmez.
\newblock Familles de repr{\'e}sentations de de {R}ham et monodromie
  $p$-adique.
\newblock {\em Ast{\'e}risque}, 319:303--337, 2008.

\bibitem[BC16]{Ber14SenLa}
Laurent Berger and Pierre Colmez.
\newblock Th\'eorie de {S}en et vecteurs localement analytiques.
\newblock {\em Ann. Sci. \'Ec. Norm. Sup\'er. (4)}, 49(4):947--970, 2016.

\bibitem[BCM24]{BriChiaMaz21}
Olivier Brinon, Bruno Chiarellotto, and Nicola Mazzari.
\newblock Multivariable de {R}ham representations, {S}en theory and $p$-adic
  differential equations.
\newblock {\em Math. Research Letters}, 31(1):25--90, 2024.

\bibitem[Bel15]{Bellovinsheaves}
Rebecca Bellovin.
\newblock {$p$}-adic {H}odge theory in rigid analytic families.
\newblock {\em Algebra Number Theory}, 9(2):371--433, 2015.

\bibitem[Ber02]{Ber02}
Laurent Berger.
\newblock Repr\'{e}sentations {$p$}-adiques et \'{e}quations
  diff\'{e}rentielles.
\newblock {\em Invent. Math.}, 148(2):219--284, 2002.

\bibitem[Ber16]{Ber14MultiLa}
Laurent Berger.
\newblock Multivariable {$(\varphi,\Gamma)$}-modules and locally analytic
  vectors.
\newblock {\em Duke Math. J.}, 165(18):3567--3595, 2016.

\bibitem[BGR84]{BoschAna}
S.~Bosch, U.~G\"{u}ntzer, and R.~Remmert.
\newblock {\em Non-{A}rchimedean analysis}, volume 261 of {\em Grundlehren der
  mathematischen Wissenschaften [Fundamental Principles of Mathematical
  Sciences]}.
\newblock Springer-Verlag, Berlin, 1984.
\newblock A systematic approach to rigid analytic geometry.

\bibitem[CC98]{cherbonnier1998representations}
Fr{\'e}d{\'e}ric Cherbonnier and Pierre Colmez.
\newblock Repr{\'e}sentations p-adiques surconvergentes.
\newblock {\em Inventiones mathematicae}, 133(3):581--611, 1998.

\bibitem[CF00]{CF00}
Pierre Colmez and Jean-Marc Fontaine.
\newblock Construction des repr\'esentations {$p$}-adiques semi-stables.
\newblock {\em Invent. Math.}, 140(1):1--43, 2000.

\bibitem[CKZ21]{KedCarZab}
Annie Carter, Kiran~S. Kedlaya, and Gergely Z\'abr\'adi.
\newblock Drinfeld's lemma for perfectoid spaces and overconvergence of
  multivariate {$(\varphi,\Gamma)$}-modules.
\newblock {\em Doc. Math.}, 26:1329--1393, 2021.

\bibitem[Col98]{Col98}
Pierre Colmez.
\newblock Th{\'e}orie d'{I}wasawa des repr{\'e}sentations de de {R}ham d'un
  corps local.
\newblock {\em Annals of Mathematics}, pages 485--571, 1998.

\bibitem[Col08]{colmez2008espaces}
Pierre Colmez.
\newblock Espaces {V}ectoriels de dimension finie et repr{\'e}sentations de de
  {R}ham.
\newblock {\em Ast{\'e}risque}, 319:117--186, 2008.

\bibitem[Dri80]{drinfeld1980langlands}
Vladimir~Gershonovich Drinfeld.
\newblock Langlands conjecture for $\mathrm{GL}(2)$ over functional fields.
\newblock In {\em Proceedings of the International Congress of Mathematicians
  (Helsinki, 1978)}, volume~2, pages 565--574, 1980.

\bibitem[Eme04]{Emer}
Matthew Emerton.
\newblock Locally analytic vectors in representations of locally $p$-adic
  analytic groups.
\newblock {\em Memoirs of the American Mathematical Society}, 248, 06 2004.

\bibitem[Fon90]{Fon90}
Jean-Marc Fontaine.
\newblock Repr{\'e}sentations p-adiques des corps locaux (1{\`e}re partie).
\newblock In {\em The Grothendieck Festschrift}, pages 249--309. Springer,
  1990.

\bibitem[Fon94]{fontaine1994corps}
Jean-Marc Fontaine.
\newblock Le corps des p{\'e}riodes $p$-adiques.
\newblock {\em Ast{\'e}risque}, (223):59--102, 1994.

\bibitem[Hub93]{huber1993continuous}
Roland Huber.
\newblock Continuous valuations.
\newblock {\em Math. Z}, 212(3):455--477, 1993.

\bibitem[Kap18]{kaplansky2018infinite}
Irving Kaplansky.
\newblock {\em Infinite abelian groups}.
\newblock Courier Dover Publications, 2018.

\bibitem[Ked15]{kedlaya2015new}
Kiran~S Kedlaya.
\newblock New methods for {$(\varphi,\Gamma)$}-modules.
\newblock {\em Research in the Mathematical Sciences}, 2:1--31, 2015.

\bibitem[Ked23]{KedlayadeRhamprop}
Kiran Kedlaya.
\newblock The de {R}ham property for representations of products of {G}alois
  groups.
\newblock \url{https://kskedlaya.org/papers/multi-derham.pdf}, 2023.

\bibitem[Laf97]{lafforgue1997chtoucas}
Laurent Lafforgue.
\newblock {\em Chtoucas de {D}rinfeld et conjecture de
  {R}amanujan-{P}etersson}.
\newblock Soci{\'e}t{\'e} math{\'e}matique de France, 1997.

\bibitem[Laf14]{lafforgue2014introduction}
Vincent Lafforgue.
\newblock Introduction to chtoucas for reductive groups and to the global
  {L}anglands parameterization.
\newblock {\em arXiv preprint arXiv:1404.6416}, 2014.

\bibitem[Laf18]{lafforgue2018chtoucas}
Vincent Lafforgue.
\newblock Chtoucas pour les groupes r{\'e}ductifs et param{\'e}trisation de
  {L}anglands globale.
\newblock {\em Journal of the American Mathematical Society}, 31(3):719--891,
  2018.

\bibitem[Mat95]{matsuda1995local}
Shigeki Matsuda.
\newblock Local indices of $p$-adic differential operators corresponding to
  {A}rtin-{S}chreier-{W}itt coverings.
\newblock {\em Duke Mathematical Journal}, 77(3):607--625, 1995.

\bibitem[Por20]{PoratHilbert}
Gal Porat.
\newblock Lubin-{T}ate theory and overconvergent {H}ilbert modular forms of low
  weight.
\newblock {\em arXiv preprint arXiv:2010.14574}, 2020.

\bibitem[Por24]{porat2022overconvergence}
Gal Porat.
\newblock Overconvergence of \'etale {$(\varphi,\Gamma)$}-modules in families.
\newblock \url{https://arxiv.org/abs/2209.05050}, 2024.

\bibitem[Poy22]{Poyetonlocanaperiods}
L{\'e}o Poyeton.
\newblock Locally analytic vectors and rings of periods.
\newblock {\em arXiv preprint arXiv:2202.08075}, 2022.

\bibitem[PZ21]{PalZab21}
Aprameyo Pal and Gergely Z\'{a}br\'{a}di.
\newblock Cohomology and overconvergence for representations of powers of
  {G}alois groups.
\newblock {\em J. Inst. Math. Jussieu}, 20(2):361--421, 2021.

\bibitem[Qin06]{qing2006algebraic}
Liu Qing.
\newblock Algebraic geometry and arithmetic curves, 2006.

\bibitem[Sch02]{Schnei}
Peter Schneider.
\newblock {\em Nonarchimedean functional analysis}.
\newblock Springer Monographs in Mathematics. Springer-Verlag, Berlin, 2002.

\bibitem[Sch11]{schneider2011p}
Peter Schneider.
\newblock {\em $p$-{A}dic {L}ie groups}, volume 344.
\newblock Springer Science \& Business Media, 2011.

\bibitem[Sch12]{scholzeperf}
Peter Scholze.
\newblock Perfectoid spaces.
\newblock {\em Publications math{\'e}matiques de l'IH{\'E}S}, 116(1):245--313,
  2012.

\bibitem[Sen80]{sen1980continuous}
Shankar Sen.
\newblock Continuous cohomology and $p$-adic {G}alois representations.
\newblock {\em Inventiones mathematicae}, 62(1):89--116, 1980.

\bibitem[{Sta}24]{stacks-project}
The {Stacks Project Authors}.
\newblock \textit{Stacks Project}.
\newblock \url{https://stacks.math.columbia.edu}, 2024.

\bibitem[SW20]{scholze2020berkeley}
Peter Scholze and Jared Weinstein.
\newblock Berkeley lectures on $p$-adic geometry.
\newblock {\em Annals of Mathematics Studies}, 207, 2020.

\bibitem[Tat67]{Tat67}
John~T Tate.
\newblock $p$-divisible groups.
\newblock In {\em Proceedings of a Conference on Local Fields}, pages 158--183.
  Springer, 1967.

\bibitem[Van24]{PieThesis}
Pietro Vanni.
\newblock {\em Topics in multivariable $p$-adic Hodge theory and in derived
  analytic geometry over a nonarchimedean field}.
\newblock PhD thesis, Università degli Studi di Padova, 2024.

\bibitem[Wei17]{weinstein2017gal}
Jared Weinstein.
\newblock {$\mathrm{Gal}(\overline{\mathbf{Q}}_p/\mathbf{Q}_p)$} as a geometric
  fundamental group.
\newblock {\em International Mathematics Research Notices},
  2017(10):2964--2997, 2017.

\bibitem[Win83]{Win83}
Jean-Pierre Wintenberger.
\newblock Le corps des normes de certaines extensions infinies de corps locaux;
  applications.
\newblock {\em Annales scientifiques de l'{\'E}cole {N}ormale
  {S}up{\'e}rieure}, 16(1):59--89, 1983.

\bibitem[Z{\'a}b18]{ZabMulti0}
Gergely Z{\'a}br{\'a}di.
\newblock Multivariable {$(\varphi,\Gamma)$}-modules and products of {G}alois
  groups.
\newblock {\em Math. Research Letters, 25(2):687–721,}, 25(2):687–721,
  2018.

\end{thebibliography}

\end{document}